\documentclass[oneside,notitlepage,10pt]{article}

\usepackage{amsthm}
\usepackage{thmtools}
\usepackage{graphicx}
\usepackage[margin=2cm,font=normalsize,labelfont=bf]{caption}
\usepackage{verbatim}
\usepackage[shortlabels]{enumitem}
\usepackage{fancyhdr}
\usepackage[]{geometry}
\usepackage{xstring}
\usepackage{needspace}
\usepackage{color}
\usepackage{amstext}
\usepackage{nameref}
\usepackage[hypertexnames=false]{hyperref}
\usepackage{mathdots}
\usepackage[normalem]{ulem}
\usepackage{chngcntr}
\usepackage[section]{placeins}
\usepackage{tikz-cd}    
  
\makeatletter
\newcounter{denseversion}
\newcounter{comments}
\newcounter{authorcounter}
\newcounter{adresscounter}

\def\title#1{\gdef\@title{#1}}
\def\@title{}
\def\subtitle#1{\gdef\@subtitle{#1}}
\def\@subtitle{}

\def\authortagsused{0}
\def\adresstag#1{\if!#1!\else$^{\;#1\;}$\fi}
\def\@authorsep#1{
  \ifnum\value{authorcounter}=#1 and \else\unskip, \fi
}

\renewcommand{\author}[2][]{
  \stepcounter{authorcounter}
  \if!#1!\else\gdef\authortagsused{1}\fi
  \ifnum\value{authorcounter}=1
    \def\@authorstringa{#2\adresstag{#1}}
    \def\@authorstringb{#2}
    \def\@authorstringc{#2\adresstag{#1}}
  \else
    \ifnum\value{authorcounter}=2
      \g@addto@macro\@authorstringa{\@authorsep{2}#2\adresstag{#1}}
      \g@addto@macro\@authorstringb{\@authorsep{2}#2}
      \g@addto@macro\@authorstringc{\\#2\adresstag{#1}}
    \else
      \g@addto@macro\@authorstringa{\@authorsep{3}#2\adresstag{#1}}
      \g@addto@macro\@authorstringb{\@authorsep{3}#2}
      \g@addto@macro\@authorstringc{\\#2\adresstag{#1}}
    \fi
  \fi}
\def\@author{\ifnum\value{denseversion}=0\@authorstringa\else\@authorstringb\fi}

\def\@adressstringa{}
\def\@adressstringb{}
\newcommand{\adress}[2][]{
  \stepcounter{adresscounter}
  \ifnum\value{adresscounter}=1
    \g@addto@macro\@adressstringa{\ifnum\authortagsused=0\def\br{\\}\else\def\br{, }\fi\adresstag{#1}#2}
    \g@addto@macro\@adressstringb{\def\br{\\}\adresstag{#1}\parbox[t]{14cm}{#2}}
  \else
    \g@addto@macro\@adressstringa{\\[\bigskipamount]\adresstag{#1}#2}
    \g@addto@macro\@adressstringb{\\[\medskipamount]\adresstag{#1}\parbox[t]{14cm}{#2}}
  \fi}

\def\preprint#1{\gdef\@preprint{#1}}
\def\@preprint{}
\def\keywords#1{\gdef\@keywords{#1}}
\def\@keywords{}
\def\msc#1{\gdef\@msc{#1}}
\def\@msc{}
\def\email#1{
   \gdef\@email{#1}
   \g@addto@macro\@authorstringc{ {\it (#1)}}}
\def\@email{}
\def\dedication#1{\gdef\@dedication{#1}}
\def\@dedication{}

\def\mybaselinestretch#1{
  \gdef\@mybaselinestretch{#1}
  \renewcommand{\baselinestretch}{\@mybaselinestretch}}

\def\myparskip#1{
  \gdef\@myparskip{#1}
  \setlength{\parskip}{\@myparskip}}

\mybaselinestretch{1.2}
\myparskip{1.5ex}

\binoppenalty=\maxdimen
\relpenalty=\maxdimen

\normalfont
\normalsize
\newlength{\@listleftmargin}
\def\setenumstandard{
  \setlist{leftmargin=\@listleftmargin,itemsep=0pt,topsep=0pt,partopsep=0pt,parsep=\@myparskip}
  \setlist[enumerate]{align=left,labelsep=*,leftmargin=\@listleftmargin,itemsep=0pt,topsep=0pt,partopsep=0pt,parsep=\@myparskip}
}

\def\denseversion{
  \setcounter{denseversion}{1}
  \newgeometry{left=3cm,right=3cm,top=3cm}
  \mybaselinestretch{1.1}
  \myparskip{0.8ex}
  \normalfont
  \def\possiblelinebreak{}
  \fancyfoot[C]{\itshape{--$\,\,$\thepage$\,\,$--}}}

\def\possiblelinebreak{\\}

\renewcommand{\emph}[1]{\def\reserved@a{it}\ifx\f@shape\reserved@a\ul{#1}\else\textit{#1}\fi}

\def\setcrefnames{}
\newcommand{\mytableofcontents}{
   \ifnum\value{denseversion}=0
     \tableofcontents
     \setcrefnames 
   \else
     \renewcommand{\baselinestretch}{1.1}
     \setlength{\parskip}{0ex}
     \normalfont
     \begingroup
     \def\addvspace##1{\vskip0.4em}
     \tableofcontents
     \setcrefnames 
     \endgroup
     \renewcommand{\baselinestretch}{\@mybaselinestretch}
     \setlength{\parskip}{\@myparskip}
     \normalfont
   \fi}

\newlength{\zeilenlaenge}
\def\putindent#1{
  \settowidth{\zeilenlaenge}{#1}
  \ifnum\zeilenlaenge>\textwidth
    #1
  \else
    \noindent #1
  \fi
}

\hypersetup{
    linktocpage = true,
    colorlinks,
    linkcolor=[RGB]{0,0,120},
    citecolor=[RGB]{0,0,120},
    urlcolor=[RGB]{0,0,120}
}

\def\pdfdaten{
  \hypersetup{
    pdftitle = {\@title},
    pdfauthor = {\@author},
    pdfkeywords = {\@keywords},    
    bookmarksopen = true,
    bookmarksopenlevel = 1
  }}  

\def\showkeywords{\begin{flushleft}\footnotesize\textbf{Keywords}: \@keywords\end{flushleft}}
\def\showmsc{\begin{flushleft}\footnotesize\textbf{MSC 2010}: \@msc\end{flushleft}}

\def\tocsection#1{\section*{#1}\addcontentsline{toc}{section}{#1}}

\def\mytitle{}
\def\zmptitle{
  \begin{tabular}{cc}
    \begin{minipage}[c]{0.4\textwidth}
      \begin{flushleft}
        \includegraphics[width=110pt]{../../tex/zmp}
      \end{flushleft}  
    \end{minipage}&
    \begin{minipage}[c]{0.55\textwidth}
      \begin{flushright}
      {\small\sf\@preprint}
      \end{flushright}
    \end{minipage}
  \end{tabular}
  \vskip 2cm}

\def\maketitle{
  \pdfdaten
  \noindent
  \mytitle
  \begin{center}
    \LARGE\@title\\
    \if!\@subtitle!\else\smallskip\LARGE\@subtitle\\\fi
    \bigskip
    \if!\@author!\else\bigskip\large\@author\\\fi
    \ifnum\value{denseversion}=0
      \if!\@adressstringa!\else\bigskip\normalsize\@adressstringa\\\fi
      \if!\@email!\else\ifnum\value{authorcounter}=1\bigskip\normalsize\textit{\@email}\\\else\fi\fi
    \else
    \fi
    \if!\@dedication!\else\bigskip\normalsize{\@dedication}\\\fi
  \end{center}
  \ifnum\value{denseversion}=0\vskip 1.5cm\else\vskip0.5cm\fi}

\def\kobib#1{
  \begin{raggedright}
  \ifnum\value{denseversion}=0\else\small\fi
  \Oldbibliography{#1/kobib}
  \bibliographystyle{#1/kobib}
  \end{raggedright}
  \ifnum\value{denseversion}=0\else
      \noindent
      \if!\@authorstringc!\else
        \ifnum\authortagsused=0\ifnum\value{authorcounter}>1\normalsize\@authorstringc\\[\medskipamount]\else\fi\else\normalsize\@authorstringc\\[\medskipamount]\fi
      \fi
      \if!\@adressstringb!\else\normalsize\@adressstringb\\{}\fi
      \ifnum\authortagsused=0
        \ifnum\value{authorcounter}=1
          \if!\@email!\else\linebreak\normalsize\textit{\@email}\\{}\fi
        \else
        \fi
      \else
      \fi
  \fi}

\let\Oldbibliography\bibliography
\def\bibliography#1{
  \begin{raggedright}
  \ifnum\value{denseversion}=0\else\small\fi
  \Oldbibliography{#1}
  \end{raggedright}
  \ifnum\value{denseversion}=0\else
      \medskip
      \noindent
      \if!\@authorstringc!\else
        \ifnum\authortagsused=0\ifnum\value{authorcounter}>1\normalsize\@authorstringc\\[\medskipamount]\else\fi\else\normalsize\@authorstringc\\[\medskipamount]\fi
      \fi
      \if!\@adressstringb!\else\normalsize\@adressstringb\\{}\fi
      \ifnum\authortagsused=0
        \ifnum\value{authorcounter}=1
          \if!\@email!\else\linebreak\normalsize\textit{\@email}\\{}\fi
        \else
        \fi
      \else
      \fi
  \fi
}

\newenvironment{commentfigure}{}
\newenvironment{sidewayscommentfigure}{\begin{minipage}}{\end{minipage}}
\newenvironment{displaycomment}{\begin{list}{}{\rightmargin=1cm\leftmargin=1cm}\item\sf\begin{small}}{\end{small}\end{list}}

\def\tocmark#1{}

\def\draftstamp#1{
  \def\tocmark##1{
    \ifnum\c@secnumdepth=0\section{##1}\fi
    \ifnum\c@secnumdepth=1\subsection{##1}\fi
    \ifnum\c@secnumdepth=2\subsubsection{##1}\fi
    \ifnum\c@secnumdepth=3\subsubsection{##1}\fi
  }
  \ifnum\value{comments}=0
    \gdef\@draft{DRAFT - Edited on \today\ by #1 - Comments are not displayed}
  \else
    \gdef\@draft{DRAFT - Edited on \today\ by #1 - Comments are displayed}
  \fi
  \fancyhead[C]{\footnotesize\tt\textcolor{red}{\@draft}}}

\def\skript{
  \renewenvironment{displaycomment}{}{}
  \ifnum\value{comments}=0
    \renewenvironment{example*}{\comment}{\endcomment}
    \renewenvironment{remark*}{\comment}{\endcomment}
  \else\fi
  \parindent=0mm        
}

\mybaselinestretch{}
\setenumstandard

\newcommand\remember[2]{
  \label{#1}
  \immediate\write\@auxout{\unexpanded{\global\long\@namedef{mytext@#1}{#2}}}%
  #2%
}

\newcommand\recall[1]{%
  \ifcsname mytext@#1\endcsname
    \@nameuse{mytext@#1}%
  \else
    ``??''
  \fi
}

\makeatother

\input{kobib.tex}

\usepackage{amssymb}
\usepackage{amsmath}
\usepackage{amstext}
\usepackage{mathrsfs}
\usepackage{euscript}
\usepackage[all,cmtip]{xy}
\usepackage{xstring}
\usepackage{slashed}
\usepackage{tensor}
\usepackage{mathtools}

\def\ul{\underline}
\entrymodifiers={+!!<0pt,\fontdimen22\textfont2>}

\def\Z {\mathbb{Z}}

\def\R {\mathbb{R}}

\def\subset{\subseteq}

\def\inv{\mathrm{inv}}
\def\cstar{C$^{\ast}$}

\makeatletter
\renewenvironment{proof}[1][\nameProof]
  {\par\pushQED{\qed}%
   \normalfont \topsep6\p@\@plus6\p@\relax
   \trivlist
   \item[\hskip\labelsep
         \itshape
         #1\@addpunct{.}]
  \leavevmode}
  {\popQED\endtrivlist\@endpefalse}
\makeatother

\def\notebox#1#2{\begin{minipage}[b]{#1}\sloppy\renewcommand{\baselinestretch}{0.8}\footnotesize \begin{center}#2\end{center}\end{minipage}}

\def\with{\;\vert\;}

\def\mquad{\hspace{-2em}}

\newcommand{\arr}[1][r]{\ar@<0.7ex>[#1]\ar@<-0.7ex>[#1]}
\newcommand{\arrr}[1][r]{\ar@<1.4ex>[#1]\ar[#1]\ar@<-1.4ex>[#1]}
\newcommand{\arrrr}[1][r]{\ar@<2.1ex>[#1]\ar@<-2.1ex>[#1]\ar@<0.7ex>[#1]\ar@<-0.7ex>[#1]}

\newlength{\myeqt} 
\newlength{\myeqs} 
\newlength{\myeqm} 
\newlength{\myeqn} 
\newcommand\symtext[3][\myeqn]{
  \settowidth{\myeqt}{#2}
  \settowidth{\myeqs}{$#3$}
  \addtolength{\myeqs}{\the\myeqm}
  \ifdim\myeqt>\myeqs
    \stackrel{\hspace{-#1}\notebox{#1}{\medskip #2 \\ $\downarrow$\smallskip}\hspace{-#1}}{#3}
  \else
    \stackrel{\text{#2}}{#3}
  \fi}

\def\brackets#1{\IfStrEq{#1}{-}{}{(#1)}}
\def\subindex#1{\IfStrEq{#1}{-}{}{_{#1}}}

\newcommand{\alxydim}[2]{\begin{aligned}\xymatrix#1{#2}\end{aligned}}



\newlength{\myl}
\newcommand\sheaf[1]{\unitlength 0.1mm
  \settowidth{\myl}{$#1$}
  \addtolength{\myl}{-0.8mm}
  \begin{picture}(0,0)(0,0)
  \put(2,0){\text{\underline{\hspace{\myl}}}}
  \end{picture}#1\hspace{-0.15mm}}

\def\ddt#1#2#3{\left.\frac{\mathrm{d}^{\IfStrEq{#1}{1}{}{#1}}}{\mathrm{d}#2}\IfStrEq{#2}{#3}{\right.}{\right|_{#3}}}

\def\man{\Mfd}

\def\diff{\mathscr{D}\mathrm{iff}}

\def\top{\mathscr{T}\mathrm{op}}
\def\deltatop{\top^{\Delta}}

\def\ev{\mathrm{ev}}

\def\lw#1#2{{}^{#1\!}#2}
\def\B{\mathcal{B}\hspace{-0.03em}}

\newlength{\widthtmp}
\def\length#1{\settowidth{\widthtmp}{#1}\the\widthtmp}

\def\ttimes#1#2{\hspace{-0.15em}\tensor[_{#1}]{\times}{_{#2}}}

\definecolor{olivegreen}{rgb}{.33,.55,.18}

\newcommand{\ie}{i.e., }

\newcommand{\mc}[1]{\mathcal{#1}}

\newcommand{\smooth}{^{\mathrm{rig}}}

\newcommand{\SO}{\operatorname{SO}}
\newcommand{\Spin}{\operatorname{Spin}}
\newcommand{\String}{\stringgerbe}

\newcommand{\U}{\operatorname{U}}

\newcommand{\opp}{\operatorname{op}}
\newcommand{\pr}{\operatorname{pr}}
\newcommand{\fuse}{\circledast}
\newcommand{\lact}{\triangleright}
\newcommand{\ract}{\triangleleft}

\newcommand{\id}{\operatorname{id}}

\newcommand{\dis}{dis}

\newcommand{\Mod}{\mathcal{M}od}

\newcommand{\Homcat}{\mathscr{H}\mathrm{om}}

\newcommand{\Aut}{\mathrm{Aut}}

\newcommand{\AUT}{\mathscr{A}\mathrm{ut}}

\newcommand{\Bimod}{\mathscr{B}\mathrm{imod}}

\newcommand{\Bdl}[1]{#1\text{-}\mathscr{B}\mathrm{dl}}

\newcommand{\BimodBdl}{\mathscr{B}\mathrm{im}\mathscr{B}\mathrm{dl}}

\newcommand{\Grb}{\mathscr{G}\mathrm{rb}}

\newcommand{\Alg}{\Incl\mathrm{lg}}

\newcommand{\AlgBdl}{\Alg\mathscr{B}\mathrm{dl}}

\def\vNAlg{\mathrm{v}\mathscr{N}\Alg}
\def\vNAlgBdl{\mathrm{v}\mathscr{N}\AlgBdl}

\def\vNAtwoVectBdl{2\mathscr{H}\mathrm{ilb}\mathscr{B}\mathrm{dl}}

\newcommand{\Mfd}{\mathscr{M}\mathrm{fd}}

\newcommand{\Incl}{\mathscr{A}}

\def\stringgerbe{\mathcal{S}\hspace{-0.05em}t\hspace{-0.02em}r\hspace{-0.05em}i\hspace{-0.07em}n\hspace{-0.07em}g}

\usepackage[latin1]{inputenc}
\usepackage[english]{babel}
\usepackage[english]{cleveref}

\def\refname{References}

\def\quot#1{``#1''}

\def\quand{\quad\text{ and }\quad}
\def\quomma{\quad\text{, }\quad}

\def\nameTheorem{Theorem}
\def\nameTheorems{Theorems}
\def\nameDefinition{Definition}
\def\nameDefinitions{Definitions}
\def\nameCorollary{Corollary}
\def\nameCorollaries{Corollaries}
\def\nameProposition{Proposition}
\def\namePropositions{Propositions}
\def\nameConstruction{Construction}
\def\nameConstructions{Constructions}
\def\nameLemma{Lemma}
\def\nameLemmas{Lemmas}
\def\nameRemark{Remark}
\def\nameRemarks{Remarks}
\def\nameProblem{Problem}
\def\nameProblems{Problems}
\def\nameExample{Example}
\def\nameExamples{Examples}
\def\nameExercise{Exercise}
\def\nameExercises{Exercises}

\def\nameProof{Proof}
\def\nameEquation{\unskip}
\def\nameEquations{\unskip}
\def\nameSection{Section}
\def\nameSections{Sections}
\def\nameAppendix{Appendix}
\def\nameAppendixs{Appendices}
\def\nameFigure{Figure}
\def\nameFigures{Figures}

\input{theorems}

\usepackage{wasysym}

\def\mathscr#1{\EuScript{#1}}

\usepackage{todonotes}

\title{The stringor bundle}

\author[a]{Peter Kristel}
\email{peter.kristel@umanitoba.ca}

\adress[a]{
        University of Manitoba\\
        Department of Mathematics\\
Winnipeg, MB R3T 2N2, Canada}

\author[b]{Matthias Ludewig}
\email{matthias.ludewig@mathematik.uni-regensburg.de}

\adress[b]{
        Universit\"at Regensburg\\
        Fakult\"at f\"ur Mathematik\\
93053 Regensburg, Germany}

\author[c]{Konrad Waldorf}
\email{konrad.waldorf@uni-greifswald.de}

\adress[c]{
        Universit\"at Greifswald\\
        Institut f\"ur Mathematik und Informatik\\
D-17487 Greifswald}

\msc{}
\dedication{}

\denseversion
\allowdisplaybreaks
\usepackage{amstext}

\newcommand{\Rep}{\mathcal{R}}

\begin{document}

\maketitle

\begin{abstract}
We set up a framework of 2-Hilbert bundles, which allows to rigorously define  the \quot{stringor bundle}, a higher differential geometric object anticipated by Stolz and Teichner in an unpublished preprint about 20 years ago. Our framework includes an associated bundle construction, allowing us to associate  a 2-Hilbert bundle with a principal 2-bundle and a unitary representation of its structure 2-group. We prove that the Stolz-Teichner stringor bundle is canonically isomorphic to the 2-Hilbert bundle obtained from applying our associated bundle construction to a string structure on a manifold and the stringor representation of the string 2-group that we discovered in earlier work. This establishes a perfect analogy to spin manifolds,  representations of the spin groups, and spinor bundles.  
\end{abstract}

\mytableofcontents

\setsecnumdepth{0}

\section{Introduction}

The \emph{spin group} $\Spin(d)$ is the simply-connected cover of the special orthogonal group $\SO(d)$ (when $d\geq 3$). 
 The frame bundle $\SO(M)$ of an oriented Riemannian manifold $M$ of dimension $d\geq 3$ -- a principal $\SO(d)$-bundle -- may admit a lift to $\Spin(d)$.
Given such a lift -- called a \emph{spin structure} -- one can form the associated vector bundle using the \emph{spinor representation} of $\Spin(d)$.
The construction of this \emph{spinor bundle} is the starting point of spin geometry.

Motivated by the success of spin geometry in geometry and particle physics, \emph{string geometry} seeks for analogous structures meeting the demands of string theory. Most successful has been the principle established by  Killingback and Witten to look at spin structures on the configuration space of strings in $M$, the free loop space $LM=C^{\infty}(S^1,M)$ \cite{killingback1,witten2}. 
Such spin structures on $LM$ --- also called \emph{loop-spin structures} on $M$ and denoted below by  $\widetilde{L\Spin}(M)$ --- are different from the spin structures mentioned above, because the structure group of $LM$ has different properties compared to the finite-dimensional situation. Nonetheless, Stolz and Teichner outlined a construction of an infinite-dimensional spinor bundle on $LM$ \cite{stolz3}. Moreover, they established the principle of \emph{fusion in loop space}, expressing the idea that the relevant geometric structures on loop space correspond to (yet unknown) geometric structures on $M$ itself. In obvious analogy, they coined the terminology \emph{stringor bundle} for this unknown structure on $M$. Work of Brylinski \cite{brylinski1} and Murray \cite{Murray1996} on gerbes suggested that the stringor bundle is not an ordinary vector bundle, but must be of  a higher-categorical nature.

Another line of attack in string geometry is to search for an analogue of the spin group. Adding further connectedness to the orthogonal group, the \emph{string group}  $\String(d)$ is defined to be the 3-connected cover of $\Spin(d)$ \cite{ST04}. The string group cannot be realized as a finite-dimensional Lie group, and in recent years, the insight emerged that it is geometrically most fruitful to realize $\String(d)$ as a categorified group, or \emph{2-group} \cite{baez5},  a point of view that will be further advocated in this paper. Several models of the string 2-group in different contexts have been constructed, e.g., as a strict Fr\'echet Lie 2-group \cite{baez9}, as a finite-dimensional smooth \quot{stacky} 2-group \cite{pries2}, or as a strict diffeological 2-group \cite{Waldorf}. A major success of these models is to allow a neat definition of a \emph{string structure} on a manifold, as a reduction of the frame bundle to a $\String(d)$-bundle gerbe, denoted below by $\stringgerbe(M)$. String structures in this sense are related to loop-spin structures on $M$; in fact, they are equivalent to an enhanced version called \emph{fusive loop-spin structures} \cite{Waldorfa,Nikolausa}. 
This relation connects the two approaches to string geometry on the level of their basic underlying structures. In the present paper, we provide a yet deeper connection between these two approaches.

We invoke two recent developments that advanced each approach. 
The first concerns the stringor bundle of Stolz and Teichner, and its higher-categorical nature. 
In a sequence of papers \cite{kristel2019b, Kristel2019, kristel2020smooth, Kristel2019c} we obtained rigorous constructions of its main ingredients: the spinor bundle on loop space and, in particular, its fusion product that was anticipated long ago by Stolz and Teichner \cite{stolz3}. 
These constructions are based on  a given fusive loop-spin structure $\widetilde{L\Spin}(M)$ on $M$, and involve von Neumann algebra bundles and Connes fusion of bimodule bundles. 
In this paper, we reveal how this structure yields a higher-categorical vector bundle, more precisely, a \emph{2-Hilbert bundle}, which we call the \emph{Stolz-Teichner stringor bundle}, denoted $\mathscr{S}(\widetilde{L\Spin}(M))$. 
The theory of 2-vector bundles was developed in \cite{Kristel2022,Kristel2020} in a finite-dimensional context, based on the idea that a 2-vector space is nothing but an algebra, while  the morphisms are bimodules instead of algebra homomorphisms. 
It was then extended to the infinite-dimensional setting of 2-Hilbert bundles in \cite{ludewig2023spinor}, where algebras are replaced by von Neumann algebras.


The second advance is the \emph{stringor representation} constructed in \cite{Kristel2023}: a continuous, unitary representation of the string 2-group on a 2-Hilbert space,
\begin{equation} \label{RepresentationIntro}
 \Rep : \String(d) \to \mathcal{U}(A).
\end{equation}
Here, $A$ is the hyperfinite type III$_1$-factor, realized as a certain von Neumann algebra completion of an infinite-dimensional Clifford algebra, and $\mathcal{U}(A)$ is the unitary automorphism 2-group of $A$ (see \cref{Automorphism2Group}).
In this paper, we introduce an associated bundle construction (\cref{def:associated2vectorbundle}) which produces a 2-Hilbert bundle 
$\mathcal{Q} \times_{\mathcal{G}} A$ from a non-abelian  bundle gerbe $\mathcal{Q}$ for a topological strict 2-group $\mathcal{G}$ and  a continuous unitary representation $\Rep: \mathcal{G} \to \mathcal{U}(A)$ on a von Neumann algebra $A$.
In particular, we may use $\mathcal{Q} = \String(M)$, a string structure on $M$, the string 2-group $\mathcal{G} = \String(d)$, and $\Rep$ the stringor representation \eqref{RepresentationIntro}.
The following is the main result of this article,  stated as \cref{th:main} in the main text.

\begin{theorem}
\label{MainTheorem}
Let $M$ be a manifold with a fusive loop-spin structure $\widetilde{L\Spin}(M)$ and corresponding string structure $\stringgerbe(M)$. 
There exists a canonical isomorphism 
\begin{equation*}
    \stringgerbe(M) \times_{\String(d)} A\; \cong\; \mathscr{S}(\widetilde{L\Spin}(M))
\end{equation*}
between the 2-Hilbert bundle associated with $\stringgerbe(M)$ via the stringor    representation,  and the stringor bundle of Stolz and Teichner.
\end{theorem}

Our work joins the main forces of the above mentioned two approaches and provides a step towards a full picture of string geometry.  
For one, it shows the relevance of the stringor representation \cref{RepresentationIntro} for string geometry. At the same time, it provides justification for Stolz-Teichner's  description of the stringor bundle, by showing its equivalence to a structure obtained in a completely different but probably conceptually clearer way.   
Last but not least, we have by now established a perfect analogy to the construction of the spinor bundle as an associated vector bundle in spin geometry, in which the stringor representation \cref{RepresentationIntro} plays the role of the spinor representation and thus deserves its name.
The new perspective on the stringor bundle as an associated 2-Hilbert bundle may be helpful in the future for studying its spaces of sections, \quot{stringors}, and for studying differential operators on such spaces.

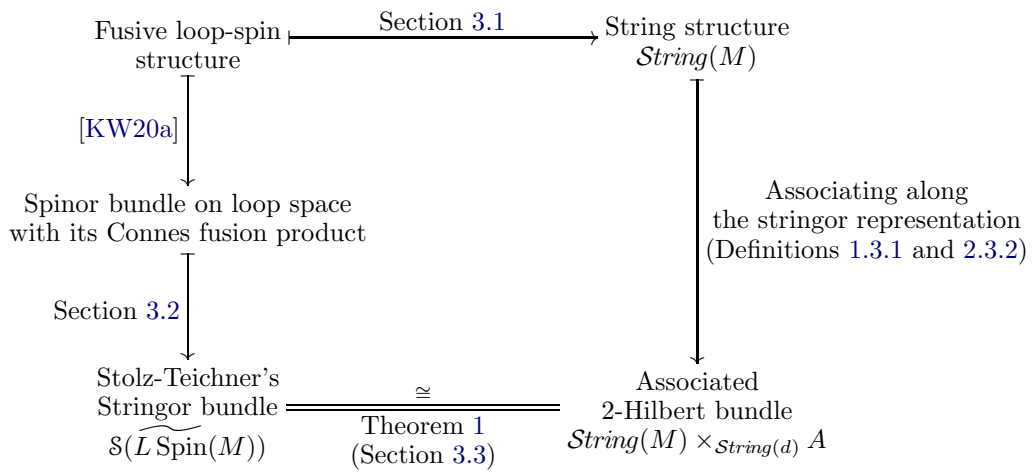
\begin{figure}[h]
\begin{equation*}
\hspace{6em}\xymatrix@R=4em@C=7em{\txt{Fusive loop-spin\\structure } \ar@{|->}[r]^{\txt{\cref{SubsectionStringStructuresFusiveLoopSpin}}} \ar@{|->}[d]_-{\txt{\cite{Kristel2019c}}} & \txt{String structure\\$\stringgerbe(M)$} \ar@{|->}[dd]^{\txt{Associating along\\the\ stringor representation\\(\cref{DefinitionStringorRep,def:associated2vectorbundle})}}   \\  \txt{Spinor bundle on loop space\\with its Connes fusion product} \ar@{|->}[d]_{\txt{\cref{sec:stringor}}} \\ \txt{Stolz-Teichner's \\Stringor bundle\\$\mathscr{S}(\widetilde{L\Spin}(M))$} \ar@{=}[r]_-{\txt{\cref{MainTheorem}\\(\cref{sec:identification})}}^-{\cong}  & \txt{Associated\\2-Hilbert bundle\\ $\stringgerbe(M) \times_{\String(d)}A$} }
\end{equation*}
\caption{A schematic description of our constructions, and where to find them. The commutativity of the diagram is the statement of our main result \cref{MainTheorem}.}
\label{Figure}
\end{figure}

This article is organized as follows. In \cref{sec:introrep} we recall the required details about von Neumann algebras, Connes fusion, and the stringor representation from our paper \cite{Kristel2023}. \Cref{sec:vonNeumannbundles} is devoted to 2-Hilbert bundles, and contains a general construction of associated 2-Hilbert bundles. 
In \cref{sec:ass2vb} we describe the Stolz-Teichner stringor bundle as a 2-Hilbert bundle, and prove our main theorem. 
We include three appendices: in  \cref{sec:convenientgeometry}  we recall 2-group bundles and bundle gerbes. In  \cref{sec:strstr} we compare two different notions of string structures involved into our constructions: fusive spin structures on loop space and bundle gerbes for the string 2-group, and we give a new direct construction to pass from the first notion to the second. In  \cref{sec:riggedtocontinuous} we establish a general relation between rigged von Neumann algebra bundles and bimodules as used in \cite{Kristel2019c} and their continuous versions established in this article; which is used in order to transfer the partial results of \cite{Kristel2019c} about the stringor bundle into the present setting. \cref{Figure} provides a schematic overview.

\paragraph{Acknowledgements. } We would like to thank Severin Bunk, Andr\'e Henriques, and Peter Teichner for helpful discussions. 
PK gratefully acknowledges support from the Pacific Institute for the Mathematical Sciences in the form of a postdoctoral fellowship. 
ML gratefully acknowledges support from SFB 1085 ``Higher invariants''.

\setsecnumdepth{2}

\section{The stringor representation}

\label{sec:introrep}

The purpose of this section is to recall from \cite{Kristel2023} the definition of the automorphism 2-group of a von Neumann algebra, the string group,  and the stringor representation, which is a homomorphism between these two  2-groups.

\subsection{The  automorphism 2-group of a von Neumann algebra}
\label{SectionAUTA}

Let $A$ be a von Neumann algebra and let $\Aut(A)$ be the group of $*$-automorphisms of $A$.
We recall that every $*$-automorphism $\theta$ is automatically continuous with respect to the ultraweak topology, and hence is the dual map of some isometric automorphism of the predual $A_*$. 
The group $\Aut(A)$ is a topological group with \emph{Haagerup's u-topology}, which is the topology induced by identifying $\Aut(A)$ with a subgroup of the isometry group of the predual $A_*$, equipped with the strong topology.

If $H$ is an $A$-$B$-bimodule (\ie a Hilbert space together with commuting $*$-representations of $A$ and $B^{\mathrm{op}}$) and $\theta_1 \in \Aut(A)$, $\theta_2 \in \Aut(B)$, we say that a unitary $U \in \U(H)$ is \emph{intertwining along $\theta_1$ and $\theta_2$} (which is short for \emph{left intertwining along $\theta_1$} and \emph{right intertwining along $\theta_2$}), if
\begin{equation}
  U(a \lact \xi \ract b) = \theta_1(a) \lact U\xi \ract \theta_2(b), \qquad a, b \in A, ~~ \xi \in H.
\end{equation}
We denote by
\begin{equation}
\label{DefinitionNH}
  I(H) \subseteq \Aut(A) \times \U(H) \times \Aut(B)
\end{equation}
the group of triples $(\theta_1, U, \theta_2)$ such that $U$ is intertwining along $\theta_1$ and $\theta_2$.
The group $I(H)$ is a topological group with the subspace topology, where the automorphism groups carry the u-topology and $\U(H)$ carries the strong topology.
By definition, the maps 
\begin{equation}
\label{SourceTargetMaps}
t_H: I(H) \to \Aut(A), \qquad s_H: I(H) \to \Aut(B),
\end{equation}
given by projection onto the left, respectively right factor, are continuous.

Canonically associated to $A$ is a Hilbert space $L^2(A)$, called the \emph{non-commutative $L^2$-space} or \emph{standard bimodule} \cite{HaagerupStandardForm}.
It comes with various extra structures, of which the following are relevant for the purposes of this paper:
\begin{enumerate}[(i)]
\item $L^2(A)$ is a faithful $A$-$A$-bimodule, with the property that any bounded operator $x \in \B(L^2(A))$ that commutes with the left (right) action of $A$ is given by right (left) multiplication with an element of $A$.
\item $L^2(A)$ is equipped with an anti-unitary involution $J$, called \emph{modular conjugation}, which satisfies
\begin{equation}
\label{RightActionModularConjugation}
  \xi \ract a = J (a^* \lact J\xi), \qquad a \in A,\; \xi \in L^2(A).
\end{equation}

\item
The association $A \mapsto L^2(A)$ is functorial on the category of von Neumann algebras and $*$-isomorphisms.
\end{enumerate}

Since $L^2(A)$ is a faithful module, the projection $I(L^2(A)) \to \U(L^2(A))$ onto the middle factor is injective.
Its image is denoted by
\begin{equation*}
  N(A) = \big\{ U \in \U(L^2(A)) \mid \exists \theta_1, \theta_2 \in \Aut(A) \text{ such that $U$ is intertwining along $\theta_1$ and $\theta_2$ } \big\}.
\end{equation*}
It turns out that the map $I(L^2(A)) \to N(A)$ is in fact a homeomorphism when $N(A) \subset \U(L^2(A))$ carries the subspace topology; see Remark~B.9 of \cite{ludewig2023spinor} or Lemma A.18 in \cite{ConformalNetsI}.
Precomposing the maps \cref{SourceTargetMaps} with the inverse of this homeomorphism, we obtain maps
\begin{equation}
\label{SourceTargetMaps2}
  s_{\Aut(A)}, t_{\Aut(A)} : N(A) \to \Aut(A).
\end{equation}
Explicitly, if $U \in N(A)$ is intertwining along $\theta_1$ and $\theta_2$, then $t_{\Aut(A)}(U) = \theta_1$ and $s_{\Aut(A)} (U)= \theta_2$. This can be reformulated to say that
\begin{equation}
\label{explicit-t}
U(a \lact U^{*}\xi) = \theta_1(a) \lact \xi,   
\qquad\text{and}\qquad 
U( U^{*}\xi \ract a) =  \xi \ract\theta_2(a) 
\end{equation}
whenever $a\in A$, $\xi \in L^2(A)$.
Moreover, it follows from \cref{RightActionModularConjugation} that $JUJ$ is intertwining along $\theta_2$ and $\theta_1$. 
Therefore, we have the relation
\begin{equation}
\label{stConjugation}
  t_{\Aut(A)}(U) = s_{\Aut(A)}(JUJ).
\end{equation}
Finally, it follows from the functoriality (iii) that for any $\theta \in \Aut(A)$, there is a unitary $L^2(\theta) \in \U(L^2(A))$ that commutes with the modular conjugation and is both left and right intertwining along $\theta$.
This provides a section
\begin{equation}
\label{CanonicalImplementation}
  L^2 : \Aut(A) \to N(A), \qquad \theta \mapsto L^2(\theta),
\end{equation}
called \emph{canonical implementation}, which is continuous and has closed image  \cite[Prop.~3.5]{HaagerupStandardForm}.

\medskip

We recall that a \emph{topological strict 2-group} is a groupoid $\mathcal{G}$ whose set $\mathcal{G}_0$  of objects and whose set $\mathcal{G}_1$ of morphisms are topological groups, and whose source map $s: \mathcal{G}_1 \to \mathcal{G}_0$, target map $t: \mathcal{G}_1 \to \mathcal{G}_0$, composition $\mathcal{G}_1 \ttimes st \mathcal{G}_1 \to \mathcal{G}_1$, identity map $i:\mathcal{G}_0 \to \mathcal{G}_1$, and inversion (w.r.t.\ composition) $\inv:\mathcal{G}_1 \to \mathcal{G}_1$ are all continuous group homomorphisms. 
A \emph{continuous  homomorphism} between topological strict 2-groups is  a functor whose assignments on objects and morphisms are continuous group homomorphisms.

It is convenient to note that in every topological strict 2-group the composition  and the inversion are already determined by the maps $s$, $t$ and $i$. 
Explicitly, they are given by
\begin{equation}
\label{MultiplicationInTermsOfist}
  X \circ Y = X i(s(X))^{-1} Y = X i(t(Y))^{-1} Y,
\end{equation}
whenever $X, Y \in \mathcal{G}_1$ are composable (\ie  $s(X) = t(Y))$, and by
\begin{equation}
\label{InversionInTermsOfist}
\inv(X)=i(s(X))X^{-1}i(t(X))\text{.}
\end{equation}
One can, conversely, define composition and inversion by these formulae, provided that the subgroups $\mathrm{ker}(s)\subset \mathcal{G}_1$ and $\mathrm{ker}(t)\subset \mathcal{G}_1$ commute. 
We refer to \cite{baez5,baez9} for a comprehensive treatment of (topological) 2-groups, and to \cite[\S 6]{Kristel2023} for more details about the formulae \cref{MultiplicationInTermsOfist,InversionInTermsOfist}.

\begin{definition}
\label{Automorphism2Group}
The \emph{unitary automorphism 2-group} $\mathcal{U}(A)$ of $A$ is the  topological strict 2-group with 
\begin{equation*}
\mathcal{U}(A)_0 := \Aut(A) \quand \mathcal{U}(A)_1 := N(A)\text{,}
\end{equation*}
source and target maps $s_{\Aut(A)}$ and $t_{\Aut(A)}$ from \cref{SourceTargetMaps2}, and identity map $i := L^2$ from  \cref{CanonicalImplementation}.
\end{definition}

\begin{remark}
In order to see that this suffices to define a strict 2-group, we need to check that $\ker(s)$ and $\ker(t)$ are commuting subgroups of $N(A)$.
We observe that $\ker(s)$  consists of unitaries $U$ on $L^2(A)$ that commute with the right $A$-action.
By property (i) of $L^2(A)$, this means that each such $U$ is left multiplication by some element of $A$.
Similarly, an element $V \in \ker(t)$ is right multiplication by some element of $A$.
Since the left and right $A$-actions commute, this shows that $U$ and $V$ commute.
We hence can define composition and inversion of $\mathcal{U}(A)$ by \cref{MultiplicationInTermsOfist,InversionInTermsOfist};
for instance, we have
\begin{equation}
\label{CompositionMapAUTA}
U \circ  V = U L^2(\theta)^* V,
\end{equation}
if $U$ is right intertwining and $V$ is left intertwining along $\theta$.
\end{remark}

The data of a topological strict 2-group $\mathcal{G}$ are conveniently encoded in its associated \emph{crossed module}.
A topological crossed module is a pair of topological groups $G$ and $H$, together with a continuous group homomorphism $t: H \to G$ and a continuous action $\alpha: G \times H \to H$ of $G$ on $H$ satisfying 
\begin{equation} \label{CrossedModuleActions}
t(\alpha(g,h))=gt(h)g^{-1}
\quand
\alpha(t(h),x)=hxh^{-1}
\end{equation}
for all $g\in G$ and $h,x\in H$. 
The crossed module associated to a topological strict 2-group $\mathcal{G}$ is $t: \mathcal{G}_s \to \mathcal{G}_0$, where 
\begin{equation*}
\mathcal{G}_s := \ker(s) \subset \mathcal{G}_1
\end{equation*}
 and  $\mathcal{G}_0$ acts on $\mathcal{G}_s$ by conjugation with $i(g)$. This procedure establishes an equivalence of categories between topological strict 2-groups and topological crossed modules, see \cite[Thm.~2]{BrownSpencerCrossed}, \cite[Thm.~5.13]{Fiore} or \cite{porst2008strict}.

The crossed module associated to the unitary automorphism 2-group $\mathcal{U}(A)$ of a von Neumann algebra $A$, denoted by $\mathscr{U}(A)$, is $t_{\Aut(A)}:\U(A) \to \Aut(A)$, where $\U(A)\subset A$ denotes the group of unitary elements of $A$ equipped with the ultraweak topology, and $\Aut(A)$ acts on $\U(A)$ by evaluation; see  \cite[Prop. 6.6]{LudewigWaldorf2Group}.

\begin{definition}
        \label{representationoftopological2group}
        A \emph{unitary representation} of a topological strict  2-group $\mathcal{G}$ on a von Neumann algebra $A$ is a continuous homomorphism of topological strict 2-groups
        \begin{equation*}
                \Rep: \mathcal{G} \to \mathcal{U}(A)\text{.}
        \end{equation*}
\end{definition}

Explicitly, $\Rep$ consists of continuous group homomorphisms $\Rep_0 : \mathcal{G}_0 \to \Aut(A)$ and $\Rep_1 : \mathcal{G}_1 \to N(A)$ with the properties that
\begin{equation}
\label{CompatibilityStrictRepresentation}
  \Rep_0 \circ s_\mathcal{G} = s_{\Aut(A)} \circ \Rep_1, \qquad \Rep_0 \circ t_\mathcal{G} = t_{\Aut(A)} \circ \Rep_1, \qquad \Rep_1 \circ i_\mathcal{G} = L^2 \circ \Rep_0.
\end{equation}
By formulae \cref{MultiplicationInTermsOfist,InversionInTermsOfist}, the conditions in \cref{CompatibilityStrictRepresentation} imply automatically that $\Rep_0$ and $\Rep_1$  intertwine the composition and inversion maps.

\subsection{Twisted standard bimodules and Connes fusion}

In this section, we compare the (strict) unitary automorphism 2-group $\mathcal{U}(A)$ of \cref{Automorphism2Group} with the abstract  automorphism 2-group of the object $A$ in  the bicategory $\vNAlg^{\mathrm{bi}}$ of von Neumann algebras \cite{Landsman,Brouwer2003}. For finite-dimensional algebras, we carried out an analogous comparison in \cite[Prop. 2.3.1]{Kristel2022}.
Along the way, we provide some results on twisted standard bimodules that will be needed subsequently.

Given von Neumann algebras $A$ and $B$, we denote by $A\text{-}B\text{-}\Bimod$ the category of $A$-$B$-bimodules and unitary intertwiners. It is the category of 1-morphisms $B\to A$ in the bicategory of von Neumann algebras, i.e.,
\begin{equation*}
\Homcat_{\vNAlg^{\mathrm{bi}}}(B,A) :=A\text{-}B\text{-}\Bimod\text{.} 
\end{equation*}
Composition in the bicategory $\vNAlg^{\mathrm{bi}}$ is given by the \emph{Connes fusion product}, which is a functor
\begin{equation*}
\boxtimes : A\text{-}B\text{-}\Bimod \times B\text{-}C\text{-}\Bimod \to A\text{-}C\text{-}\Bimod,
\end{equation*}
and should be viewed as the appropriate ``tensor product over $B$'' for bimodules \cite{Landsman,Brouwer2003,thom11}. 
In particular, the Connes fusion product turns $A\text{-}A\text{-}\Bimod$ into a monoidal category.

There are several (more or less involved) explicit constructions of the Connes fusion product, but in this paper, we only need its abstract properties.
In particular, its functoriality means that two unitary intertwiners $U : H \to H^\prime$ and $V : K \to K^\prime$ have a fusion product $U \boxtimes V : H \boxtimes K \to H^\prime \boxtimes K^\prime$.
The fusion product $U \boxtimes V$ is, more generally, also defined if $H$ and $H^\prime$ are right $B^\prime$-modules, $K$ and $K^\prime$ are left $B^\prime$-modules, and $U$, $V$ intertwine the right (respectively left) module actions along some $*$-isomorphism $\varphi: B \to B^\prime$ (see \cite[Proposition~A.2.3]{Kristel2019c} or \cite[\S A.3]{ludewig2023spinor}).
In fact, this generalized Connes fusion product for intertwiners provides von Neumann algebras and their bimodules with the structure of a \emph{double category}; see \cite{shulman1}.


For $\theta \in \Aut(A)$, we denote by $L^2(A)_\theta$ the $A$-$A$-bimodule with underlying Hilbert space $L^2(A)$, the standard left action, but right action modified by $\theta$; we refer to $L^2(A)_\theta$ as a \emph{twisted standard bimodule}.
We consider the functor
\begin{equation}
\label{StrictToNonStrictVersion}
\mathcal{T} : \mathcal{U}(A) \to A\text{-}A\text{-}\Bimod
\end{equation}
that sends an automorphism $\theta$ to the twisted standard bimodule $L^2(A)_\theta$, while an element $U \in N(A)$ that is intertwining along $\theta_1$ and $\theta_2$ is sent to the intertwiner $\mathcal{T}(U) := L^2(\theta_1) U^* : L^2(A)_{\theta_2} \to L^2(A)_{\theta_1}$.
We emphasize that this is an ``honest'' intertwiner, in that it is intertwining along the identity automorphism on both sides.
%
%

If $\theta_1, \theta_2 \in \Aut(A)$, then there is a canonical isomorphism
\begin{equation}
\label{IsoOfTwistedStandardForms}
  \chi_{\theta_1,\theta_2} : L^2(A)_{\theta_1} \boxtimes L^2(A)_{\theta_2} \to L^2(A)_{\theta_1\circ\theta_2},
\end{equation}
see \cite[Example A.6]{ludewig2023spinor} for its definition in terms of a particular model for the Connes fusion product.
Axiomatically, the isomorphisms \cref{IsoOfTwistedStandardForms} can be characterized by the properties that (i) when one of $\theta_1$ or $\theta_2$ is the identity, they coincide with the usual unitor for the Connes fusion product, and (ii), when given unitaries
\begin{equation*}
  U_1 : L^2(A)_{\theta_1} \to L^2(A)_{\theta_1^\prime}, \qquad U_2 : L^2(A)_{\theta_2} \to L^2(A)_{\theta_2^\prime},
\end{equation*}
which are right intertwining (respectively left intertwining) along some automorphism $\varphi$, the isomorphisms \cref{IsoOfTwistedStandardForms} fit into the commutative diagram
\begin{equation}
\label{Naturality}
                \alxydim{@C=4em}{
                L^2(A)_{\theta_1} \boxtimes L^2(A)_{\theta_2} \ar[r]^-{\chi_{\theta_1,\theta_2}} \ar[d]_{U_1 \boxtimes U_2} & L^2(A)_{\theta_1\circ\theta_2} \ar[d]^{U_1 L^2(\theta_1)L^2(\varphi)^*U_2 L^2(\theta_1)^*} \\ L^2(A)_{\theta_1^\prime}\boxtimes L^2(A)_{\theta_2^\prime}  \ar[r]_-{\chi_{\theta_1^\prime,\theta_2^\prime}} & L^2(A)_{\theta_1^\prime \circ \theta_2^\prime\text{.}} 
                }
\end{equation}
%
%
Taking $\varphi$ to be the identity, this shows that the isomorphisms \cref{IsoOfTwistedStandardForms} are the components of a natural transformation  $\chi$.
Indeed, let $U_i \in N(A)$ be intertwining along $\theta_i^\prime$ and $\theta_i$. 
Then, $\mathcal{T}(U_i) = L^2(\theta_i^\prime) U_i^*$ is an intertwiner from $L^2(A)_{\theta_i}$ to $L^2(A)_{\theta_i^\prime}$ and $\mathcal{T}(U_1U_2) = L^2(\theta_1^\prime\theta_2^\prime)U_2^*U_1^*$ is an intertwiner from $L^2(\theta_1\theta_2)$ to $L^2(\theta_1^\prime\theta_2^\prime)$.
Since
\begin{equation*}
\begin{aligned}
\mathcal{T}(U_1) L^2(\theta_1)\mathcal{T}(U_2) L^2(\theta_1)^* = L^2(\theta_1^\prime) \underbrace{U_1^*L^2(\theta_1)}_{\in A}\underbrace{L^2(\theta_2^\prime) U_2^*}_{\in A^\prime}L^2(\theta_1)^* = L^2(\theta_1^\prime\theta_2^\prime)U_2^*U_1^* = \mathcal{T}(U_1U_2),
\end{aligned}
\end{equation*}
the diagram \cref{Naturality} becomes the claimed naturality diagram.
The  isomorphisms \cref{IsoOfTwistedStandardForms} satisfy, moreover, the obvious associativity condition for triples of automorphisms (involving the associator of the Connes fusion product), and hence turn the functor $\mathcal{T}$ into a  monoidal functor, in other words, a   homomorphism of 2-groups.

It is easy to check that $\mathcal{T}$ is fully faithful.
Hence, if we denote by $\AUT(A):= A\text{-}A\text{-}\Bimod$ the automorphism 2-group of the von Neumann algebra $A$ as an object in the bicategory $\vNAlg^{\mathrm{bi}}$ of von Neumann algebras, $\mathcal{T}$ embeds our strict automorphism 2-group $\mathcal{U}(A)$ as a sub-2-group of $\AUT(A)$. Moreover, going through the Murray-von-Neumann classification of factors, one obtains that $\mathcal{T}$ is essentially surjective if $A$ is a factor of type I or type III. Hence, in these cases, the strict 2-group $\mathcal{U}(A)$ is equivalent to the general automorphism 2-group $\AUT(A)$.

\subsection{The string 2-group and the stringor representation}
\label{SectionStringorRep}

For a smooth manifold $M$, we denote by $PM$ the space of smooth paths $\beta : [0, \pi] \to M$, which are \emph{flat} at the end points, i.e., all derivatives vanish (in some, hence all local charts).
For $x \in M$, we write $P_x M \subset PM$ for the subspace of paths $\beta$ that additionally satisfy $\beta(0) = x$.
Moreover, we denote by $LM$ the space of smooth loops $S^1 \to M$, where we set $S^1 = \R / 2 \pi \Z$. 
We denote by $PM^{[k]}$ the $k$-fold fibre product of the end-points-map $PM \to M \times M$, and consider the map
\begin{equation}
  \label{CupMap}
  \cup : PM^{[2]} \to LM, \qquad \; (\beta_1 \cup \beta_2)(t) := \begin{cases} \beta_1(t)  & t \in [0, \pi] \\ \beta_2(2 \pi - t) & t \in [\pi, 2\pi] \end{cases}
\end{equation}
that combines two paths $\beta_1,\beta_2$ with common endpoints to a loop, which is automatically smooth since the paths are flat. 
All path spaces discussed above have canonical structures of infinite-dimensional manifolds.
In particular, for a Lie group $G$, we have $P_e G$, the space of flat paths starting at the identity element $e$. 
Both $P_e G$ and $LG$ are infinite-dimensional (Fr\'echet) Lie groups.

Let $\widetilde{L\Spin}(d)$ be a basic central extension of $L\Spin(d)$. 
Up to isomorphism of central extensions, there are two possible choices, and each is unique up to unique isomorphism \cite{LudewigWaldorf2Group}.
Both of these choices give rise to canonically isomorphic string groups \cite{LudewigWaldorf2Group}.
It is a fact that $\widetilde{L\Spin}(d)$ (in fact, \emph{any} central extension of $L\Spin(d)$ \cite{LudewigWaldorf2Group}) admits a unique \emph{fusion factorization} \cite[\S 5.2]{Kristel2019}, i.e., a Lie group homomorphism
\begin{equation}
\label{FusionFactorization}
  i : P \Spin(d) \to \cup^*\widetilde{L\Spin}(d)
\end{equation}
covering the diagonal map $P \Spin(d) \to P\Spin(d)^{[2]}$.
Explicitly, the elements of $\cup^* \widetilde{L\Spin}(d)$ have the form $(\gamma_1, \gamma_2, X)$, with $\gamma_1, \gamma_2 \in P \Spin(d)$ such that $\gamma_1(0)=\gamma_2(0)$,  $\gamma_1(\pi) = \gamma_2(\pi)$ and $X \in \widetilde{L\Spin}(d)$ projecting to $\gamma_1 \cup \gamma_2$.

\begin{definition}
\label{DefinitionStringorRep}
  The \emph{string  2-group $\String(d)$} is the strict Lie 2-group with
  \begin{equation*}
    \String(d)_1 := \cup^* \widetilde{L\Spin}(d)|_{P_e\Spin(d)^{[2]}} \quand \String(d)_0 := P_e \Spin(d),
  \end{equation*}
   with source and target maps 
  \begin{equation*}
s_{\String(d)}(\gamma_1, \gamma_2, X) := \gamma_2\quand t_{\String(d)}(\gamma_1, \gamma_2, X) := \gamma_1,
  \end{equation*}
  and with identity map given by the restriction of the fusion factorization $i$ to $P_e \Spin(d)$.
  
\end{definition}

\begin{remark}
The structure in \cref{DefinitionStringorRep} determines a Lie 2-group via \cref{InversionInTermsOfist,MultiplicationInTermsOfist} because $\widetilde{L\Spin}(d)$ is disjoint commutative; see \cite{LudewigWaldorf2Group} for a detailed treatment. 
In particular, the composition, determined by \cref{MultiplicationInTermsOfist}, is given by
\begin{equation}
\label{composition-in-string-2-group}
(\gamma_1,\gamma_2,X_{12}) \circ (\gamma_2,\gamma_3,X_{23}) = (\gamma_1,\gamma_3,X_{12}i(\gamma_2)^{*}X_{23})\text{.}
\end{equation}
\end{remark}

\begin{remark}
\label{projection-to-spin}
The string 2-group is a covering group of the spin group, in the sense that there is a strict 2-group homomorphism
\begin{equation*}
q:\String(d) \to \Spin (d)_{\dis}\text{,} 
\end{equation*}
where $\Spin(d)_{\dis}$ denotes the standard way to view a group as a 2-group (set $(\Spin(d)_{\dis})_0$ = $(\Spin(d)_{\dis})_1 $ = $\Spin(d)$, and $s=t=i=\id$). The homomorphism $q$ is given by $q_0 := \ev_\pi$, the evaluation of paths at their endpoint.  Under geometric realization, $q$ becomes a 3-connected covering map \cite{baez9,LudewigWaldorf2Group}. 
\end{remark}

In \cite{Kristel2023}, we describe a representation of the string 2-group $\String(d)$ on a 2-Hilbert space, whose underlying von Neumann algebra is the hyperfinite type III$_1$ factor $A$.
We will not need the explicit construction of this representation, but we now recall the ingredients needed for the purposes of this paper.
The main players are group homomorphisms
\begin{align}
\label{KleinOmega}
  \omega : P \Spin(d) &\to \Aut(A),\\
  \label{GrossOmega}
  \Omega : \widetilde{L\Spin}(d) &\to N(A),
\end{align}
which are continuous with respect to the u-topology on $\Aut(A)$ and the strong topology on $N(A)$, respectively.
A concrete definition of $\omega$ is  in  \cite[Eq. 5.6]{Kristel2023}, and of $\Omega$  in \cite[Eq. 5.3, Lemma~5.1]{Kristel2023}. 

We will use the following two properties of the maps $\omega$ and $\Omega$. If $X \in \widetilde{L\Spin}(d)$ lies over $\gamma_1 \cup \gamma_2 \in L\Spin(d)$, then  \cite[Theorem 6.9]{Kristel2023} shows that $t_{\Aut(A)}(\Omega(X)) = \omega_{\gamma_1}$ and $s_{\Aut(A)}(\Omega(X)) = \omega_{\gamma_2}$, for the maps $s_{\Aut(A)}$ and $t_{\Aut(A)}$ from \cref{SourceTargetMaps2}.
 In other words, the unitary map  $\Omega(X) \in N(A)\subset \U(L^2(A))$ is left intertwining along $\omega_{\gamma_1}$ and right intertwining along $\omega_{\gamma_2}$. 
In formulas,
\begin{equation}
\label{IntertwiningOmegaomega}
\Omega(X)(a \lact \xi \ract b) = \omega_{\gamma_1}(a) \lact \Omega(X)\xi \ract \omega_{\gamma_2}(b)\text{.}
\end{equation}
Moreover, \cite[Thm. A.9]{Kristel2023} implies that
\begin{equation}
\label{OmegaAndI}
  \Omega(i(\gamma)) = L^2(\omega_\gamma)
\end{equation}
for all $\gamma \in P\Spin(d)$, where $i$ is the fusion factorization \cref{FusionFactorization} and $L^2(\omega_\gamma)$ is the canonical implementation \cref{CanonicalImplementation} of the automorphism $\omega_{\gamma}$.

\begin{definition}
\label{def-stringor-rep}
The \emph{stringor representation}
\begin{equation*}
  \Rep : \String(d) \to \mathcal{U}(A)
\end{equation*}
 consists of the group homomorphisms
\begin{equation*}
  \Rep_0 := \omega|_{P_e \Spin(d)}: P_e \Spin(d) \to \Aut(A)
\qquad \text{and} \qquad
  \Rep_1 := \Omega: \widetilde{L\Spin}(d) \to N(A)\text{.}
\end{equation*}
\end{definition}

\begin{remark}It follows directly from the properties of $\omega$ and $\Omega$ alluded to above that $\Rep_0$ and $\Rep_1$ satisfy the compatibility relations of \eqref{CompatibilityStrictRepresentation} which ensure that $\Rep$ is indeed a  homomorphism of strict 2-groups, namely
\begin{align*}
\label{SourceTargetCompatibility}
  (s_{\Aut(A)} \circ \Rep_1)(\gamma_1,\gamma_2,X) &= \omega_{\gamma_2}=\mathcal{R}_0(s(\gamma_1,\gamma_2,X))
  \\
  (t_{\Aut(A)} \circ \Rep_1)(\gamma_1,\gamma_2,X) &= \omega_{\gamma_1}=\mathcal{R}_0(t(\gamma_1,\gamma_2,X))
\end{align*}
as well as
\begin{equation}
\label{stringor-rep-comp-with-fusion}
\Rep_1 \circ i = L^2 \circ \Rep_0\text{.}
\end{equation}
\end{remark}

Additionally to the stringor representation, we will consider the group homomorphism
\begin{equation}
\label{representation-on-standard-bimodule}
\Omega': \widetilde{L\Spin}(d) \to \U(L^2(A)),\quad \Omega'(X) := J\Omega(X)J\text{,}
\end{equation}
which establishes a unitary representation of $\widetilde{L\Spin}(d)$ on the standard bimodule $L^2(A)$.
The conjugation by $J$ achieves an exchange of the left/right intertwining properties, so that we get
\begin{equation}
\label{IntertwiningOmegaprime}
\Omega'(X)(a \lact \xi \ract b)= \omega_{\gamma_2}(a) \lact \Omega'(X)( \xi )\ract \omega_{\gamma_1}(b)
\end{equation}
whenever $X$ projects to $\gamma_1 \cup\gamma_2$. This will be required to obtain a bimodule structure on the spinor bundle on loop space that is compatible with our conventions for 2-Hilbert  bundles; see \cref{sec:stringor}. 
We remark that relation \cref{OmegaAndI} persists to hold for $\Omega'$, as $J$ commutes with canonical implementation. 

%

\section{2-Hilbert bundles}
\label{sec:genproc}
\label{sec:vonNeumannbundles}

In \cref{von-Neumann-algebra-bundles}, we define a bicategory of von Neumann algebra bundles over a topological space $X$, whose 1-morphisms are bimodule bundles. 
Viewing the base space as a variable, these  form a presheaf of bicategories. 
In \cref{stackification}, we argue that it is necessary to stackify this presheaf to obtain a sheaf of bicategories, or 2-stack. The objects in this 2-stack are our \emph{2-Hilbert bundles}.
In \cref{SectionAssociatedBimoduleBundles}, we introduce the \emph{associated 2-Hilbert bundle construction}, which produces a 2-Hilbert bundle $\mathcal{Q} \times_\mathcal{G} A$ over a space $X$ from a non-abelian bundle gerbe $\mathcal{Q}$ over  $X$ for a topological strict  2-group $\mathcal{G}$ and a continuous unitary representation $\mathcal{G} \to \AUT(A)$ of  $\mathcal{G}$ on a von Neumann algebra $A$.

\subsection{Von Neumann algebra bundles}

\label{von-Neumann-algebra-bundles}

Let $X$ be a topological space.
In this section, we define the bicategory $\vNAlgBdl^{\mathrm{bi}}(X)$ of von Neumann algebra bundles over $X$, focussing on the properties necessary for the present paper.
A more extensive treatment has been moved to a separate paper \cite[\S A\&B]{ludewig2023spinor}.

The objects of $\vNAlgBdl^{\mathrm{bi}}(X)$ are locally trivial von Neumann algebra bundles over $X$.
Such a bundle $\mathcal{A}$ consists of a von Neumann algebra $\mathcal{A}_x$ for each point $x \in X$, together with a collection of local trivializations $\phi : \mathcal{A}|_O \to O \times A$ ($O \subseteq X$ open, $A$ the typical fibre, a von Neumann algebra) such that the transition functions $\phi^\prime \circ \phi^{-1}$ are continuous when considered as maps $O \cap O^\prime \to \Aut(A)$; here $\Aut(A)$ carries the u-topology, as always.
We refer to \S B.1 of \cite{ludewig2023spinor} for a more extensive treatment of this notion.
For the purposes of this paper, we only need the following feature: 
Whenever $G$ is a topological group with a continuous group homomorphism $G \to \Aut(A)$ and $P$ is a principal $G$-bundle over $X$, then the associated bundle construction provides a von Neumann algebra bundle
\begin{equation}
\label{associated-von-Neumann-algebra-bundle}
  \mathcal{A} = P \times_G A\text{,}
\end{equation}
see Example~B.6 \cite{ludewig2023spinor}.

If $\mathcal{A}$, $\mathcal{B}$ are von Neumann algebra bundles, we denote by $\mathcal{A}\text{-}\mathcal{B}\text{-}\BimodBdl(X)$ the category of $\mathcal{A}$-$\mathcal{B}$-bimodule bundles, which serves as the category of morphisms $\mathcal{B} \to \mathcal{A}$ in $\vNAlgBdl^{\mathrm{bi}}(X)$.
Here, an $\mathcal{A}$-$\mathcal{B}$-bimodule bundle $\mathcal{H}$ is a continuous Hilbert  bundle whose fibres $\mathcal{H}_x$ carry the structure of an $\mathcal{A}_x$-$\mathcal{B}_x$-bimodule, and which admits local trivializations 
\begin{equation*}
  u: \mathcal{H}|_O \to O \times H
\end{equation*}
over open sets $O \subseteq X$, such that $H$ is a bimodule for the typical fibres of $\mathcal{A}$ and $\mathcal{B}$ and $u$ is intertwining along local trivializations of $\mathcal{A}$ and $\mathcal{B}$ \cite[Definition~B.6]{ludewig2023spinor}.
Such trivializations are called \emph{local bimodule trivializations}.
Morphisms between $\mathcal{A}$-$\mathcal{B}$-bimodules are Hilbert bundle homomorphisms that are fibrewise intertwiners.

\begin{example}
\label{ExTwistedRightAction}
 If $A$, $B$ are von Neumann algebras and $H$ is an $A$-$B$-bimodule, we obtain the trivial von Neumann algebras bundles $\underline{A} = X \times A$, $\underline{B} = X \times B$ and the trivial $\underline{A}$-$\underline{B}$-bimodule bundle $\underline{H}$ over $X$.
If moreover $\theta : X \to \Aut(B)$ is a continuous map, we denote by $\underline{H}_\varphi$ the $\underline{A}$-$\underline{B}$-bimodule bundle with total space $X \times H$ and bimodule action given by
\begin{equation*}
  (x, a) \lact (x, \xi) \ract (x, b) = \bigl(x, a \lact \xi \ract \theta(b)\bigr).
\end{equation*} 
\end{example}

\begin{example}
\label{StandardBimoduleBundle}
If $\mathcal{A}$ is a von Neumann algebra bundle, then $L^2(\mathcal{A})$ is the $\mathcal{A}$-$\mathcal{A}$-bimodule bundle whose fibre over $x$ is $L^2(\mathcal{A}_x)$, the standard bimodule of $\mathcal{A}_x$, with local trivializations given by $L^2(\varphi)$, where $\varphi$ is a local trivialization of $\mathcal{A}$.
\end{example}

In order to define the composition of 1-morphisms, it is important to restrict to the subcategory of bimodules whose typical fibre $H$ is \emph{right implementing}, in the sense that the map $s_H$ defined in \cref{SourceTargetMaps} admits a unit-preserving section near the unit element.
This is in particular the case for $H = L^2(A)$, the standard bimodule, as follows from the existence of the canonical implementation.
Denoting by $\mathcal{A}\text{-}\mathcal{B}\text{-}\BimodBdl^{\mathrm{imp}}(X)$ the corresponding subcategory of \emph{right implementing $\mathcal{A}$-$\mathcal{B}$-bimodule bundles}, composition is a functor
\begin{equation}
\label{CompositionOfBimoduleBundles}
\mathcal{A}\text{-}\mathcal{B}\text{-}\BimodBdl^{\mathrm{imp}}(X) \times \mathcal{B}\text{-}\mathcal{C}\text{-}\BimodBdl^{\mathrm{imp}}(X) \to \mathcal{A}\text{-}\mathcal{C}\text{-}\BimodBdl^{\mathrm{imp}}(X)\text{,}
\end{equation} 
which is  given fibrewise by the Connes fusion product.
In more detail, let $\mathcal{H}$ be a right implementing $\mathcal{A}$-$\mathcal{B}$-bimodule bundle and $\mathcal{K}$ be a right implementing $\mathcal{B}$-$\mathcal{C}$-bimodule bundle and let $u$ and $v$ be local bimodule bundle trivializations of $\mathcal{H}$, respectively $\mathcal{K}$ over an open set $O \subseteq X$.
Then, if $u$ and $v$ intertwine along the same the local trivialization $\varphi$ of $B$, the map
\begin{equation*}
  u \boxtimes v : \mathcal{H} \boxtimes \mathcal{K}|_O \to O \times (H \boxtimes K)
\end{equation*} 
given by fibrewise Connes fusion is a local trivialization of $\mathcal{H} \boxtimes \mathcal{K}$.
The point here is that the right implementing condition on $\mathcal{H}$ ensures that near any point, there exist local trivializations $u$ and $v$ that intertwine along the same the local trivialization $\varphi$ of $B$, thus providing the fibrewise Connes fusion product with a Hilbert bundle structure, see \cite[Proposition B.21]{ludewig2023spinor}.

The standard bimodule bundles $L^2(\mathcal{A})$, for $\mathcal{A}$ a von Neumann algebra bundle, are the identity 1-morphisms for the composition \cref{CompositionOfBimoduleBundles}.
The associativity of Connes fusion then shows that \cref{CompositionOfBimoduleBundles} is the composition of a bicategory. 

We wrap up the discussion of this section as follows.

\begin{definition}
The \emph{bicategory $\vNAlgBdl^{\mathrm{bi}}(X)$  of von Neumann algebra bundles over $X$} consists of the following data.
\begin{itemize}
\item
Objects are von Neumann algebra bundles;
\item
the category of morphisms $\mathcal{B} \to \mathcal{A}$ is the category $\mathcal{A}\text{-}\mathcal{B}\text{-}\BimodBdl^{\mathrm{imp}}(X)$ of right implementing $\mathcal{A}$-$\mathcal{B}$-bimodule bundles over $X$;
\item 
 composition is fibre-wise Connes fusion,  \cref{CompositionOfBimoduleBundles};
\item
the identity morphism of $\mathcal{A}$ is the standard bimodule bundle $L^2(\mathcal{A})$ from \cref{StandardBimoduleBundle}; and
\item
associators and unitors are the bundle maps obtained by taking fibrewise the associators and unitors of the bicategory of von Neumann algebras and bimodules.
\end{itemize}
\end{definition}

If $f: X \to Y$ is a continuous map, we obtain an obvious pullback functor 
        \begin{equation*}
        f^* : \vNAlgBdl^{\mathrm{bi}}(Y) \to  \vNAlgBdl^{\mathrm{bi}}(X).
        \end{equation*}
        Hence, the bicategories $\vNAlgBdl^{\mathrm{bi}}(X)$ assemble to a presheaf of bicategories $\vNAlgBdl^{\mathrm{bi}}$ on the category $\top$ of topological spaces.
This presheaf is actually a pre-2-stack, since bimodule bundles form a stack (see \cite[Prop. 4.5.1]{Kristel2022}).

\subsection{Stackification}

\label{stackification}

The pre-2-stack $\vNAlgBdl^{\mathrm{bi}}$ is a \emph{preliminary} version of the 2-stack of 2-Hilbert bundles. 
It is preliminary because this pre-2-stack does not satisfy descent, and needs to be stackified.
 This phenomenon is well-understood  in the smooth setting, where the plus construction $(..)^{+}$ of Nikolaus-Schweigert \cite{nikolaus2}  can be used to turn a pre-2-stack into a 2-stack. 
 This has been extensively studied for finite-dimensional, smooth 2-vector bundles in \cite{Kristel2020} and can be carried over to the continuous von Neumann algebra setting in a straight-forward way.

We remark that the category $\top$ of topological spaces has several inequivalent Grothendieck topologies. The most common one is the Grothendieck topology generated by open covers, which is equivalent to the one generated by locally split maps, i.e., continuous maps $\pi:Y \to X$ such that each point $x\in X$ has an open neighborhood with a section. We use this Grothendieck topology in the plus construction. 

\begin{definition}
\label{2StackofvNA2VectorBundles}
        The \emph{2-stack of 2-Hilbert bundles} is defined by
\begin{equation*}
\vNAtwoVectBdl := (\vNAlgBdl^{\mathrm{bi}})^{+}\text{.}
\end{equation*} 
\end{definition}

In \cite[\S2.3]{Kristel2020} we spelled out all details of the plus construction in the finite-dimensional smooth setting, and this explicit description carries over to the current context essentially without changes. 
Here, we only need two pieces: the 2-Hilbert bundles themselves, and a certain notion of isomorphism, called refinement. 
We now recall these notions.

\begin{definition}\label{Definition2VectorBundle}
        A \emph{2-Hilbert bundle} over a space $X$ is a tuple $\mathscr{V}=(Y,\pi,\mathcal{A},\mathcal{M},\mu)$ consisting of a locally split map $\pi :Y \to X$, a von Neumann algebra bundle $\mathcal{A}$ over $Y$, an invertible $\pr_2^*\mathcal{A}$-$\pr_1^*\mathcal{A}$-bimodule bundle $\mathcal{M}$ over $Y^{[2]}$ and a unitary intertwiner
        \begin{equation*}
                \mu: \pr_{23}^* \mathcal{M} \boxtimes \pr_{12}^*\mathcal{M} \to \pr_{13}^*\mathcal{M}
        \end{equation*}
        of $\pr_3^*\mathcal{A}$-$\pr_1^*\mathcal{A}$-bimodule bundles over $Y^{[3]}$ such that the diagram
        \begin{equation*}
\alxydim{@C=7em}{\pr_{34}^*\mathcal{M} \boxtimes \pr_{23}^*\mathcal{M} \boxtimes \pr_{12}^*\mathcal{M} \ar[r]^-{\id \boxtimes \pr_{123}^*\mu} \ar[d]_{\pr_{234}^*\mu \boxtimes \id}  & \pr_{34}^*\mathcal{M} \boxtimes \pr_{13}^*\mathcal{M} \ar[d]^{\pr_{134}^*\mu} \\ \pr_{24}^*\mathcal{M} \boxtimes \pr_{12}^*\mathcal{M} \ar[r]_{\pr_{124}^*\mu} & \pr_{14}^*\mathcal{M} }
\end{equation*}
over $Y^{[4]}$ commutes.
\end{definition}

This structure can be depicted as follows:
\begin{equation*}
        \mathcal{V} = \left\{
                \alxydim{}{
                        \mathcal{A} \ar[d] & \mathcal{M} \ar[d] & \mu \ar@{.}[d]   \\
                        Y \ar[d]& Y^{[2]} \arr[l]  & Y^{[3]}  \arrr[l] \\
                        X  
                }
        \right\}
\end{equation*}

In \cref{sec:ass2vb} we describe two examples of 2-Hilbert bundles, and show that they are isomorphic. 
The isomorphisms we introduce are so-called \emph{refinements}, parallel to  \cite[Def. 3.5.1]{Kristel2020}. 

\begin{definition}
\label{def:refinement}
Let $\mathscr{V}=(Y,\pi,\mathcal{A},\mathcal{M},\mu)$ and $\mathscr{V}'=(Y',\pi',\mathcal{A}',\mathcal{M}',\mu')$ be 2-Hilbert bundles over $X$. 
A \emph{refinement} $\mathscr{V} \to \mathscr{V}'$ is a triple  $\mathscr{R}=(\rho,\varphi,u)$ consisting of a continuous map $\rho:Y \to Y'$ such that $\pi'\circ \rho=\pi$, of an isomorphism $\varphi:\mathcal{A}\to \rho^{*}\mathcal{A}'$ of von Neumann  algebra bundles over $Y$, and of 
a bimodule bundle isomorphism $u:\mathcal{M} \to (\rho^{[2]})^* \mathcal{M}'$  over $Y^{[2]}$ 
along the algebra homomorphisms $\pr_1^{*}\varphi$ and $\pr_2^{*}\varphi$,
such that the diagram
\begin{equation}
\label{eq:refinementdiag}
\begin{aligned}
\xymatrix@C=6em{\pr_{23}^{*}\mathcal{M} \boxtimes \pr_{12}^{*}\mathcal{M}\ar[d]_{\pr_{23}^{*}u \boxtimes \pr_{12}^{*}u} \ar[r]^-{\mu} & \pr_{13}^{*}\mathcal{M} \ar[d]^{\pr_{13}^{*}u}\\   (\rho^{[2]})^*(\pr_{23}^{*}\mathcal{M}' \boxtimes \pr_{12}^{*}\mathcal{M}')  \ar[r]_-{(\rho^{[3]})^{*}\mu'} & \pr_{13}^{*}(\rho^{[2]})^*\mathcal{M}'}
\end{aligned}
\end{equation}
is commutative.
\end{definition}

\subsection{Associated 2-Hilbert bundles}
\label{SectionAssociatedBimoduleBundles}

Throughout this section, we fix a von Neumann algebra $A$.
We consider the sub-bicategory of $\vNAlgBdl^{\mathrm{bi}}(X)$ over a single object, the trivial von Neumann algebra bundle $\underline{A}=X \times A$. 
This is the delooping 
\begin{equation}
\label{SubcategoryOverA}
\mathscr{B}(\underline{A}\text{-}\underline{A}\text{-}\BimodBdl^{\mathrm{imp}}(X))\subset \vNAlgBdl^{\mathrm{bi}}(X)
\end{equation}
of the monoidal category $\underline{A}\text{-}\underline{A}\text{-}\BimodBdl^{\mathrm{imp}}(X)$ of endomorphisms of $\underline{A}$. 
Letting $X$ vary, we obtain a presheaf $\underline{A}\text{-}\underline{A}\text{-}\BimodBdl^{\mathrm{imp}}$ of monoidal categories. 
One can show that this presheaf is in fact  a monoidal stack.

Next we consider the automorphism 2-group  $\mathcal{U}(A)$ as defined in \cref{Automorphism2Group}, and set up in the following a morphism of monoidal stacks
\begin{equation}
\label{BimoduleFunctor}
\Mod: \Bdl {\mathcal{U}(A)} \to  \underline{A}\text{-}\underline{A}\text{-}\BimodBdl^{\mathrm{imp}}
\end{equation}
from principal $\mathcal{U}(A)$-bundles (see \cref{sec:principalbundles}) to $\underline{A}$-$\underline{A}$-bimodule bundles. It will be the central ingredient of our associated 2-Hilbert bundle construction.

First of all, we recall that associated to the topological strict 2-group $\mathcal{G}=\mathcal{U}(A)$ are the topological groups $\mathcal{G}_0=\Aut(A)$, $\mathcal{G}_1=N(A)$ 
and $\mathcal{G}_s=\U(A)$.
Let $P$ be a principal $\mathcal{U}(A)$-bundle over a topological space $X$,   \ie $P$ is a principal $\U(A)$-bundle over $X$ together with anchor $\phi: P \to \Aut(A)$ satisfying \cref{AnchorPropertyCrossedModuleBundle}, which we write as $\phi(pu) = t(u^{*})\phi(p)$.
We define
\begin{equation}
\label{AssociatedBimodule}
        \Mod (P) := (P \times L^2(A))/\U(A)\text{,}
\end{equation}
where the $\U(A)$-action is the diagonal right action; i.e., we identify $(p,\xi) \sim (p\cdot u,\xi \ract u)$, where $p\in P$, $\xi\in L^2(A)$, and $u\in \U(A)$.
As $\U(A)$ acts strongly continuously on $L^2(A)$, the usual associated bundle construction provides $\Mod(P)$ with the structure of a  Hilbert  bundle. 
%
We equip the fibres of $\Mod (P)$ with the $\underline{A}$-$\underline{A}$-bimodule bundle structure defined by
\begin{equation}
\label{bimodule-actions-on-associated-bundle}
        a \lact [p,\xi] \ract b := [p,a \lact \xi \ract \phi(p)(b)], \qquad a, b \in A, ~~\xi \in L^2(A), ~~p \in P.
\end{equation}
One easily checks well-definedness on equivalence classes.
%
%
Next we show that $\Mod(P)$ is a right implementing $\underline{A}$-$\underline{A}$-bimodule bundle in the sense of \cref{von-Neumann-algebra-bundles}.
Any section $p$ of $P$ defined over an open set $O \subset X$ gives a local trivialization
        \begin{equation}
                \label{eq:intN}
                \tau_p: \Mod (P)|_O \to \sheaf{L^2(A)}_{\,\phi(p)} \qquad [p(x),\xi] \mapsto (x, \xi) \text{,}
        \end{equation}
        where the right hand side denotes the $\underline{A}$-$\underline{A}$-bimodule bundle obtained by twisting the right action with $\phi \circ p : O \to \Aut(A)$, see \cref{ExTwistedRightAction}. 
In other words, $\tau_p$ is a family of unitary isomorphisms that is intertwining along the identity and $\phi(p)$.
Indeed, 
        \begin{equation} \label{IntertwiningUsigma2}
                \tau_p (a \lact [p(x), \xi] \ract b) = \tau_p\bigl( \bigl[p(x), a \lact \xi \ract \phi(p(x))(b)\bigr]\bigr) = \bigl(x, a\lact \xi \ract \phi(p(x))(b)\bigr).
        \end{equation}
        In other words, $\tau_p$ is an intertwiner along the local trivializations $\underline{A}|_O \to O \times A$ given by $(x,a)\mapsto (x,a)$ and $(x,a) \mapsto (x,\phi(p(x))(a))$, respectively. 
        This shows that $\Mod(P)$ is a bimodule bundle. 
        It is right implementing by \cite[Example~B.15]{ludewig2023spinor}.
       For later use, it will be good to determine the transition function between two local trivializations.  
If $p^\prime: V \to P$ is another local section defined over an open set $O^\prime \subseteq X$, then, over $O\cap O^\prime$, we have $p^\prime = p \cdot u$ for a continuous function $u : O \cap O^\prime \to \U(A)$.
                Therefore,
                \begin{equation*}
                        (\tau_{p^\prime} \circ \tau_p^{*})(x, \xi) = \tau_{p^\prime}([p(x), \xi]) = \tau_{p^\prime}([p^\prime(x) \cdot u(x)^*, \xi]) = \tau_{p^\prime}([p^\prime(x), \xi \ract u(x)]) = (x, \xi \ract u(x)).
                \end{equation*}
                for all $x \in O \cap O^\prime$. 
                We see that the transition functions $\tau_{p^\prime} \circ \tau_p^*$ are given by right multiplication by $u$, which is intertwining along the identity and $t(u)$.
                In other words, we have the commutative diagram of bimodule bundle isomorphisms over $O \cap O':$
                 \begin{equation}
                 \label{DiagramChangeOfTrivialization}
        \alxydim{}{
                & \Mod (P) \ar[dl]_{\tau_{p}} \ar[dr]^{\tau_{p'}} & \\
                \sheaf{L^2(A)}_{\phi(p)} \ar[rr]_{\ract \, u}  & & \sheaf{L^2(A)}_{\phi(p^\prime)}.
                } 
        \end{equation}   
        
It is straightforward to see that any morphism $f:P \to Q$ between between principal $\mathcal{U}(A)$-bundles defines an intertwiner 
\begin{equation}
\label{FormulaModf}
\Mod(f) : \Mod(P) \to \Mod(Q), \qquad [p, \xi] \mapsto [f(p), \xi]\text{.}
\end{equation}
For later use, we observe that since $f$ intertwines the anchors $\phi:P \to \Aut(A)$ and $\psi:Q \to \Aut(A)$, i.e., $\phi = \psi \circ f$, it has the property that the diagram
        \begin{equation} \label{FunctorModProperty}
        \alxydim{}{
               \Mod(P) \ar[r]^-{\Mod(f)} \ar[d]_{\tau_{p}} & \Mod(Q) \ar[d]^{\tau_{f\circ p}} \\
                \sheaf{L^2(A)}_{\phi(p)} \ar@{=}[r]& \sheaf{L^2(A)}_{\psi(f \circ p)}
               }
        \end{equation}
        commutes for each local trivialization $p$ of $P$.

So far we have defined $\Mod$ as a functor. 
It is clear $\Mod$ it is compatible with pullbacks, and so it is a stack morphism as in \cref{BimoduleFunctor}.
It remains to verify that it is monoidal, relating the tensor product of principal $\mathcal{U}(A)$-bundles (see \cref{sec:principalbundles}) with the Connes fusion product of bimodule bundles defined in \cref{von-Neumann-algebra-bundles}.
We use again the local trivializations $\tau_{p}$ constructed in \cref{eq:intN} from a local section $p$ of $P$, and recall  from \cref{IntertwiningUsigma2} that they are bimodule bundle isomorphisms $\Mod (P)|_O \to \sheaf{L^2(A)}_{\phi( p)}$. 
 
\begin{proposition}
        \label{prop:compassbimod}
        Let $P$ and $Q$ be principal $\mathcal{U}(A)$-bundles with anchors  $\phi:P \to \Aut(A)$ and $\psi:Q \to \Aut(A)$, respectively. 
        Then, there exists a unique intertwiner 
        \begin{equation} \label{MonoidalityMorphism}
            \Mod (P) \boxtimes \Mod(Q) \cong \Mod(P \otimes Q)
        \end{equation}
        such that for all local sections $p$ of $P$ and $q$ of $Q$ over a common open set $O \subset X$, the diagram
        \begin{equation} \label{DefiningDiagramMonoidalityMorphism}
        \alxydim{@C=5em}{
               \Mod(P) \boxtimes \Mod(Q)  \ar[d]_{\tau_p \boxtimes \tau_q} \ar[r] & \Mod(P \otimes Q) \ar[d]^{\tau_{p \otimes q}} \\
               \sheaf{L^2(A)}_{\phi ({p})} \boxtimes \sheaf{L^2(A)}_{\psi (q)}  \ar[r]_-{\chi_{\phi (p), \psi(q)}} & 
\sheaf{L^2(A)}_{\phi (p) \circ \psi (q)}  
                 }
        \end{equation}
        is commutative.
        Here, $\chi$ is the natural transformation \cref{IsoOfTwistedStandardForms}.
\end{proposition}

\begin{proof}
First we observe that by the definition \cref{DefinitionAnchorTensorProduct} of the anchor map of the tensor product $P \otimes Q$, we have $\phi(p) \circ \psi(q) = (\phi \otimes \psi)(p \otimes q)$; hence, the target of the local trivialization $\tau_{p \otimes q}$ is indeed  $\sheaf{L^2(A)}_{\phi (p) \circ \psi (q)}$.

        It is clear that  diagram \cref{DefiningDiagramMonoidalityMorphism} determines the intertwiner completely for a given choice of local sections. 
        It remains to prove that different choices of sections yield the same intertwiner. 
        Let $p$ and $p' = p \cdot u$ be two different local sections of $P$ over some open set $O$, for which we have the commutative diagram \cref{DiagramChangeOfTrivialization}.
        Similarly, let $q$ and $q' = q \cdot v$ be local sections of $Q$ over $O$, for which we have an analogous commutative triangle.
        These two triangles -- together with functoriality of the fusion product -- show that the left square of the diagram
           \begin{equation} \label{BigDiagramInFunctorialityOfModProof}
                \alxydim{@C=.7em}{
               \Mod(P) \boxtimes \Mod(Q)  \ar[dr]_{\tau_p \boxtimes \tau_q} \ar@{=}[ddd] \ar[rrrrrr] &&&& & & \Mod(P \otimes Q) \ar[dl]^{\tau_{p \otimes q}} \ar@{=}[ddd] \\
               &\sheaf{L^2(A)}_{\phi(p)} \boxtimes \sheaf{L^2(A)}_{\psi(q)} \ar[d]_{(\ract \, u) \boxtimes (\ract \, v)} \ar[rrrr]^-{\chi_{\phi(p), \psi(q)}}  &&&& \sheaf{L^2(A)}_{\phi(p)\psi(q)}  \ar[d]^{\ract \, \phi(p)(v)u} & \\ 
              &   \sheaf{L^2(A)}_{\phi(p')} \boxtimes \sheaf{L^2(A)}_{\psi(q')} \ar[rrrr]_-{\chi_{\phi(p^\prime), \psi(q^\prime)}} &&&& \sheaf{L^2(A)}_{\phi(p')\psi(q')} & \\
               \Mod(P) \boxtimes \Mod(Q)  \ar[ur]_{\tau_{p^\prime} \boxtimes \tau_{q^\prime}} \ar[rrrrrr] &&& && & \Mod(P \otimes Q) \ar[ul]_{\tau_{p^\prime \otimes q^\prime}}\text{.}
                }
        \end{equation}
         is commutative.
The top and bottom square are the defining squares \cref{DefiningDiagramMonoidalityMorphism} for the morphism \cref{MonoidalityMorphism}. 
        Commutativity of the central square is a special case of \cref{Naturality}.
        %
           Hence, in order to verify that the top and bottom horizontal maps agree, it remains to show that the right square is commutative.

        To this end, we need to compare the elements $p\otimes q$ and $p'\otimes q'$ of $P \otimes Q$.
        We recall that $P \otimes Q$ is a quotient of $P \times_M Q$ by the equivalence relation \cref{eq:MonoidalStructureQuotient}, and that the $\U(A)$-action on $P \otimes Q$ is induced by the action of $\U(A)$ on the first factor of $P \times_M Q$.
        Using these rules, we calculate
        \begin{align*}
                p' \otimes q' &= p\cdot u \otimes q\cdot v
                                    \\&=p\cdot u \otimes q\cdot \phi(p\cdot u)^{-1}\bigl( \phi(p\cdot u)(v)\bigr)
                                    \\&=p \cdot u \cdot \phi(p\cdot u)(v) \otimes q       & &\cref{eq:MonoidalStructureQuotient}
                                    \\&=(p \otimes q)\cdot  u  \phi(p\cdot u)(v)          & &\text{ (Definition of right action)}
                                    \\&=(p \otimes q)\cdot u  (t(u^{*})\circ \phi(p))(v)  & &\cref{AnchorPropertyCrossedModuleBundle}
                                    \\&=(p \otimes q)\cdot \phi(p)(v)  u\text{.}
        \end{align*}         
        With a view on \cref{DiagramChangeOfTrivialization}, this shows commutativity of the right square in \cref{BigDiagramInFunctorialityOfModProof}.
\end{proof}

Now we are in position to describe our construction of associated 2-Hilbert bundles.
Let $\mathcal{G}$ be a topological strict 2-group and let $\Rep:\mathcal{G} \to \mathcal{U}(A)$ be a unitary representation on a von Neumann algebra $A$.
As recalled in \cref{lem:bundleextension}, $\Rep$ induces a morphism of monoidal stacks 
\begin{equation*}
\Rep_* : \Bdl{\mathcal{G}} \to \Bdl{\mathcal{U}(A)},
\end{equation*}
which can be composed with the morphism $\Mod$ from \cref{BimoduleFunctor}, resulting in a morphism
\begin{equation}
\label{ModR}
\Mod_\Rep := \mathcal{R}_{*}  \circ \Mod : \Bdl{\mathcal{G}} \to \underline{A}\text{-}\underline{A}\text{-}\BimodBdl^{\mathrm{imp}}
\end{equation}
of monoidal stacks.

For convenience, we will spell out the composition \cref{ModR}, and simplify the result slightly. If $P$ is a principal $\mathcal{G}$-bundle over $M$, then we have in the first place
\begin{equation}
\label{ModR-spelled-out}
\Mod_{\mathcal{R}}(P) =\Big(\big((P \times\U(A))/\mathcal{G}_s\big) \times L^2(A)\Big)/\U(A)\text{,} 
\end{equation}
where $g\in \mathcal{G}_s$ acts by $(p,u)\cdot g =(pg,\mathcal{R}_1(g)^{-1}u)$, and $v\in \U(A)$ acts by $([p,u],\xi)\cdot v = ([p,uv],\xi \ract v)$. The bimodule structure is given by
\begin{equation*}
a\lact [[p,u],\xi]\ract b =[[p,u],a\lact  \xi \ract t(u)^{-1}\mathcal{R}_0(\phi(p))(b)]\text{,}
\end{equation*}
where $\phi$ is the anchor of $P$.
However,
 the single quotient 
\begin{equation}
\label{ModR-simplyification}
(P \times L^2(A))/\mathcal{G}_s\text{,}
\end{equation}
where $\mathcal{G}_s$ acts on $P \times L^2(A)$ by $(p,\xi)\cdot g := (pg,\xi \ract \mathcal{R}_1(g))$ yields a canonically isomorphic Hilbert bundle. Indeed, the  map $[p,\xi] \mapsto [[p,1],\xi]$ is a continuous, fibre-preserving, fibre-wise unitary isomorphism. 
The bimodule action then simply becomes
\begin{equation}
\label{BimoduleStructureModLambda}
a\lact [p,\xi]\ract b =[p,a \lact \xi \ract \mathcal{R}_0(\phi(p))(b)]\text{.}
\end{equation}
Summarizing, the bimodule bundle $\Mod_{\mathcal{R}}(P)$ is given by the quotient \cref{ModR-simplyification} and is equipped with the bimodule actions \cref{BimoduleStructureModLambda}.
For a morphism $f:P \to Q$ of principal $\mathcal{G}$-bundles, the intertwiner $\Mod_{\Rep}(f)$ is defined by the same formula \cref{FormulaModf} as before.

\begin{definition}
\label{def:associated2vectorbundle}
Let $\mathcal{G}$ be a topological strict 2-group, and let $\Rep: \mathcal{G} \to \mathcal{U}(A)$ be a unitary representation of $\mathcal{G}$ on a von Neumann algebra $A$. 
The \emph{associated 2-Hilbert bundle construction} is the morphism of 2-stacks
\begin{equation*}
\alxydim{@C=4em}{{\mathcal{G}}\text{-}\Grb = \mathscr{B}(\Bdl{\mathcal{G}})^{+} \ar[r]^-{\mathscr{B}\Mod_{\Rep}^+} & \mathscr{B}(\underline{A}\text{-}\underline{A}\text{-}\BimodBdl^{\mathrm{imp}})^+ \subseteq (\vNAlgBdl^{\mathrm{bi}})^{+} = \vNAtwoVectBdl \text{.}}
\end{equation*}
If $\mathcal{Q}$ is a ${\mathcal{G}}$-bundle gerbe over a space $X$, its image is called the \emph{associated 2-Hilbert bundle} (for the representation $\Rep$) and is denoted by $\mathcal{Q} \times_{\mathcal{G}} A$. 
\end{definition}

Here, we have used the definition of ${\mathcal{G}}$-bundle gerbes via the plus construction, see \cref{def:bundlegerbe,DiffeologicalPlusConstruction}. 
We use further that $\Mod_{\mathcal{R}}$ induces a functor $\mathscr{B}\Mod_{\mathcal{R}}$ between bicategories with a single object, due to the fact that it is monoidal; and finally, we use that the plus construction is functorial.

We shall spell out  the data of the associated 2-Hilbert bundle $\mathcal{Q} \times_{{\mathcal{G}}} A$ explicitly. 
For this purpose, we suppose that a ${\mathcal{G}}$-bundle gerbe $\mathcal{Q}$ over $X$ consists of a  locally split map $\pi:Y \to X$, of a principal ${\mathcal{G}}$-bundle $P$ over $Y^{[2]}$, and of a bundle gerbe product $\mu$ over $Y^{[3]}$, just as in \cref{def:bundlegerbe}. Then, the associated  2-Hilbert bundle $\mathcal{Q} \times_{{\mathcal{G}}} A$ is the following:  
\begin{itemize}
\item 
Its locally split map is $\pi:Y \to X$.
\item
Its von Neumann algebra bundle over $Y$ is the trivial bundle $\underline{A} = Y \times A$. 
\item
Its bimodule bundle over $Y^{[2]}$ is $\Mod_\Rep(P)$.
\item
Its product over $Y^{[3]}$ is given by $\Mod_{\Rep}(\mu)$. 
More precisely, it is the composite
\begin{equation}
\label{DefinitionModTilde}
        \alxydim{@C=5em@R=1.5em}{\pr_{23}^{*}\Mod_\Rep(P) \boxtimes \pr_{12}^{*}\Mod_\Rep(P) \ar[d]_{\cong}   &
           \\
          \Mod_\Rep(\pr_{23}^{*}P) \boxtimes \Mod_\Rep(\pr_{12}^{*}P) \ar[d]_{\cong} & \\
           \Mod_\Rep(\pr_{23}^{*}P \otimes \pr_{12}^{*}P)  \ar[r]^-{\Mod_\Rep(\mu)}   & \Mod_\Rep(\pr_{13}^{*}P) \ar[d]^{\cong} \\& \pr_{13}^{*}\Mod_\Rep(P)\text{,}}
\end{equation} 
where the vertical arrows are the canonical structure isomorphisms of the monoidal stack morphism $\Mod_{\mathcal{R}}$.
\end{itemize}

\section{The stringor bundle}

\label{sec:ass2vb}

Let $M$ be a string manifold. There are at least four equivalent options to model string structures on $M$; 
in the present paper, two of these options are relevant: (a) fusive loop-spin structures and (b) $\String(d)$-bundle gerbes.  
In \cref{SubsectionStringStructuresFusiveLoopSpin} we recall these  and recall the relation between them, on the basis of  \cref{sec:strstr}.  
In  \cref{sec:stringor} we  define the Stolz-Teichner stringor bundle as a 2-Hilbert bundle, based on a fusive loop-spin structure on $LM$. 
In \cref{sec:identification}, we prove our main result: the stringor bundle is isomorphic to the 2-Hilbert bundle obtained by associating the stringor representation to the $\String(d)$-bundle gerbe obtained from the fusive loop-spin structure.

\subsection{Fusive loop-spin structures and string structures  }
\label{SubsectionStringStructuresFusiveLoopSpin}

We recall that a spin structure on an oriented Riemannian manifold $M$ is a principal $\Spin(d)$-bundle $\Spin(M)$ over $M$ that lifts the orthogonal frame bundle of $M$; i.e., it is equipped with a smooth map $q:\Spin(M) \to  \mathrm{SO}(M)$ that covers the identity on $M$ and is equivariant along the projection $\Spin(d) \to \SO(d)$.
Taking free loops in $\Spin(M)$, we obtain a principal $L\Spin(d)$-bundle $L\Spin(M) $ over $LM$. 
Let $\widetilde{L\Spin}(d)$ be a basic central extension of $L\Spin(d)$ (see \cref{SectionStringorRep}). 
A \emph{loop-spin structure} is a principal $\widetilde{L\Spin}(d)$-bundle $\widetilde{L\Spin}(M)$ over $LM$ together with a smooth map
\begin{equation*}
        p:\widetilde{L\Spin}(M) \to L\Spin(M)
\end{equation*} 
covering the identity on $LM$, and which is equivariant along the projection $\widetilde{L\Spin}(d) \to L\Spin(d)$ of the basic central extension \cite{killingback1}.
In other words, a loop-spin structure is a lift of the structure group of $\widetilde{L\Spin}(M)$ from $L\Spin(d)$ to its basic central extension.
We remark that the map $p$ is automatically (the projection of) a principal $\U(1)$-bundle. 

For the following definition, we use the fact that the basic central extension has a canonical  \emph{fusion product} \cite[Definition 3.4]{Waldorfa}, i.e., an isomorphism
        \begin{equation*}
                \mu : \mathrm{pr}_{23}^*\cup^{*} \widetilde{L\Spin}(d) \otimes \mathrm{pr}_{12}^*\cup^{*} \widetilde{L\Spin}(d)  \to \mathrm{pr}_{13}^*\cup^{*} \widetilde{L\Spin}(d)
        \end{equation*}
of principal $\U(1)$-bundles over $P\Spin(d)^{[3]}$, which is associative over $P\Spin(d)^{[4]}$ and additionally a group homomorphism. Fusion products on loop group extensions determine, and are determined by fusion factorizations \cite[\S 5]{Kristel2019}; in the present situation, we may use the fusion factorization $i$ from \cref{FusionFactorization} and set
\begin{equation}
\label{fusion-product-on-basic-ce}
 \mu (X_{23}\otimes X_{12}) 
 :=X_{23}i(\gamma_2)^{*}X_{12}\text{,}
\end{equation}
where $X_{12}$ projects to $\gamma_1\cup\gamma_2$ and $X_{23}$ projects to $\gamma_2\cup\gamma_3$.

\begin{definition}
\label{DefinitionFusiveSpinStructure}
A \emph{fusive loop-spin structure} is a loop-spin structure whose principal $\U(1)$-bundle $p$ is equipped with a fusion product, i.e. a bundle isomorphism
 \begin{equation*}
  \lambda : \mathrm{pr}_{23}^*\cup^{*} \widetilde{L\Spin}(M) \otimes \mathrm{pr}_{12}^*\cup^{*} \widetilde{L\Spin}(M)  \to \mathrm{pr}_{13}^*\cup^{*} \widetilde{L\Spin}(M)
 \end{equation*}     
 over $P\Spin(M)^{[3]}$  
 that is associative over $P\Spin^{[4]}$ and is compatible with the fusion product $\mu$ on the basic central extension under the principal action, i.e., 
\begin{equation}\label{eq:FusionIsEquivariant}
                \lambda\bigl( (\Phi_{23} \cdot X_{23}) \otimes (\Phi_{12} \cdot X_{12})\bigr) = \lambda(\Phi_{23} \otimes \Phi_{12}) \cdot \mu(X_{23}\otimes X_{12})
        \end{equation}
        holds for all $\Phi_{12}, \Phi_{23} \in \widetilde{L\Spin}(M)$ and $X_{12}, X_{23} \in \widetilde{L\Spin}(d)$ such that both sides of \cref{eq:FusionIsEquivariant} are defined. 
\end{definition}

Fusive loop-spin structures can be viewed as a loop space version of string structures on manifolds \cite{Waldorfa,Waldorfb}. We now relate them to the following, more instructive notion of a string structure. We assume again that $M$ is a spin manifold with spin structure $\Spin(M)$. A \emph{string structure} on $M$ is a principal $\String(d)$-bundle gerbe $\String(M)$ that lifts the structure group of $\Spin(M)$  along the 2-group homomorphism
\begin{equation*}
q:\String(d) \to \Spin (d)_{\dis} 
\end{equation*}
from \cref{projection-to-spin}. We refer to \cite[\S 7]{Nikolausa} and to \cref{sec:strstr} for a discussion and comparison to yet other versions of string structures. The main idea is to view the bundle gerbe $\String(M)$ as the string-oriented frame bundle of the string manifold $M$.

  We describe now how to convert a fusive loop-spin structure into a string structure. 
We provide in \cref{sec:strstr} a proof showing that  this conversion is well-defined and fits into a partially known picture of equivalences between different notions of string structures.
We first have to fix a point $x\in M$ and a spin-oriented frame ${\hat x}$ at $x$, i.e. an element $\hat x \in \Spin(M)_x$. Then, the bundle gerbe $\stringgerbe(M)$ consists of the following structure:
\begin{itemize}
        \item 
                        The surjective submersion  $Y:=P_{\hat x}\Spin(M) \to M$ is the endpoint evaluation, followed by the bundle projection $\Spin(M) \to M$.
        \item 
                Over the double fibre product $Y^{[2]}$, we have the following $\String(d)$-principal bundle (in the sense of \cref{def:principalgammabundle}): 
                \begin{itemize}
                        \item 
                                Its total space is (see \cref{AppEqBundleQ}) 
                                \begin{equation}
                                \label{EqBundleQ}
                                P:= P_{\hat x}\Spin(M)^{[2]} \times_{L\Spin(M)} \widetilde{L\Spin}(M) \times P_e\Spin(d)\text{,}
                                \end{equation}
                                where $P_{\hat x}\Spin(M)^{[2]}$ denotes the fibre product over $\Spin(M)$. Thus, its elements are quadruples $(\beta_1,\beta_2,\Phi,\gamma)$,
                                where $\beta_1, \beta_2 \in P_{\hat x} \Spin(M)$ with $\beta_1(\pi) = \beta_2(\pi)$, $\Phi$ is a lift of $\beta_1 \cup \beta_2 \in L\Spin(M)$ to $\widetilde{L \Spin}(M)$ and $\gamma \in P_e \Spin(d)$. 
\\
                                %
                        \item
                                The bundle projection is (see \cref{AppEqBundleProjectionQ}) 
                                \begin{equation} 
                                \label{EqBundleProjectionQ}
                                        \recall{AppEqBundleProjectionQ}.
                                \end{equation}
                                %
                        \item
                                The anchor map is (see \cref{AppEqAnchorQ}) 
                                \begin{equation} \label{EqAnchorQ}
                                        \recall{AppEqAnchorQ}.
                                \end{equation} 
\item
the principal $\String(d)_s$-action is (see \cref{AppEqPrincipalActionQ}) 
\begin{equation}
\label{EqPrincipalActionQ}
\recall{AppEqPrincipalActionQ}\text{.}
\end{equation}                                
Here, $i: P_e\Spin(d) \to \widetilde{L\Spin}(d)$ is the fusion factorization of \cref{FusionFactorization}. 
                                
                \end{itemize}
        \item
                On the triple fibre product $Y^{[3]}$, the  bundle gerbe product
                \begin{equation} 
                \label{EqBundleGerbeProduct}
                        \mu_{\stringgerbe(M)} : \pr_{23}^{*}P \otimes \pr_{12}^{*}P \to \pr_{13}^* P
                \end{equation}
is given by (see \cref{Appeq:bgp:G})
                \begin{equation}
                        \label{eq:bgp:G}
                        \recall{Appeq:bgp:G}\text{,}
                \end{equation}
                where $\lambda$ is the fusion product of the fusive loop-spin structure, see \cref{DefinitionFusiveSpinStructure}. 
\end{itemize}
This completes the definition of the (smooth) $\String(d)$-bundle gerbe $\stringgerbe(M)$. 
Later in \cref{sec:identification}, we will pass to the underlying topologies and regard $\stringgerbe(M)$ as a topological bundle gerbe, without introducing an explicit notation. 

\subsection{The stringor bundle of Stolz-Teichner}


\label{sec:stringor}

In this section, we construct a  2-Hilbert bundle $\mathscr{S}(\widetilde{L\Spin}(M))$ on the string manifold $M$, using a fusive loop-spin structure $\widetilde{L\Spin}(M)$ on the loop space $LM$ of  $M$.
The construction stems from the loop space approach to string geometry; it has been outlined by Stolz-Teichner \cite{stolz3} and then constructed rigorously in \cite{Kristel2019,kristel2020smooth,Kristel2019c} in a setting of \quot{rigged von Neumann algebra bundles} over diffeological spaces, which is a method to work with smooth, infinite-dimensional bundles over  infinite-dimensional manifolds.

In the following, we  give an independent definition of the Stolz-Teichner stringor bundle in a purely topological setting, and we show afterwards that it reflects the construction given in \cite{Kristel2019c}. The first ingredient  is the associated von Neumann algebra bundle (see \cref{associated-von-Neumann-algebra-bundle})
\begin{equation}
\label{eq:A}
\mathcal{A} := P\Spin(M) \times_{P\Spin(d)} A,
\end{equation}
where $P\Spin(M)$ is the principal $P\Spin(d)$-bundle over $PM$ obtained by taking flat paths in the total space of the spin structure $\Spin(M)$, and  $P\Spin(d)$ acts on $A$ through through the homomorphism $\omega$ from \cref{KleinOmega}. 
We note that the definition of $\mathcal{A}$ only requires a spin structure on $M$, not the loop-spin structure.

The second ingredient 
is the \emph{spinor bundle on loop space}:  the associated Hilbert bundle 
\begin{equation}
\label{DefinitionSpinorBundle}
  \mathcal{S}_{LM} := \widetilde{L\Spin}(M) \times_{\widetilde{L\Spin}(d)} L^2(A),
\end{equation}
where $\widetilde{L\Spin}(M)$ is the loop-spin structure, and  $\widetilde{L\Spin}(d)$ acts on $L^2(A)$ via the representation $\Omega'$ defined in \cref{representation-on-standard-bimodule}.
We exhibit the pullback $\cup^*\mathcal{S}_{LM}$ to $PM^{[2]}$ as a $\pr_2^{*}\mathcal{A}$-$\pr_1^{*}\mathcal{A}$-bimodule bundle. 
The bimodule actions are defined by
\begin{equation}
\label{LeftRightActionFLM}
[\beta_2,a] \lact [\Phi, \xi] \ract [\beta_1,b] = [\Phi, a \lact \xi \ract b]\text{,}
\end{equation}
where $\Phi \in \widetilde{L\Spin}(M)$ projects to $\beta_1\cup\beta_2 \in L\Spin(M)$.
We show the well-definedness of this bimodule structure:
if $\Phi^\prime$ projects to $\beta^\prime_1 \cup \beta_2^\prime$, then $\beta_i^\prime = \beta_i \cdot \gamma_i$ for $(\gamma_1, \gamma_2) \in P_e \Spin(d)^{[2]}$ and $\Phi^\prime = \Phi \cdot X$ for some $X \in \widetilde{L\Spin}(d)$ projecting to $\gamma_1 \cup \gamma_2$.
We recall that  $ \Omega'(X)$ is left intertwining along $\omega(\gamma_2)$ and right intertwining along $\omega(\gamma_1)$, see \cref{IntertwiningOmegaprime}.
Hence, if formula \eqref{LeftRightActionFLM} holds for $\beta_1, \beta_2$ and $\Phi$, then we also have
\begin{align*}
[\beta_2^\prime,a] \lact [\Phi^\prime, \xi] \ract [\beta_1^\prime,b]
&= [\beta_2, \omega_{\gamma_2}(a)] \lact [\Phi, \Omega(X) \xi] \ract [\beta_1, \omega_{\gamma_1}(b)]\\
&= [\Phi, \omega_{\gamma_2}(a) \lact \Omega(X) \xi \ract \omega_{\gamma_1}(b)]\\
&= [\Phi, \Omega(X)(a \lact \xi \ract b)]\\
&= [\Phi^\prime, a \lact \xi \ract b].
\end{align*}
Any local section $\Phi$ of $\widetilde{L\Spin}(M)$ provides a local trivialization of $\mathcal{S}_{LM}$ (as a Hilbert bundle).
It then follows directly from the formula \cref{LeftRightActionFLM} that the induced local trivialization $\cup^*\Phi$ of $\cup^*\mathcal{S}_{LM}$ is a bimodule trivialization.
%

The third ingredient is 
 the \emph{fusion product} on  $\mathcal{S}_{LM}$, see \cite{stolz3}, \cite[\S5.3]{Kristel2019c}, and \cite[Def. 2.15]{ludewig2023spinor}: a unitary isomorphism
\begin{equation}
\label{FusionProductForS}
\Upsilon: \pr_{23}^{*}\cup^{*}\mathcal{S}_{LM} \boxtimes \pr_{12}^{*}\cup^{*}\mathcal{S}_{LM} \to \pr_{13}^{*}\cup^{*}\mathcal{S}_{LM}
\end{equation}
of $\pr_3^{*}\mathcal{A}$-$\pr_1^{*}\mathcal{A}$-bimodule bundles over $PM^{[3]}$ that satisfies the following associativity condition over $PM^{[4]}$:
\begin{equation}
\label{associativity-of-fusion-product}
\alxydim{@C=6em}{\pr_{34}^{*}\cup^{*}\mathcal{S}_{LM} \boxtimes \pr_{23}^{*}\cup^{*}\mathcal{S}_{LM} \boxtimes \pr_{12}^{*}\cup^{*}\mathcal{S}_{LM} \ar[r]^-{\pr_{234}^{*}\Upsilon \boxtimes \id} \ar[d]_{\id \boxtimes \pr_{123}^{*}\Upsilon} & \pr_{24}^{*}\cup^{*}\mathcal{S}_{LM} \boxtimes \pr_{12}^{*}\cup^{*}\mathcal{S}_{LM} \ar[d]^{\pr_{124}^{*}\Upsilon} \\ \pr_{34}^{*}\cup^{*}\mathcal{S}_{LM} \boxtimes \pr_{13}^{*}\cup^{*}\mathcal{S}_{LM} \ar[r]_-{\pr_{134}^{*}\Upsilon} & \pr_{14}^{*}\cup^{*}\mathcal{S}_{LM} }
\end{equation}
In order to construct $\Upsilon$, we recall that whenever  $\Phi: O \to \widetilde{L\Spin}(M)$ is a local section of the principal bundle $\widetilde{L\Spin}(M)$, we obtain a local trivialization $u: \mathcal{S}_{LM}|_O \to O \times L^2(A)$ of the associated bundle  $\mathcal{S}_{LM}$ by requiring that $u([\Phi,\xi])=\xi$ holds for all $\xi \in L^2(A)$.

\begin{theorem}
\label{unique-fusion-product}
Let $M$ be a spin manifold equipped with a fusive loop-spin structure $\widetilde{L\Spin}(M)$ with fusion product $\lambda$. Then, the spinor bundle on loop space $\mathcal{S}_{LM}$ admits a unique fusion product $\Upsilon$ such that the following condition holds: whenever $O\subset PM^{[3]}$ is an open set with sections $\Phi_{ij}:O \to \widetilde{L\Spin}(M)$ along $\cup\circ \pr_{ij}:PM^{[3]} \to LM$ such that $\Phi_{13}=\lambda(\Phi_{23} \otimes \Phi_{12})$, and $u_{ij}: \pr_{ij}^{*}\cup^{*}\mathcal{S}_{LM}|_O \to O \times L^2(A)$ are the corresponding local trivializations, then the diagram
\begin{equation*}
\xymatrix@C=6em{
 \pr_{23}^{*}\cup^{*}\mathcal{S}_{LM} \boxtimes \pr_{12}^{*}\cup^{*}\mathcal{S}_{LM} \ar[d]_{u_{23} \boxtimes u_{12}} \ar[r]^-{\Upsilon} &
  \pr_{13}^{*}\cup^{*}\mathcal{S}_{LM} \ar[d]^{u_{13}} \\
L^2(A) \boxtimes L^2(A) \ar[r]_-{\chi} & L^2(A)
}
\end{equation*}
is commutative, where $\chi$ is the natural isomorphism \cref{IsoOfTwistedStandardForms}. 
\end{theorem}

\begin{proof}
It is clear that every point $(\beta_1,\beta_2,\beta_3)\in PM^{[3]}$ has an open neighborhood $O$ over which sections $\Phi_{ij}$ with $\Phi_{13}=\lambda(\Phi_{23}\otimes \Phi_{12})$ exist.
This shows uniqueness of $\Upsilon$.
%
For existence, we define $\Upsilon|_O$ separately on each open set $O$ for some fixed choices of sections $\Phi_{ij}$, in such a way that above diagram is commutative. This yields  unitary isomorphisms of bimodule bundles.  Next we show that these isomorphisms do not depend on the choice of sections. This is proved in \cite[Thm. 5.3.1]{Kristel2019c}; for the sake of clarity we  adapt the proof to the present setting.  

Let $\Phi_{ij}^\prime = \Phi_{ij} \cdot X_{ij}$ be a different choice of local sections, with $X_{ij} \in \widetilde{L\Spin}(d)$ lifting the loop $\gamma_i \cup \gamma_j \in L\Spin(d)$, and such that $\Phi_{13}^\prime = \lambda(\Phi_{23}^\prime \otimes  \Phi_{12}^\prime)$. 
By \cref{eq:FusionIsEquivariant,fusion-product-on-basic-ce}, this implies that  
\begin{equation}
\label{X13FromX23AndX12}
X_{13} =\mu(X_{23} \otimes X_{12})= X_{23} \, i(\gamma_2)^{-1} X_{12}.
\end{equation}
Let $u_{ij}^\prime$ be the local trivialization of $\mathrm{pr}_{ij}^*\cup^{*}\mathcal{S}_{LM}$ corresponding to $\Phi_{ij}^\prime$.
Then, $u_{ij}^\prime =\Omega'(X_{ij}) \circ u_{ij}$, and by functoriality of Connes fusion,
\begin{equation*}
  u_{23}^\prime \boxtimes u_{12}^\prime = \bigl(\Omega'(X_{23}) \boxtimes \Omega'(X_{12})\bigr) \circ(u_{23} \boxtimes u_{12}) .
\end{equation*}
Therefore,
\begin{equation*}
\begin{aligned}
\Upsilon^\prime_O &=  (u_{13}^\prime)^* \chi (u_{23}^\prime \boxtimes u_{12}^\prime) = u_{13}^*\Omega'(X_{13})^* \chi \bigl(\Omega'(X_{23}) \boxtimes \Omega'(X_{12})\bigr) (u_{23} \boxtimes u_{12}).
\end{aligned}
\end{equation*}
Comparing with $\Upsilon_O = u_{13}^* \chi (u_{23} \boxtimes u_{12})$, we aim to show
\begin{equation}
\label{AimToShow}
\Omega' (X_{13}) \chi  = \chi (\Omega'(X_{23}) \boxtimes \Omega'(X_{12}))
\end{equation}
By \eqref{IntertwiningOmegaprime}, $\Omega'(X_{ij}) = u_{ij}^\prime \circ u_{ij}^*$ is intertwining along $\omega_{\gamma_j}$ and $\omega_{\gamma_i}$.
By the commutative diagram \cref{Naturality} for $\chi$ (applied with $\theta_1 = \theta_2 = \theta_1^\prime = \theta_2^\prime = \id$ and $\varphi = \omega_{\gamma_2}$), we therefore obtain
\begin{equation*}
\begin{aligned}
\chi(\Omega'(X_{23}) \boxtimes \Omega'(X_{12})) = \Omega'(X_{23}) L^2(\omega_{\gamma_2})^* \Omega'(X_{12}) \chi = \Omega'(X_{23} \,i(\gamma_2)^{-1} X_{12}) \chi,
\end{aligned}
\end{equation*}
where we used \eqref{OmegaAndI} in the second step.
The desired identity \eqref{AimToShow} now follows from \eqref{X13FromX23AndX12}.
It remains to prove that the associativity condition \cref{associativity-of-fusion-product} holds -- this follow from the associativity of $\chi$ and of Connes fusion is carried out as \cite[Prop. 5.3.3]{kristel2020smooth}.
\end{proof}

We are now in position to give a complete definition of  Stolz-Teichner's stringor bundle as a 2-Hilbert bundle. 

\begin{definition}
\label{StringBundle}
Let $M$ be a spin manifold equipped with a fusive loop-spin structure $\widetilde{L\Spin}(M)$, and let $x\in M$. 
The \emph{stringor bundle} $\mathscr{S}(\widetilde{L\Spin}(M))$ of $M$ (relative to the base point $x$) is the 2-Hilbert bundle over $M$ with
\begin{itemize}
\item 
locally split map $\ev_\pi: P_xM \to M$,

\item
the restriction of the von Neumann algebra bundle $\mathcal{A}$ defined in \cref{eq:A} to $P_xM \subset PM$;

\item
the restriction of the bimodule bundle $\cup^{*}\mathcal{S}_{LM}$ defined in \cref{DefinitionSpinorBundle} to $P_xM^{[2]} \subset PM^{[2]}$;

\item
the restriction of the fusion product $\Upsilon$ of \cref{unique-fusion-product} to $P_xM^{[3]}\subset PM^{[3]}$. 
\end{itemize}
\end{definition}

We may sketch this 2-Hilbert bundle as follows:
\begin{equation*}
\label{sketch-stringor-bundle}
        \mathscr{S}(\widetilde{L\Spin}(M)) = \left\{\alxydim{}{
                        \mc{A} \ar[d] & \cup^{*}\mathcal{S}_{LM} \ar[d] & \Upsilon \ar@{.}[d] \\
                        P_xM \ar[d] & P_xM^{[2]} \arr[l] & P_xM^{[3]} \arrr[l]  \\
                        M & &
                }
        \right\}
\end{equation*}

\begin{remark}
A somewhat more general construction  is carried out in \cite[\S2.5]{ludewig2023spinor}, taking as input an arbitrary spinor bundle $\mathcal{S}$ on the loop space $LM$, defined as a certain irreducible left module bundle for the Clifford von Neumann algebra bundle on the loop space (see \cite[Definition~1.4]{ludewig2023spinor} and \cite{LudewigClifford}).
Given the input of a loop-spin structure $\widetilde{L\Spin}(M)$, the bundle $\mathcal{S}_{LM}$ from \cref{DefinitionSpinorBundle} is an example for such a  spinor bundle.
We remark that in \cite{ludewig2023spinor}, the map $\cup$ of  \eqref{CupMap} is replaced by an operation $\fuse$ arising from exchanging the two factors.
The use of $\cup$ here entails the conjugation by $J$ present in the representation $\Omega^\prime$ used in \eqref{DefinitionSpinorBundle}.
\end{remark}

In the remainder of this section we compare the  stringor bundle defined above with  \cite{Kristel2019c}, which represents the -- up to this point -- most complete  construction of the stringor bundle.
As mentioned above,  \cite{Kristel2019c} works in a smooth setting of rigged von Neumann algebra bundles and rigged bimodule bundles, and additionally treats loop spaces and path spaces in the setting of diffeological spaces. More precisely, \cite{Kristel2019c} provides the following structure:
 \begin{itemize}

\item 
a rigged von Neumann algebra bundle $\mathcal{A}\smooth$ over the diffeological space  $P_{si}M$ of paths with sitting instants in $M$.  

\item
a rigged Hilbert bundle $\mathcal{S}_{LM}\smooth$ over $LM$, the \emph{smooth spinor bundle on loop space}. Its pullback along the map  $\cup_{si}: P_{si}M^{[2]} \to LM$  is a rigged von Neumann $\pr_1^{*}\mathcal{A}\smooth$-$\pr_2^{*}\mathcal{A}\smooth$-bimodule bundle.

\item
a fusion product: a fibrewise defined intertwiner of $\mathcal{A}\smooth_{\beta_1}$-$\mathcal{A}\smooth_{\beta_2}$-bimodules
\begin{equation*}
\Upsilon\smooth_{\beta_1,\beta_2,\beta_3}: \mathcal{S}\smooth_{LM}|_{\beta_1\cup\beta_2} \boxtimes \mathcal{S}_{LM}\smooth|_{\beta_2\cup\beta_3} \to \mathcal{S}\smooth_{LM}|_{\beta_1\cup\beta_3}
\end{equation*}
 for each $(\beta_1,\beta_2,\beta_3) \in P_{si}M^{[3]}$. Here, the fibres of the rigged bundles $\mathcal{A}\smooth$ and $\mathcal{S}_{LM}\smooth$ are completed to actual von Neumann algebras and von Neumann bimodules, respectively, and $\boxtimes$ is Connes fusion.  Moreover, the intertwiners $\Upsilon\smooth_{\beta_1,\beta_2,\beta_3}$ are smooth in a certain sense and satisfy an associativity condition over $P_{si}M^{[4]}$ \cite[Prop. 5.3.3]{Kristel2019}.

\end{itemize}
As noticed in \cite[\S 5.4]{Kristel2019c}, above structure is already close to a 2-vector bundle. The only caveat is that in \cite{Kristel2019c}  we have not been able to lift Connes fusion (in the domain of the intertwiners $\Upsilon\smooth_{\beta_1,\beta_2,\beta_3}$) to a rigged setting, and thus were not able to claim that these intertwiners yield a \emph{smooth} homomorphism between rigged Hilbert bundles.

Comparing this with our present version of the stringor bundle comprises three issues: 
the first is to compare the \emph{rigged}  with the \emph{continuous} setting, the second issue is to compare the \emph{diffeological}   with the \emph{manifold} setting, and the third issue is that (in order to meet the conventions we fixed beforehand for 2-vector bundles) the ordering of factors in the fusion product is here opposite to the one of \cite{Kristel2019c}.

Concerning the first issue, we describe in 
\cref{sec:riggedtocontinuous} a general procedure how to complete \emph{rigged} von Neumann algebra bundles $\mathcal{D}$ to \emph{continuous} von Neumann algebra bundles $\mathcal{D}''$ (\cref{prop:fromriggedtocontinuous}), and to complete rigged bimodule bundles $\mathcal{E}$ into continuous bimodule bundles $\widehat{\mathcal{E}}$ (\cref{prop:riggedtocontinuousbimodulebundle}). Concerning the second issue, we recall that Fr\'echet manifolds embed fully faithfully into diffeological spaces, and we note that we have an inclusion $i: P_{si}M \to PM$ from the diffeological  space of paths with sitting instants as in  \cite{Kristel2019c} to the Fr\'echet manifold $PM$ used here. 

The rigged von Neumann algebra bundle $\mathcal{A}\smooth$ over $P_{si}M$ was obtained in \cite[\S 5.1]{Kristel2019c} by pullback along the diagonal map $\Delta_{si}: P_{si}M \to LM$ from a rigged von Neumann algebra bundle $\mathcal{A}\smooth_{LM}$ whose definition is recalled in \cref{ex:riggedcliffordalgebrabundle}, i.e.,  $\mathcal{A}\smooth := \Delta_{si}^{*}\mathcal{A}\smooth_{LM}$. In \cref{ex:cliffordalgebrabundle} we show that $(\mathcal{A}\smooth_{LM})''\cong L\Spin(M) \times_{L\Spin(d)} A$ as von Neumann algebra bundles over $LM$. We have $\Delta^{*}(L\Spin(M) \times_{L\Spin(d)} A)=\mathcal{A}$, the von Neumann algebra bundle defined in \cref{eq:A}. Since $\Delta \circ i = \Delta_{si}$, this shows that we have a canonical isomorphism
\begin{equation}
\label{iso-between-algebra-bundles}
(\mathcal{A}\smooth)'' \cong i^{*}\mathcal{A}
\end{equation}
of continuous von Neumann algebra bundles over $P_{si}M$,  establishing the claimed relation. We remark that the elements on both sides can be represented by pairs $(\gamma,a)$ where $\gamma\in P_{si}\Spin(M)$ and $a\in A$, and that the isomorphism of \cref{iso-between-algebra-bundles} is induced by the identity map on these pairs.

The smooth spinor bundle $\mathcal{S}_{LM}\smooth$ of \cite{Kristel2019c} and the continuous spinor bundle $\mathcal{S}_{LM}$ defined in \cref{DefinitionSpinorBundle} are related by an isomorphism
\begin{equation}
\label{iso-between-spinor-bundles}
\widehat{\mathcal{S}_{LM}\smooth} \cong s^{*}\mathcal{S}_{LM}
\end{equation}  
of Hilbert bundles over $LM$, described in \cref{ex:spinorbundleonloopspace}, where $s:LM \to LM$ is induced by the complex conjugation on $S^1$. 
The elements on both sides can be represented by pairs $(\Phi,v)$ where $\Phi\in \widetilde{L\Spin}(M)$ and $v\in L^2(A)$ and the isomorphism \cref{iso-between-spinor-bundles} is induced by the map $(\Phi,v) \mapsto (\tilde s(\Phi),v)$ on these pairs, where $\tilde s$ lifts $s$ to $\widetilde{L\Spin}(M)$.   
It remains to compare the bimodule structure, for which we first have to address the third issue mentioned above. We let $s_2: P_{si}M^{[2]} \to P_{si}M^{[2]}$ be the swap map, i.e., $s_2(\gamma_1,\gamma_2)=(\gamma_2,\gamma_1)$. We note that $\cup_{si} \circ s_2 = s \circ \cup_{si}$, so that \cref{iso-between-spinor-bundles} becomes an isomorphism
\begin{equation}
\label{iso-between-spinor-bundles2}
\cup_{si}^{*}\widehat{\mathcal{S}_{LM}\smooth} \cong s_2^{*}\cup_{si}^{*}\mathcal{S}_{LM}
\end{equation}  
of Hilbert bundles over $P_{si}M^{[2]}$. By \cref{prop:riggedtocontinuousbimodulebundle} the left hand side is a   $\pr_1^{*}(\mathcal{A}\smooth)'' $-$\pr_2^{*}(\mathcal{A}\smooth)'' $-bimodule bundle.   
We note that under the algebra isomorphisms \cref{iso-between-algebra-bundles} the right hand side also becomes a  $\pr_1^{*}(\mathcal{A}\smooth)'' $-$\pr_2^{*}(\mathcal{A}\smooth)'' $-bimodule bundle, and we claim that the isomorphism \cref{iso-between-spinor-bundles2} is indeed an isomorphism
of bimodule bundles. 
This is proved by observing that  the formulae \cref{LeftRightActionFLM} for left and right actions are precisely those for the right and left actions defined in \cite[Lemma 5.2.1]{Kristel2019c}. 

Finally, we claim that the fusion products coincide fibre-wise, i.e., that
\begin{equation}
\label{upsilon-and-chi}
\Upsilon|_{\beta_3,\beta_2,\beta_1} = \Upsilon\smooth_{\beta_1,\beta_2,\beta_3}
\end{equation}
for all $(\beta_1,\beta_2,\beta_3)\in P_{si}M^{[3]}$. This follows from the fact that both $\Upsilon$ and $\Upsilon\smooth$ are characterized uniquely by the same property, see \cref{unique-fusion-product} and \cite[Thm. 5.3.1]{Kristel2019c}.
We remark that \cref{upsilon-and-chi} shows, in particular, that the fibrewise defined intertwiners $\Upsilon\smooth_{\beta_1,\beta_2,\beta_3}$ form a continuous morphism between bimodule bundles.

\subsection{The stringor bundle is an associated bundle}

\label{sec:identification}

In this section, we prove the main result of this article:

\begin{theorem}
\label{th:main}
Let $M$ be a spin manifold equipped with a fusive loop-spin structure $\widetilde{L\Spin}(M)$. Let $\String(M)$ be the corresponding string structure constructed in \cref{SubsectionStringStructuresFusiveLoopSpin}. Then, there is a canonical isomorphism
        \begin{equation*}
                \stringgerbe(M) \times_{\String(d)} A \cong \mathscr{S}(\widetilde{L\Spin}(M))
        \end{equation*}
        of 2-Hilbert bundles over $M$,  between the 2-Hilbert bundle associated with $\stringgerbe(M)$ and the stringor representation   $\Rep:\String(d) \to \mathcal{U}(A)$  and the Stolz-Teichner stringor bundle. 
\end{theorem}

We start by spelling out the details of the associated 2-Hilbert bundle $\stringgerbe(M) \times_{\String(d)} A$ on the basis of \cref{SectionAssociatedBimoduleBundles}, but now using the explicit form of the string structure $\String(M)$ from \cref{SubsectionStringStructuresFusiveLoopSpin}.
This 2-Hilbert bundle consists of the locally split map $\ev_\pi: P_{\hat x}\Spin(M) \to M$, and the trivial von Neumann algebra bundle $\underline{A}=P_{\hat x}\Spin(M) \times A$ with typical fibre $A$.
Over $P_{\hat x}\Spin(M)^{[2]}$, we have the $\underline{A}$-$\underline{A}$-bimodule bundle 
\begin{equation*}
\Mod_\Rep(P) = (P \times L^2(A)) / \String(d)_s,
\end{equation*}
where $P$ is the principal $\String(d)_s$-bundle \cref{EqBundleQ}. 
 Hence the elements of $\Mod_\Rep(P)$ are represented by pairs $(p,\xi) \in P \times L^2(A)$, subject to the equivalence relation
\begin{equation}
\remember{SecondEquivalenceRelation}{(p,\xi) \sim \bigl(p \cdot (\gamma',e,X), \xi \ract \Rep_1(X)\bigr)}
\end{equation}
for any $p \in P$ and $(\gamma',e, X) \in \String(d)_s$. 
The bimodule actions are given by
\begin{equation} \label{LeftRightActionsAssociatedBundle}
\begin{aligned}
        a \lact (p ,\xi) \ract b &= (p ,a \lact \xi \ract \omega_{\gamma}(b)) \qquad \text{for} \quad p = (\beta_1,\beta_2,\Phi,\gamma) \text{.}
\end{aligned}
\end{equation}
Finally, the intertwiner over $P_{\hat x}\Spin(M)^{[3]}$ is the morphism $\Mod_{\Rep}(\mu)$ defined in \cref{DefinitionModTilde}.
The whole structure may be sketched as follows:
\begin{equation*}
\stringgerbe(M) \times_{\String(d)} A = \left\{\alxydim{}{
                        \underline{A} \ar[d] & \Mod_\Rep(P) \ar[d] & \Mod_\Rep(\mu) \ar@{.}[d] \\
                        P_{\hat x} \Spin(M) \ar[d] & P_{\hat x} \Spin(M)^{[2]} \arr[l] & P_{\hat x} \Spin(M)^{[3]} \arrr[l]  \\
                        M & &
                }
        \right\}
\end{equation*}

The isomorphism in \cref{th:main} is constructed as a refinement (see \cref{def:refinement}) from the associated 2-Hilbert bundle $\stringgerbe(M) \times_{\String(d)} A$ to the Stolz-Teichner stringor bundle (see \cref{sketch-stringor-bundle}), depicted as follows:
\begin{equation*}
\alxydim{@C=2em}{\underline{A} \ar[drr]^\varphi \ar[dd] &&& \Mod_\Rep(P) \ar[dd]\ar[drr]^u\\
 && \mc{A} \ar[dd] &&&\cup^{*}\mathcal{S}_{LM} \ar[dd]  \\
P_{\hat x}\Spin(M) \ar[drr]^\rho \ar[rdd]&& & P_{\hat x}\Spin(M)^{[2]} \arr[lll] \ar[drr]^{\rho^{[2]}}  
 \\
&&    P_xM \ar[dl] &&& P_xM^{[2]} \arr[lll]  
               \\
                   & M
}
\end{equation*}
 The first ingredient is the \quot{foot point} projection
\begin{equation*}
        \rho:P_{\hat x}\Spin(M) \to P_xM,
        \qquad \rho(\beta)(t) := r(\beta(t))\text{,}
\end{equation*}
where $r: \Spin(M) \to M$ is the bundle projection, going between the domains of the locally split maps of the two 2-Hilbert bundles.
The map $\rho$ is covered by the map
\begin{equation}
\label{AlgebraIsoOfIso}
\varphi: \underline{A} \to \mathcal{A}, \qquad (\beta, a) \mapsto [\beta, a],
\end{equation}
which yields  an isomorphism $\varphi:\underline{A} \to \rho^{*}\mathcal{A}$ of von Neumann algebra bundles over $P_{\hat x}\Spin(M)$.

The second ingredient is the map 
\begin{equation*}
u: \Mod_{\mathcal{R}}(P) \to \mathcal{S}_{LM}, \qquad [\beta_1,\beta_2,\Phi,\gamma,\xi] \mapsto[\Phi,L^2(\omega_{\gamma})^{*}\xi],
\end{equation*}
where $\omega_{\gamma} \in \Aut(A)$ is the automorphism obtained from the map $\omega$ in \cref{KleinOmega} and $L^2(\omega_{\gamma})$ is its canonical implementation \cref{CanonicalImplementation}.

\begin{lemma}
The map $u$ induces a well-defined morphism  $u: \Mod_{\mathcal{R}}(P) \to (\rho^{[2]})^{*}\cup^{*}\mathcal{S}_{LM}$ of bimodule bundles over $P_{\hat x}\Spin(M)^{[2]}$, intertwining along $\mathrm{pr}_2^*\varphi$ and $\mathrm{pr}_1^*\varphi$.
\end{lemma}

\begin{proof}
To show well-definedness, we need to show that $u$ is compatible with the equivalence relation \cref{SecondEquivalenceRelation}, using the principal $\String(d)_s$-action on $P$ given in \cref{EqPrincipalActionQ}. 
%
%
Here, for $p = (\beta_1, \beta_2, \Phi, \gamma) \in P$ and $( \gamma',e, X) \in \String(d)_s$, we calculate 
\begin{align*}
&\mquad u\bigl(\bigr[p \cdot (\gamma^\prime, e, X), \xi \ract \Omega(X)\bigr]\bigr)
\\&=
 u\bigl([\beta_1,\beta_2\gamma^{-1}\gamma'^{-1}\gamma,\Phi\cdot i(\gamma^{-1})\cdot X\cdot i(\gamma'^{-1}\gamma),\gamma'^{-1}\gamma,\xi \ract \Omega(X)]\bigr) && \text{(Action \cref{EqPrincipalActionQ})}  \\
        &=\bigl[\Phi\cdot i(\gamma)^{-1}\cdot X \cdot i(\gamma'^{-1}\gamma), L^2(\omega_{\gamma'^{-1}\gamma})^{*}(\xi \ract \Omega(X))\bigr] && \text{(Definition of $u$)} 
\\&=\bigl[\Phi, \Omega'\bigl( i(\gamma)^{-1}\cdot X \cdot i(\gamma'^{-1}\gamma)\bigr)L^2(\omega_{\gamma'^{-1}\gamma})^{*}(\xi \ract \Omega(X))\bigr] && \text{(Def.\ of $\mathcal{S}_{LM}$, \cref{DefinitionSpinorBundle})}\mquad 
\\&=\bigl[\Phi, J\Omega\bigl( i(\gamma)^{-1}\cdot X \cdot i(\gamma'^{-1}\gamma)\bigr)JL^2(\omega_{\gamma'^{-1}\gamma})^{*}(\xi \ract \Omega(X))\bigr] && \text{(Def.\ of $\Omega'$, \cref{representation-on-standard-bimodule})}\mquad 
\\&=\bigl[\Phi, J\Omega\bigl( i(\gamma)^{-1}\cdot X \cdot i(\gamma'^{-1}\gamma)\bigr)L^2(\omega_{\gamma'^{-1}\gamma})^{*}J(\xi \ract \Omega(X))\bigr] && \text{($J$ and $L^2(\omega_{\gamma})$ commute)}\mquad 
\\&=\bigl[\Phi, JL^2(\omega_\gamma)^{*} \Omega(X)J(\xi \ract \Omega(X))\bigr] && \text{(Relation  \cref{OmegaAndI})} 
\\&=\bigl[\Phi, JL^2(\omega_\gamma)^{*} \Omega(X)JJ \Omega(X)^*J\xi\bigr] && \text{(Right action, \cref{RightActionModularConjugation})}\mquad         
\\&=[\Phi, L^2(\omega_\gamma)^{*} \xi\bigr] && \text{($J$ and $L^2(\omega_{\gamma})$ commute)}\mquad         
\\&= u\bigl([p, \xi]\bigr)  && \text{(Definition of $u$)}
\end{align*}
Now, $u$ is the quotient map of the continuous map
\begin{equation}
\label{eq:coveringmapUpsilon}
P \times L^2(A) \to \widetilde{L\Spin}(M) \times L^2(A), \qquad (\beta_{1},\beta_{2},\Phi,\gamma,\xi) \mapsto (\Phi, L^2(\omega_{\gamma})^{*}\xi),
\end{equation}    
and hence is  continuous. 
It is also fibre-preserving, as 
\begin{equation*}
\rho^{[2]}(\pr_P(\beta_1,\beta_2,\Phi,\gamma))=\rho^{[2]}(\beta_1,\beta_2\gamma^{-1})=(\rho(\beta_1),\rho(\beta_2\gamma^{-1}))=(\rho(\beta_1),\rho(\beta_2))=\pr_{\mathcal{S}_{LM}}(\Phi,L^2(\omega_{\gamma})^{*}\xi)\text{,}
\end{equation*}
where $\pr_P$ is the  projection \cref{EqBundleProjectionQ} of the principal bundle $P$, and $\pr_{\mathcal{S}_{LM}}$ is  the projection of the spinor bundle on loop space.
To verify the intertwining property, we calculate
\begin{align*}
&\mquad u\bigl((\beta_2\gamma^{-1},a) \lact (\beta_1,\beta_2,\Phi,\gamma,\xi) \ract (\beta_1,b)\bigr)
\\&=u\bigl(\beta_1,\beta_2,\Phi,\gamma,a \lact \xi \ract \omega_\gamma(b) \bigr) && \text{(Actions  \cref{LeftRightActionsAssociatedBundle})} \\
             &= \bigl[\Phi,L^2(\omega_{\gamma})^{*} (a \lact \xi \ract \omega_{\gamma}(b))\bigr]  && \text{(Definition of $u$)}  
             \\&= \bigl[\Phi,\omega_{\gamma^{-1}}(a) \lact L^2(\omega_{\gamma})^{*} \xi \ract b\bigr]  && \text{($L^2(\omega_{\gamma})$  intertwines along $\omega_{\gamma}$)} \mquad 
             \\&=[\beta_2, \omega_{\gamma^{-1}}(a)] \lact [\Phi,L^2(\omega_{\gamma})^{*}\xi] \ract [\beta_1,b]) && \text{(Actions on $\mathcal{S}_{LM}$, \cref{LeftRightActionFLM}})\mquad              \\
             &=[\beta_2\gamma^{-1}, a] \lact [\Phi,L^2(\omega_{\gamma})^{*}\xi] \ract [\beta_1,b]) && \text{(Definition of $\mathcal{A}$, \cref{eq:A})}\mquad 
             \\&=\varphi(\beta_2\gamma^{-1}, a) \lact u(\beta_1,\beta_2,\Phi,\gamma,\xi) \ract \varphi(\beta_1,b)\text{,} && \text{(Definitions of $u$ and $\varphi$)} 
\end{align*}
which is the desired identity.
\end{proof}

So far, we have shown that $u$ is an intertwiner of bimodule bundles, and thus provided the structure of a refinement between 2-Hilbert bundles, see \cref{def:refinement}. 
It remains to check the compatibility with the intertwiners on triple fibre products, see \cref{eq:refinementdiag}. 

\begin{proposition}
\label{prop:fusion}
The following diagram over $P_{\hat x}\Spin(M)^{[3]}$ is commutative:
\begin{equation} \label{Endgegnerdiagram}
\xymatrix@C=6em{\mathrm{pr}_{23}^*\Mod_{\mathcal{R}}(P) \boxtimes \mathrm{pr}_{12}^*\Mod_{\mathcal{R}}(P) \ar[d]_{\mathrm{pr}_{23}^*u \boxtimes \mathrm{pr}_{12}^*u} \ar[r]^-{\Mod_{\mathcal{R}}(\mu)} & \mathrm{pr}_{13}^*\Mod_{\mathcal{R}}(P) \ar[d]^{\mathrm{pr}_{13}^*u}\\   (\rho^{[3]})^*(\mathrm{pr}_{23}^*\cup^{*}\mathcal{S}_{LM} \boxtimes \mathrm{pr}_{12}^*\cup^{*}\mathcal{S}_{LM}) \ar[r]_-{(\rho^{[3]})^*\Upsilon} & (\rho^{[3]})^*\mathrm{pr}_{13}^*\cup^{*}\mathcal{S}_{LM}}
\end{equation}
\end{proposition}

\begin{proof}
We will check the commutativity of  diagram \cref{Endgegnerdiagram} fibrewise in the fibre over  a point $(\hat\beta_1, \hat\beta_2, \hat\beta_3) \in P_{\hat x} \Spin(M)^{[3]}$ (recall that the fibre product is taken along the end-point evaluation to $M$, so that the end points of the paths  $\hat\beta_i$ do not necessarily coincide but lie in the same fibre of $\Spin(M)$).
Write $\beta_i := \rho(\hat\beta_i) \in P_xM$ for their foot point curves. 
We make the following choices:
\begin{itemize}
\item
Let $(\hat\beta'_1,\hat\beta'_2,\hat\beta'_3)\in P_{\hat x}\Spin(M) \times_{\Spin(M)} P_{\hat x}\Spin(M) \times_{\Spin(M)} P_{\hat x}\Spin(M)$ be other lifts of $\beta_i$. That is, $\rho(\hat\beta'_i)=\beta_i$, and now the paths  $\hat\beta'_1, \hat\beta'_2, \hat\beta'_3$ have common start and end point in $\Spin(M)$, and  the common start point of the $\hat\beta'_i$ is ${\hat x}$. 
Hence there exist paths $\gamma_i\in P_e\Spin(d)$ such that 
\begin{equation*}
\hat\beta'_i=\hat\beta_i\gamma_i \qquad i=1, 2, 3.
\end{equation*}
\item
Let moreover $\Phi_{12},\Phi_{23}\in\widetilde{L\Spin}(M)$ be lifts of $\hat\beta'_1\cup\hat\beta'_2$ and $\hat\beta'_2 \cup \hat\beta'_3$,  respectively, and set $\Phi_{13} = \lambda(\Phi_{23} \otimes \Phi_{12})$, which lifts $\hat\beta'_1 \cup \hat\beta'_3$.
\end{itemize}
We obtain corresponding $\ast$-isomorphisms $\psi_i : \mathcal{A}_{\beta_i} \to A$, $[\hat\beta'_i, a] \mapsto a$, and unitary intertwiners 
\begin{equation*}
u_{ij}: (\mathcal{S}_{LM})_{\beta_i \cup \beta_j}\to L^2(A), \qquad [\Phi_{ij},\xi] \mapsto \xi,
\end{equation*}
along $\psi_j$ and $\psi_i$, respectively. 
On the other side, setting $\gamma_{ij} = \gamma_j \gamma_i^{-1}$, we consider the elements 
\begin{align*}
        p_{ij} :=&~ (\hat\beta'_i \gamma_i^{-1},\hat\beta'_j \gamma_i^{-1},\Phi_{ij}\cdot i(\gamma_i)^{-1},\gamma_j\gamma_i^{-1}) = (\hat\beta_i,\hat\beta_j \gamma_{ij},\Phi_{ij}\cdot i(\gamma_i)^{-1},\gamma_{ij})
\end{align*}
of $P$ over $(\hat\beta'_i\gamma_i^{-1},\hat\beta'_j\gamma_i^{-1}\gamma_{ij}^{-1})=(\hat\beta_i,\hat\beta_j)$. 
By \cref{eq:intN} these define unitary intertwiners
\begin{equation*}
        \tau_{p_{ij}}:\Mod_\Rep(P)_{\hat\beta_i\cup\hat\beta_j} \to L^2(A)_{\omega_{\gamma_{ij}}}, \qquad [p_{ij}, \xi] \mapsto \xi.
\end{equation*} 
These make the diagram
\begin{equation}
\label{RightMostDiagram}
                \alxydim{@C=6em}{
              \Mod_{\Rep}(P)_{\hat\beta_i\cup\hat\beta_j} \ar[r]^-{\tau_{p_{ij}}}  \ar[d]_{u_{\hat\beta_i, \hat\beta_j}} &  L^2(A)_{\omega_{\gamma_{ij}}} \\
              (\mathcal{S}_{LM})_{\beta_i\cup\beta_j} \ar[r]_-{u_{ij}}   & L^2(A) \ar[u]_{L^2(\omega_{\gamma_j})} 
                }
\end{equation}
commute as follows from the calculation
\begin{align*}
(u_{ij}\circ u)([p_{ij},\xi])
&= (u_{ij}\circ u)\bigl(\bigl[\hat\beta_i,\hat\beta_j\gamma_{ij},\Phi_{ij}\cdot i(\gamma_i)^{-1},\gamma_{ij}, \xi\bigr]\bigr)
\\&= u_{ij}\bigl(\bigl[\Phi_{ij}\cdot i(\gamma_i)^{-1}, L^2(\omega_{\gamma_{ij}})^* \xi\bigr]\bigr)
\\&= u_{ij}\bigl(\bigl[\Phi_{ij} , L^2(\omega_{\gamma_i})^*L^2(\omega_{\gamma_{ij}})^* \xi\bigr]\bigr)
\\&= u_{ij}\bigl(\bigl[\Phi_{ij} ,L^2(\omega_{\gamma_j})^{*}\xi\bigr]\bigr)
\\&= L^2(\omega_{\gamma_j})^{*}\xi  
\\&= L^2(\omega_{\gamma_j})^*\tau_{p_{ij}}([p_{ij}, \xi]).
\end{align*}
Now we consider the diagram
        \begin{equation*}
                \alxydim{@C=1.8em}{
                \Mod_\Rep(P)_{\hat\beta_2\cup\hat\beta_3} \boxtimes \Mod_\Rep(P)_{\hat\beta_1\cup\hat\beta_2} \ar[dr]_{\tau_{p_{23}}\boxtimes \tau_{p_{12}}} \ar[r]\ar[ddd]_{u_{\hat\beta_2 \hat\beta_3} \boxtimes u_{\hat\beta_1 \hat\beta_2}} & 
                \Mod_\Rep(P_{\hat\beta_2\cup\hat\beta_3} \otimes P_{\hat\beta_1\cup\hat\beta_2}) \ar[dr]_{\tau_{p_{23} \otimes p_{12}}}\ar[rr]^{\Mod_\Rep(\mu)} && 
                                \Mod_\Rep(P)_{\hat\beta_1\cup\hat\beta_3} \ar[dl]^{\tau_{p_{13}}} \ar[ddd]^{u_{\hat\beta_1\hat\beta_3}} \\ 
                & \mquad L^2(A)_{\omega_{\gamma_{23}}}\boxtimes L^2(A)_{\omega_{\gamma_{12}}}  \ar[r]_-{\chi_{\omega_{\gamma_{23}}}} & L^2(A)_{\omega_{\gamma_{13}}}  \\ 
& \mquad L^2(A) \boxtimes L^2(A) \ar[u]^{L^2(\omega_{\gamma_3})\boxtimes L^2(\omega_{\gamma_2})} \ar[r]_-{\chi} & L^2(A)\ar[u]_{L^2(\omega_{\gamma_3})} \\ 
(\mathcal{S}_{LM})_{\beta_2\cup\beta_3} \boxtimes (\mathcal{S}_{LM})_{\beta_1\cup\beta_2} \ar[ur]^-{u_{23} \boxtimes u_{12}} \ar[rrr]_-{\Upsilon_{\beta_1\beta_2\beta_3}} &&& \ar[ul]_{u_{13}} (\mathcal{S}_{LM})_{\beta_1\cup\beta_3}.}
        \end{equation*}
 The four-sided diagram at the bottom commutes by definition.
The central square is a special case of the commutative diagram \cref{Naturality}.
 The four-sided diagrams on the right and on the left are copies of \cref{RightMostDiagram}; for commutativity of the left one, we also use functoriality of Connes fusion.
 The top left diagram is a copy of the diagram \cref{DefiningDiagramMonoidalityMorphism}, which is commutative by monoidality of $\Mod_\Rep$.
 The triangle on the top right is a copy of the commutative diagram \cref{FunctorModProperty}.
 Hence the whole diagram is commutative; this shows the claim.
\end{proof}

\begin{appendix}

\setsecnumdepth{1}

\section{2-group bundles and non-abelian bundle gerbes}

\label{sec:convenientgeometry}
\label{sec:principalbundles}

If $G$ is a topological group, we denote by $\Bdl G$ the stack of principal $G$-bundles over the site $\top$ of topological spaces. The Grothendieck topology on $\top$ is the one generated by locally split maps, i.e., by maps $\pi:Y \to X$ such that each $x\in X$ has an open neighborhood $U\subset X$ with a section $U \to Y$. This Grothendieck topology coincides with the one generated by open covers. We recall that a continuous group homomorphism $f:G \to H$ induces a morphism
\begin{equation}
\label{eq:bundleextension}
f_{*}:\Bdl G \to \Bdl H
\end{equation}
of stacks, called \emph{bundle extension}. In short, $f_{*}(P) := (P \times H)/G$, where $G$ acts on $P\times H$ by $(p,h)\cdot g := (pg,f(g)^{-1}h)$. 

Next we upgrade from principal bundles for ordinary groups  to principal bundles for 2-groups. We emphasize that these are \emph{not} categorified principal bundles, instead, they are ordinary bundles for categorified groups.

\begin{definition}
\label{def:principalgammabundle}
Let $\mathcal{G}$ be a topological strict 2-group, and let $\mathcal{G}_s := \mathrm{ker}(s) \subset \mathcal{G}$ be the subgroup that belongs to the crossed module of $\mathcal{G}$.
A \emph{principal $\mathcal{G}$-bundle} over a topological space $X$ is a principal $\mathcal{G}_s$-bundle $P$ over $X$ together with a $\mathcal{G}_s$-anti-equivariant continuous map  $\phi:P \to \mathcal{G}_0$ called \emph{anchor}. A morphism between principal $\mathcal{G}$-bundles is a principal bundle morphism that preserves the anchors. 
\end{definition}

The anti-equivariance of the anchor means that
\begin{equation}
\label{AnchorPropertyCrossedModuleBundle}
\phi(ph) = t(h)^{-1}\phi(p)
\end{equation}
holds for all $p\in P$ and $h\in \mathcal{G}_s$. 
The main point of principal $\mathcal{G}$-bundles is that their category $\Bdl{\mathcal{G}}(X)$ is monoidal, in contrast to the category $\Bdl{\mathcal{G}_s}(X)$ of ordinary principal $\mathcal{G}_s$-bundles.    We recall this now. 
If $P_1$ and $P_2$ are principal ${\mathcal{G}}$-bundles over $X$ with anchors $\phi_1$ and $\phi_2$, respectively, then  the fibre product $P_1 \times_X P_2$ carries a left ${\mathcal{G}}_s$-action defined
by 
\begin{equation}
\label{eq:MonoidalStructureQuotient}
h\cdot (p_1,p_2) := (p_1\cdot h^{-1},p_2\cdot \alpha(\phi_1(p_1)^{-1},h)))\text{,}
\end{equation}
where $\alpha$ is the action of $g\in \mathcal{G}_0$ on $\mathcal{G}_s$  in the crossed module of $\mathcal{G}$, i.e., it is conjugation by $i(g)$.
We define $P_1 \otimes P_2$ as the quotient by this action and denote the equivalence classes by $p_1 \otimes p_2$. 
The quotient $P_1 \otimes P_2$ comes  equipped with the right ${\mathcal{G}}_s$-action $(p_1 \otimes p_2) \cdot h := p_1 h \otimes p_2$, the obvious projection map $P_1 \otimes P_2 \to X$, and the anchor  
\begin{equation} \label{DefinitionAnchorTensorProduct}
\phi_1 \otimes \phi_2 : p_1\otimes p_2 \mapsto \phi_1(p_1)\phi_2(p_2).
\end{equation} 
That this construction result in a principal ${\mathcal{G}}$-bundle is shown in \cite[\S 2.4]{Nikolaus}.

It is clear that one can pull back principal ${\mathcal{G}}$-bundles along  continuous maps, and that the tensor product is compatible with such pullbacks. It is then straightforward to show that principal ${\mathcal{G}}$-bundles form a monoidal stack $\Bdl{\mathcal{G}}$ over the site $\top$ (w.r.t. to open covers).

Suppose $F: {\mathcal{G}} \to {\mathcal{H}}$ is a continuous homomorphism of topological strict 2-groups. Let $f:{\mathcal{G}}_s \to {\mathcal{H}}_s$ be the restriction, which is a continuous group homomorphism. 
Using bundle extension \cref{eq:bundleextension} a principal ${\mathcal{G}}$-bundle $P$ with anchor $\phi$ becomes a principal ${\mathcal{H}}$-bundle $f_{*}(P)$ with anchor $[p,h]\mapsto t(h)^{-1}F(\phi(p))$. 
This defines a morphism of stacks
\begin{equation}
\label{eq:Principal2GroupBundleExtension}
F_{*}: \Bdl{\mathcal{G}} \to \Bdl{\mathcal{H}}\text{.}
\end{equation}

\begin{lemma}
\label{lem:bundleextension}
The bundle extension $F_{*}: \Bdl{\mathcal{G}} \to \Bdl{\mathcal{H}}$ is a monoidal functor. 
\end{lemma}

\begin{proof}
We provide a bundle morphism $\psi_{P_1,P_2}:F_{*}(P_1) \otimes F_{*}(P_2) \to F_{*}(P_1 \otimes P_2)$. 
Let us first describe both sides. An element in $F_{*}(P_1) \otimes F_{*}(P_2)$ is represented by an element $((p_1,h_1) , (p_2,h_2))$. 
An element in $F_{*}(P_1 \otimes P_2)$ is represented by a pair $(p_1,p_2),h)$. 
We define $\psi_{P_1,P_2}$ by
\begin{equation*}
\psi_{P_1,P_2}((p_1,h_1),(p_2,h_2)) := ((p_1,p_2),\alpha(F(\phi_1(p_1)),h_2)h_1)\text{.}
\end{equation*}
It is straightforward to show that this preserves anchors and the ${\mathcal{H}}_s$-action, and a bit tedious but still straightforward to prove that $\psi_{P_1,P_2}$ is well-defined under the two layers of equivalence relations present on $F_{*}(P_1) \otimes F_{*}(P_2)$.
\end{proof}

The monoidal structure of 2-group bundles is the key ingredient for the definition of non-abelian gerbes. The following definition is \cite[\S 5]{Nikolaus}.

\begin{definition}
\label{def:bundlegerbe}
Let $X$ be a topological space, and let ${\mathcal{G}}$ be a topological strict 2-group. A \emph{${\mathcal{G}}$-bundle gerbe} $\mathcal{Q}$ over $X$ consists of the following structure:
\begin{enumerate}

\item 
a topological space $Y$ together with a locally split map $\pi:Y \to X$.

\item
a principal ${\mathcal{G}}$-bundle $P$ over the double fibre product $Y^{[2]}$, in the sense of \cref{def:principalgammabundle}. 

\item
a bundle morphism $\mu: \pr_{23}^{*}P \otimes \pr_{12}^{*}P \to \pr_{13}^{*}P$ over $Y^{[3]}$, called the \emph{bundle gerbe product} of $\mathcal{G}$.

\end{enumerate}
It is required that the usual associativity condition for bundle gerbe products over $Y^{[4]}$ is satisfied. 
\end{definition}

\cref{lem:bundleextension} implies the following.

\begin{corollary}
\label{InducedOperationsOnBundleGerbes}
\label{InducedOperationsOnBundleGerbes:Extension} 
Suppose $F:{\mathcal{G}} \to {\mathcal{H}}$ is a continuous  homomorphism between topological strict 2-groups, and $\mathcal{Q}=(Y,\pi,P,\mu)$ is a ${\mathcal{G}}$-bundle gerbe over $X$.
Then, $F_{*}(\mathcal{Q}) := (Y,\pi,F_{*}(P),F_{*}(\mu))$ is a ${\mathcal{H}}$-bundle gerbe over $X$.
\end{corollary}

\begin{remark}
\label{DiffeologicalPlusConstruction}
One may adapt Nikolaus-Schweigert's plus construction \cite{nikolaus2} to topological spaces.
This  exhibits above definition of a bundle gerbe as the objects of a bicategory
\begin{equation*}
\Grb_{\mathcal{G}} (X) = \mathscr{B}(\Bdl{\mathcal{G}})^{+}(X)\text{,}
\end{equation*}
where $\mathscr{B}\mathscr{C}$ denotes -- when $\mathscr{C}$ is a monoidal category -- the corresponding bicategory with a single object. 

\end{remark}

\setsecnumdepth 1

\section{Four equivalent versions of string structures}

\label{sec:strstr}

In the differential-geometric setting of the present article, there  are four equivalent versions of the notion of a string structure on a spin manifold $M$:
\begin{enumerate}[(1)]

        \item 
                A trivialization of the Chern-Simons 2-gerbe over $M$, see \cite{waldorf8}.

        \item
                A thin fusive loop-spin structure, see \cite{Waldorfa}.

        \item
                A lift of the spin frame bundle $\Spin(M)$ of $M$ to a principal $\String(d)$-2-bundle. 
                
                \item
                A lift of the spin frame bundle $\Spin(M)$ of $M$ to a $\String(d)$-bundle gerbe, see \cite{stevenson3,jurco1}.

\end{enumerate}
Versions (3) and (4) involve models of the string 2-group. 

The equivalence between versions (1) and (4) has been established in \cite[Theorem 7.9]{Nikolausa}. 
The equivalence between (1) and (2) has been established in \cite[Theorem A]{Waldorfb}. 
The equivalence between (3) and (4) comes from the  general equivalence between principal 2-bundles and bundle gerbes \cite[Section 7.1]{Nikolaus}.
In this section we work out explicitly the passage from version (2) to version (4), which is induced by the above mentioned equivalences. We need this explicit description because the stringor bundle of Stolz-Teichner (see \cref{sec:stringor}) is defined using version (2) while the associated 2-Hilbert bundle (see \cref{SubsectionStringStructuresFusiveLoopSpin}) is defined using version (4).

We suppose that we have a string structure in version (2), i.e., a fusive loop-spin structure on the spin manifold $M$, as defined in \cite[Definition 3.6]{Waldorfa} and recalled above in \cref{DefinitionFusiveSpinStructure}.
The passage to version (4) most naturally factors through version (3), so we shall first recall that setting. 

 Let $\String(d)$ be the smooth string 2-group of \cref{SectionStringorRep}.  A \emph{principal $\String(d)$-2-bundle} over $M$ consists of a (Fr\'echet) Lie groupoid $\mathcal{P}$, a smooth functor $\pi : \mathcal{P} \to M_{dis}$ that is a submersion on the level of objects, and a smooth right action $R:\mathcal{P} \times \String(d) \to \mathcal{P}$ that preserves $\pi$, such that the smooth
        functor
        \begin{equation*}
                (\mathrm{pr}_\mathcal{P}, R) : \mathcal{P} \times \String(d) \to  \mathcal{P} \times_M \mathcal{P}
        \end{equation*}
        is a weak equivalence, see \cite[Def. 6.1.1]{Nikolaus}. Here, by smooth right action we mean a smooth functor that strictly satisfies the axioms of a right action, and by weak equivalence we mean a smooth functor that is invertible by a smooth anafunctor, or bibundle. 
Now we are in position to explain   version (3) of a string structure on $M$.

\begin{definition}
\label{def:ssv3}
A \emph{lift of the spin frame bundle  $\Spin(M)$ of $M$ to $\String(d)$} is a principal $\String(d)$-2-bundle $\mathcal{P}$ over $M$ together with a smooth functor $\mathcal{P} \to \Spin(M)_{dis}$ that respects the projections to $M_{dis}$ and is strictly equivariant along the projection $\String(d) \to \Spin(d)_{dis}$. 
\end{definition}

Next we explain how to construct a lift of $\Spin(M)$ to $\String(d)$ from a fusive loop-spin structure. This construction is new and in fact simple and straightforward.
We start with  the total space, the Lie groupoid $\mathcal{P}$.
We need to choose a base point $x \in M$ and a lift ${\hat x}\in \Spin(M)$. 
We assume throughout that $M$ is connected; otherwise, our procedure can be applied to each connected component separately.
\begin{itemize}
        \item 
                The Fr\'echet manifold of objects of $\mathcal{P}$ is 
                        $\mathcal{P}_0 := P_{\hat x}\Spin(M)$, the paths in $\Spin(M)$ that start at ${\hat x}$ and are flat at the endpoints. 

       \item
                The Fr\'echet manifold of morphisms of $\mathcal{P}$ is 
                \begin{equation*}
                        \mathcal{P}_1 := P_{\hat x}\Spin(M)^{[2]} \times_{L\Spin(M)} \widetilde{L\Spin}(M)\text{.}
                \end{equation*}
                Here, the fibre product $P_{\hat x}\Spin(M)^{[2]}$ is the fibre product over $\Spin(M)$, and so contains pairs of paths in $\Spin(M)$ with starting point ${\hat x}$ and the same end point, and the map to $L\Spin(M)$ is the map $\cup$ defined in \cref{CupMap}.  
                In total, $\mathcal{P}_1$ consists of elements $(\beta_1, \beta_2, \Phi)$, where $\beta_1, \beta_2 \in P_{\hat x} \Spin(M)$, $\beta_1(\pi) = \beta_2(\pi)$ and $\Phi$ is a lift of $\beta_1 \cup \beta_2 \in L\Spin(M)$ to $\widetilde{L \Spin}(M)$.
        \item
                Source and target maps are $s(\beta_1,\beta_2,\Phi) := \beta_2$ and $t(\beta_1,\beta_2,\Phi) := \beta_1$. 
        \item
                Composition is the fusion product $\lambda$, see \cref{DefinitionFusiveSpinStructure}. More precisely, 
                \begin{equation*}
                        (\beta_3,\beta_2,\Phi_{23}) \circ (\beta_2,\beta_1,\Phi_{12}) := (\beta_3,\beta_1,\lambda(\Phi_{12} \otimes \Phi_{23}))\text{.}
                \end{equation*}
                %
        \item
                Identity morphisms are induced from the fusion product $\lambda$: $\id_\beta$ is the unique element such that $\lambda_{\beta,\beta,\beta}(\id_\beta \otimes \id_\beta)=\id_\beta$; see \cite[Prop. 3.1.1]{Waldorfc}.
\end{itemize}  
The bundle projection is given by $\mathcal{P}_0 \to M: \beta\mapsto \varpi(\beta(\pi))$ where $\varpi:\Spin(M) \to M$ is the bundle projection. 
This is a surjective submersion and extends to a smooth functor $\mathcal{P} \to M_{dis}$. 
Next, we define the principal action $R: \mathcal{P} \times \String(d) \to \mathcal{P}$.
On the level of objects, we put 
\begin{equation*}
        R_0(\beta,\gamma) := \beta\gamma.
\end{equation*}
This uses  pointwise the action of $\Spin(d)$ on $\Spin(M)$, keeping in mind the fact that the objects of $\mathcal{P}$ are paths in $\Spin(M)$ starting at ${\hat x}$, while the objects of $\String(d)$ are paths in $\Spin(d)$, starting at the neutral element.
On the level of morphisms, we put 
\begin{equation}
\label{2-principal-action}
        R_1((\beta_1,\beta_2,\Phi),(\gamma_1,\gamma_2,X)) := (\beta_1\gamma_1,\beta_2\gamma_2,\Phi X),
\end{equation}
using the principal action of $\widetilde{L\Spin}(d)$ on $\widetilde{L\Spin}(M)$. 
This clearly preserves source and target, and it respects the composition precisely due to \cref{eq:FusionIsEquivariant}.
Thus, we have defined a smooth functor, which obviously preserves the bundle projection and is a strict right action. It remains to check the following.

\begin{lemma}
        The functor $\tilde R:=(\pr_{\mathcal{P}},R): \mathcal{P} \times \String(d) \to \mathcal{P} \times_M \mathcal{P}$ is a weak equivalence.
\end{lemma}

\begin{proof}
        A well-known criterion to check for a weak equivalence is to check that the functor is \emph{smoothly essentially surjective} and \emph{smoothly fully faithful}.         The first means that the map 
        \begin{equation*}
                (s \times s)\circ \pr_2: (\mathcal{P}_0 \times \String(d)_0) \ttimes{\tilde R_0}{t\times t} (\mathcal{P}_1 \times_M \mathcal{P}_1) \to \mathcal{P}_0 \times_M \mathcal{P}_0
        \end{equation*}
        must be a surjective submersion. We shall  see that it is surjective in the first place, which means precisely that $\tilde R$ is essentially surjective in the classical sense. Given $(\beta_1,\beta_2)\in \mathcal{P}_0 \times_M \mathcal{P}_0$, i.e., $\varpi(\beta_1(\pi))=\varpi(\beta_2(\pi))$, we let $g\in \Spin(d)$ be the unique element such that $\beta_1(\pi)g=\beta_2(\pi)$. Since $\Spin(d)$ is connected, there exists $\gamma\in P_e\Spin(d)$ with $\gamma(\pi)=g$. Thus, $\beta_1\gamma$ and $\beta_2$ have the same initial point and the same end point, and hence yield a loop $\beta_1\gamma \cup \beta_2\in L\Spin(M)$. This loop admits a lift $\Phi\in \widetilde{L\Spin}(M)$, i.e., a morphism in $\mathcal{P}$ from $\beta_2$ to $\beta_1\gamma$. We see that 
        \begin{equation*}
                ((\beta_1,\gamma),(\id_{\beta_1},\Phi))\in (\mathcal{P}_0 \times \String(d)_0) \ttimes{\tilde R}{t\times t} (\mathcal{P}_1 \times_M \mathcal{P}_1)
        \end{equation*}
        is a well-defined preimage of $(\beta_1,\beta_2)$. It is clear that the choices of $\gamma$ and $\Phi$ can be attained in a locally smooth way, which then shows that $\tilde R$ is even \emph{smoothly} essentially surjective. 
       
        That $\tilde{R}$ is smoothly fully faithful means that the diagram
        \begin{equation*}
                \xymatrix@C=5em{\mathcal{P}_1 \times \String(d)_1 \ar[d]_{(s \times s,t \times t)} \ar[r]^-{\tilde R_1} & \mathcal{P}_1 \times_M \mathcal{P}_1\ar[d]^{(s \times s,t \times t)} \\ (\mathcal{P}_0 \times \String(d)_0) \times (\mathcal{P}_0 \times \String(d)_0) \ar[r]_-{\tilde R_0 \times \tilde R_0} &  (\mathcal{P}_0 \times_M \mathcal{P}_0) \times (\mathcal{P}_0 \times_M \mathcal{P}_0)}
        \end{equation*}
        is a pullback diagram. In order to prove this, let us assume that we have a cone, \ie a Fr\'echet manifold $N$ with smooth maps $f$, $g$ such that the diagram
        \begin{equation*}
                \xymatrix@C=5em{N \ar[d]_{g} \ar[r]^-{f} & \mathcal{P}_1 \times_M \mathcal{P}_1\ar[d]^{(s \times s,t \times t)} \\ (\mathcal{P}_0 \times \String(d)_0) \times (\mathcal{P}_0 \times \String(d)_0) \ar[r]_-{\tilde R_0 \times \tilde R_0} &  (\mathcal{P}_0 \times_M \mathcal{P}_0) \times (\mathcal{P}_0 \times_M \mathcal{P}_0)}
        \end{equation*}
        is commutative. We write $g=(\beta_2,\gamma_2,\beta_1,\gamma_1)$ and $f=(\Phi_1,\Phi_2)$. Commutativity then means that for all $x\in N$, $\Phi_1(x)$ is a morphism from $\beta_2(x)$ to $\beta_1(x)$ and $\Phi_2(x)$ is a morphism from $(\beta_2\gamma_2)(x)$ to $(\beta_1\gamma_1)(x)$. In other words, $\Phi_1(x)$ projects to $\beta_1(x)\cup \beta_2(x)\in L\Spin(M)$, and $\Phi_2(x)$ projects to $(\beta_1\gamma_1)(x)\cup(\beta_2\gamma_2)(x)\in L\Spin(M)$. Since both loops in $L\Spin(M)$ project to the same loop in $LM$, and $\widetilde{L\Spin}(M)$ is a principal $\widetilde{L\Spin}(d)$-bundle over $LM$, there exists a unique smooth map $X:N \to \widetilde{L\Spin}(d)$ such that $X(x)$ projects to $\gamma_1(x) \cup \gamma_2(x) \in L\Spin(d)$ and $\Phi_2(x)=\Phi_1(x)X(x)$. 
        This gives a smooth map
        \begin{equation*}
                N \to \mathcal{P}_1 \times \String(d)_1: x \mapsto (\Phi_1(x),X(x))\text{.}
        \end{equation*}
        It is easy to see that it is the unique map rendering the  required diagrams commutative. 
\end{proof}

Now we have constructed a principal $\String(d)$-2-bundle $\mathcal{P}$ over $M$. In order to have a string structure as in \cref{def:ssv3}, we have to show that it lifts $\Spin(M)$. For this purpose, we consider the functor
\begin{equation*}
        P:\mathcal{P} \to \Spin(M)_{dis}
\end{equation*}
given by $\beta \mapsto \beta(\pi)$ on the level of objects. Since morphisms in $\mathcal{P}$ between $\beta_1$ and $\beta_2$ exist only if $\beta_1(\pi)=\beta_2(\pi)$, this extends to a functor to $\Spin(M)_{dis}$. The projection to the base $M$ is clearly preserved. We recall that the projection $\String(d) \to \Spin(d)_{dis}$ is given by $\gamma\mapsto \gamma(\pi)$ on the level of objects. Thus, we see that $P$ is strictly equivariant under this projection. Summarizing, we have the following result.

\begin{proposition}
        Given a fusive loop-spin structure on $LM$, the Fr\'echet Lie groupoid $\mathcal{P}$ together with the action $R$ is a principal $\String(d)$-2-bundle over $M$, and it lifts the structure group of $\Spin(M)$ from $\Spin(d)$ to $\String(d)$.
\end{proposition}

Next we pass from version (3) to version (4), using the functor constructed in Section 7.1 of \cite{Nikolaus}.
We obtain from $\mathcal{P}$ the following $\String(d)$-bundle gerbe $\stringgerbe(M)$ (in the sense of \cref{def:bundlegerbe}):
\begin{itemize}
        \item 
                Its surjective submersion is $Y:=\mathcal{P}_0=P_{\hat x}\Spin(M) \to M$, $\beta \mapsto \varpi(\beta(1))$.
        \item 
                Over the double fibre product $Y \times_M Y$ it has the following $\String(d)$-principal bundle (in the sense of \cref{def:principalgammabundle}): 
                \begin{itemize}
                        \item 
                                Its total space is 
                                \begin{equation}
                                \label{AppEqBundleQ}
                                P:= \mathcal{P}_1 \times P_e\Spin(d)\text{,}
                                \end{equation}
                                and so the elements are quadruples $(\beta_1,\beta_2,\Phi,\gamma)$,
                                where $\beta_1, \beta_2 \in P_{\hat x} \Spin(M)$ with $\beta_1(\pi) = \beta_2(\pi)$, $\Phi$ is a lift of $\beta_1 \cup \beta_2 \in L\Spin(M)$ to $\widetilde{L \Spin}(M)$ and $\gamma \in P_e \Spin(d)$.
                        \item
                                The bundle projection is 
                                \begin{equation} \remember{AppEqBundleProjectionQ}
                                        {(\beta_1,\beta_2,\Phi,\gamma)\mapsto (\beta_1,\beta_2\gamma^{-1})}.
                                \end{equation}
                                %
                        \item
                                The anchor map $\phi:P \to \String(d)_0$ is 
                                \begin{equation} \remember{AppEqAnchorQ}
                                        {(\beta_1,\beta_2,\Phi,\gamma) \mapsto \gamma}.
                                \end{equation} 
\item
the principal $\String(d)_s$-action is 
\begin{equation}
 \remember{AppEqPrincipalActionQ}
 {(\beta_1,\beta_2,\Phi,\gamma)\cdot (\gamma',e,X) := (\beta_1,\beta_2\gamma^{-1}\gamma'^{-1}\gamma,\Phi\cdot i(\gamma^{-1})\cdot X \cdot i(\gamma'^{-1}\gamma),\gamma'^{-1}\gamma)}\text{.}
\end{equation}                                
This  requires some explanation, because \cite{Nikolaus} does not use principal ${\mathcal{G}}$-bundles as in \cref{def:principalgammabundle} but an equivalent formulation whose total space does not carry an action of  the group ${\mathcal{G}}_s$ but rather an action of the groupoid ${\mathcal{G}}$, see \cite[Def. 2.2.1]{Nikolaus}. This $\mathcal{G}$-action is  given by \cite[Eq. 7.1.1]{Nikolaus}, namely,
\begin{align}
 \label{EqPrincipalActionQOld}
(\beta_1,\beta_2,\Phi,\gamma) \circ (\gamma,\gamma',X) := (\beta_1,\beta_2\gamma^{-1}\gamma',\Phi\cdot \id_{\gamma^{-1}}\cdot X,\gamma')\text{.}
\end{align}

The equivalence between the two notions of principal ${\mathcal{G}}$-bundles is described in \cite[Lemma 2.2.9]{Nikolaus}. Under this equivalence, a ${\mathcal{G}}$-action is transformed into a ${\mathcal{G}}_s$-action via the formula
\begin{equation*}
p\cdot h := p \circ (h,t(h)^{-1}\phi(p))\text{.}
\end{equation*}                       
Here, $h\in {\mathcal{G}}_s$ and $(h,t(h^{-1})\phi(p)) \in {\mathcal{G}}_s \rtimes_{\alpha} G \cong {\mathcal{G}}_1$. Under the latter canonical diffeomorphism, see, e.g. \cite[\S 3]{LudewigWaldorf2Group}, 
\begin{equation*}
(h,t(h)^{-1}\phi(p))\mapsto hi(t(h)^{-1}\phi(p))\in {\mathcal{G}}_1\text{.}
\end{equation*}

In the present situation, we get for $p=(\beta_1,\beta_2,\Phi,\gamma)$ and $h=(\gamma',1,X)$
\begin{equation*}
((\gamma',1,X),t(\gamma'^{-1},1,X^{-1})\gamma)\mapsto (\gamma',1,X)\cdot \id_{\gamma'^{-1}\gamma}=(\gamma,\gamma'^{-1}\gamma,X\cdot \id_{\gamma'^{-1}\gamma})\text{.}
\end{equation*}
Letting this act according to \cref{EqPrincipalActionQOld}, we get the claimed expression \cref{AppEqPrincipalActionQ}.

                \end{itemize}
        \item
                On the triple fibre product $Y \times_M Y \times_M Y$, it has the following bundle morphism
                \begin{equation} \label{AppEqBundleGerbeProduct}
                        \mu_{\stringgerbe(M)} : \pr_{23}^{*}P \otimes \pr_{12}^{*}P \to \pr_{13}^* P.
                \end{equation}
                As recalled in  \cref{sec:principalbundles}, elements in the tensor product $\pr_{23}^{*}P \otimes \pr_{12}^{*}P$ are represented by pairs
\begin{equation}
\label{eq:stringtp1}
(\beta_2',\beta_3,\Phi_{23},\gamma_{23})\otimes (\beta_1,\beta_2,\Phi_{12},\gamma_{12}) \in P \times P
\end{equation}
such that $\beta_2'=\beta_2\gamma_{12}^{-1}$; such a pair projects then to $(\beta_1,\beta_2\gamma_{12}^{-1},\beta_3\gamma_{23}^{-1})\in Y^{[3]}$. 
The anchor map sends above element to $\gamma_{23}\gamma_{12}$, and the principal ${\mathcal{G}}_s$-action is the one on the first factor. 
The bundle gerbe product \cref{AppEqBundleGerbeProduct} is then given by (see \cite[Eq.~(7.1.6)]{Nikolausa})
                \begin{equation}
                        \remember{Appeq:bgp:G}
                        {\mu_{\stringgerbe(M)}((\beta_2',\beta_3,\Phi_{23},\gamma_{23}),(\beta_1,\beta_2,\Phi_{12},\gamma_{12})) = (\beta_1,\beta_3\gamma_{12},\lambda(\Phi_{23}\cdot \id_{\gamma_{12}} \otimes \Phi_{12}),\gamma_{23}\gamma_{12})}\text{.}
                \end{equation}

\end{itemize}

In the remainder of this section we will prove that the procedure introduced above to get from Version (2) to Version (4) establishes the equivalence. We consider the following diagram
\begin{equation}
        \label{eq:diagram}
\begin{aligned}
        \xymatrix{\text{Version (1)}  \ar[r]^{T} & \text{Version (2)}\ar[d]^{H} \\ \text{Version (4)} \ar[u]^{L} & \text{Version (3)} \ar[l]^{\ddot{A}}}
\end{aligned}   
\end{equation}
The map $T$ takes a trivialization of the Chern-Simons 2-gerbe and transgresses it to a trivialization of the spin lifting gerbe on $LM$, which in turn can be translated into a loop-spin structure \cite{Waldorfa,Waldorfb}. The map $L$ regards the Chern-Simons 2-gerbe as the lifting 2-gerbe for lifting the structure group of $\Spin(M)$ to $\String(d)$ \cite{Nikolausa}, and regards a  ${\mathcal{G}}$-bundle gerbe as a solution to this lifting problem.  The maps $T$ and $L$ are bijections (on a level of equivalence classes). The map $H$ is  the one constructed above, and the map $\ddot{A}$ is the canonical equivalence between principal ${\mathcal{G}}$-bundles and ${\mathcal{G}}$-bundle gerbes. 

\begin{proposition}
\label{prop:versions}
Diagram \cref{eq:diagram} is commutative.
\end{proposition}

\begin{proof}
        We show that $T^{-1}=L \circ \ddot{A} \circ H$. 
        To this end, we consider a fusive loop-spin structure $\widetilde{L\Spin}(M)$ and show that the two trivializations of the Chern-Simons gerbe obtained by $T^{-1}$ and $L\circ \ddot{A}\circ H$ are equivalent.
        We will use the fact that two trivializations of the Chern-Simons 2-gerbe are equivalent if and only if the corresponding string classes coincide, i.e., the 3-classes of the bundle gerbes over $\Spin(M)$, see \cite[Thm. 5.3.1]{Waldorfb}. 

        Under the map $T^{-1}$, the string class is represented by the regression of the fusive principal $\U(1)$-bundle that underlies the given fusive loop-spin structure, see \cite[Cor. 4.4.8]{Waldorfa}; namely, the bundle $\widetilde{L\Spin}(M) \to L\Spin(M)$ and its fusion product $\lambda$. 
        Thus, the regression (w.r.t. the already fixed point ${\hat x}$) is the following bundle gerbe over $\Spin(M)$:
        \begin{itemize}
                \item 
                        the surjective submersion is the end point evaluation $Y=P_{\hat x}\Spin(M) \to \Spin(M)$.
                \item
                        the principal $\U(1)$-bundle over $Y^{[2]}$ is $\cup^{*}\widetilde{L\Spin}(M)$. 
                \item
                        the bundle gerbe product is $\lambda$.
        \end{itemize}
        By construction, this bundle gerbe represents the string class.

        Now we look at the map $L\circ \ddot{A}\circ H$. 
        According to the description of the map $L$ in \cite{Nikolausa} we have to consider the $\String(d)$-bundle gerbe $\stringgerbe(M)=\ddot{A}(\mathcal{P}) $ associated the principal $\String(d)$-2-bundle $\mathcal{P}$, take its pullback along $\varpi: \Spin(M) \to M$, and then identify $\varpi^{*}\stringgerbe(M)$ with a $\String(d)$-bundle gerbe of the form $i_{*}(\stringgerbe(M))$, where $\mathcal{Q}$ is a $\U(1)$-bundle gerbe over $\Spin(M)$ and $i: \mathscr{B}\!\U(1) \to \String(d)$ is the central inclusion.
        Then, $\mathcal{Q}$ represents the string class. 
        The commutative diagram
        \begin{equation*}
                \xymatrix{P_{\hat x}\Spin(M) \ar[d] \ar@{=}[r] & P_{{\hat x}}\Spin(M) \ar[d] \\ \Spin(M)\ar[r]_-{\varpi} & M}
        \end{equation*} 
        shows that $\varpi^*\mathcal{Q}$ has the surjective submersion $P_{\hat x}\Spin(M) \to \Spin(M)$. On double fibre product (over $\Spin(M)$), we have a commutative diagram
        \begin{equation*}
                \xymatrix{\mathcal{P}_1 \ar[r]^{\varphi} \ar[d]  & P \ar[d] \\ P_{\hat x}\Spin(M) \times_{\Spin(M)} P_{\hat x}\Spin(M) \ar[r] & P_{\hat x}\Spin(M) \times_{M} P_{\hat x}\Spin(M)}
        \end{equation*}
        where the left vertical map is $(\beta_1,\beta_2,\Phi) \mapsto (\beta_1,\beta_2)$, the right vertical map is the bundle projection \cref{EqBundleProjectionQ} and the map $\varphi$  is $(\beta_1,\beta_2,\Phi)\mapsto (\beta_1,\beta_2,\Phi,e)$, with $e$ the constant path at $e\in \Spin(d)$. 
        The map left vertical map is a principal $\U(1)$-bundle, and $\varphi$ is equivariant along $\mathscr{B}\!\U(1) \to \String(d)$, as
        \begin{align*}
                \varphi((\beta_1,\beta_2,\Phi) \cdot z) &=\varphi(\beta_1,\beta_2,\Phi z)\\
                                                        &=(\beta_1,\beta_2,\Phi z, c_e)\\
                                                        &=(\beta_1,\beta_2,\Phi,c_e)\circ(c_e,c_e,z)\\
                                                        &=\varphi(\beta_1,\beta_2,\Phi)\circ (c_e,c_e,z)\text{,}
        \end{align*}
        with the principal action  defined in \cref{EqPrincipalActionQOld}.
        This shows that $\varphi$ is an isomorphism  
        \begin{equation*}
                i_{*}(\mathcal{P}_1) \cong P|_{Y^{[2]}}\text{.}
        \end{equation*}
        Finally, restricting the bundle gerbe product \cref{Appeq:bgp:G} along $\varphi$, we  recover $\lambda$. 
        Summarizing, the $\U(1)$-bundle gerbe $\mathcal{Q}$ we are looking for is precisely the one we got under $T^{-1}$.
\end{proof}

\setsecnumdepth{1}

\section{From rigged bundles to continuous bundles}

\label{sec:riggedtocontinuous}

In this section we recall the notions of rigged Hilbert spaces and rigged von Neumann algebras, and the corresponding notions of locally trivial bundles, as set up in \cite{kristel2020smooth,Kristel2019c}. Then, we explain how to pass from this \emph{rigged} setting to the \emph{continuous} setting considered in \cref{sec:vonNeumannbundles}.

        A \emph{rigged Hilbert space} is a Fr\'{e}chet space $E$ equipped with a continuous (sesquilinear) inner product;
we denote by $\hat{E}$ its Hilbert completion.
        A \emph{rigged \cstar-algebra} is a Fr\'{e}chet algebra $D$, equipped with a continuous norm and a continuous complex anti-linear involutive anti-automorphism, such that its norm completion $\hat D$ is a \cstar-algebra.                 A \emph{rigged $D$-module} is a rigged Hilbert space $E$ together with a representation of $D$ on $E$ whose action map $D \times E \rightarrow E$ is smooth, and the following conditions hold for all $a \in D$, and all $v,w \in E$
        \begin{equation}
        \label{def:RiggedCStarModule}
                \| a\lact \xi\| \leqslant \| a \| \| \xi\| , \quand \langle a \lact \xi,\eta \rangle = \langle \xi, a^* \lact \eta \rangle.
        \end{equation}
The conditions in \cref{def:RiggedCStarModule} guarantee  that the action induces a  $\ast$-homomorphism $\hat D \to \mathcal{B}(\hat E)$, i.e., a representation of the \cstar-algebra $\hat D$ on the Hilbert space $\hat E$, see \cite[Rem. 2.2.11]{kristel2020smooth}.

\begin{definition}
        A \emph{rigged von Neumann algebra} is a pair $(D,E)$ consisting of a rigged \cstar-algebra $D$ and a rigged $D$-module $E$, with the property that the representation $\hat{D} \rightarrow \mc{B}(\hat{E})$ is faithful.
\end{definition}

From a rigged von Neumann algebra $(D,E)$ we obtain an ordinary von Neumann algebra $D'' \subseteq \mc{B}(\hat{E})$, see \cite[Remark 2.1.7]{Kristel2019c}. 
In particular, we consider later the topological group $I(\hat{E}) \subset \B(\hat E)$ defined  in \cref{DefinitionNH}.

\begin{example}
\label{ex:riggedspaces}
Let $A$ be the hyperfinite type III$_1$ factor.
In \cite[\S 3.2]{Kristel2019c} we constructed a Fr\'echet subalgebra $A\smooth \subset A$ and a Fr\'echet subspace $L^2(A)\smooth \subset L^2(A)$ such that $L^2(A)\smooth$ is a rigged $A\smooth$-module, and $(A\smooth,L^2(A)\smooth)$ is a rigged von Neumann algebra, with completions
\begin{equation}
\label{completion-l2A}
\widehat{A\smooth} = A
\quomma
\widehat {L^2(A)\smooth} = L^2(A)
\end{equation} 
We remark that the group homomorphism $\omega: P\Spin(d) \to \Aut(A)$ from \cref{KleinOmega} extends to a group homomorphism $\omega': L\Spin(d) \to \Aut(A)$ along the doubling map $\Delta: P\Spin(d) \to L\Spin(d)$, defined by
\begin{equation*}
  \omega^\prime(\gamma) a = t_{\Aut(A)}(\Omega(\tilde{\gamma}))(a),
\end{equation*}
where $\tilde{\gamma}$ is any lift of $\gamma$ to the central extension and $\Omega$ is the group homomorphism from \cref{GrossOmega}.
The action $\omega$ of $P\Spin(d)$ as well as the extension $\omega'$ of $L\Spin(d)$ on $A$ restrict to  \emph{smooth} actions on $A\smooth$, while the action $\Omega$ of $\widetilde{L\Spin}(d)$ on $N(A)$  from \cref{GrossOmega} restricts to a smooth action on $L^2(A)\smooth$ \cite[Prop. 3.2.2]{Kristel2019c}.
\end{example}

We continue with recalling the notion of rigged \emph{bundles} on the basis of Section 2 of \cite{kristel2020smooth}. 
Let $E$ be a rigged Hilbert space. 
A \emph{rigged Hilbert bundle} over a Fr\'echet manifold $M$ with typical fibre $E$ is a Fr\'echet vector bundle  $\mc{E}$ over $M$  with typical fibre $E$, equipped with fibrewise inner products such that local trivializations can be chosen to be fibrewise isometric. 
A \emph{unitary morphism of rigged Hilbert bundles} is an isomorphism of Fr\'echet vector bundles. 
The fibrewise completion $\hat{\mathcal{E}}$ of $\mathcal{E}$ is a  locally trivial continuous Hilbert bundle over $M$ with typical fibre $\hat E$ \cite[Lem.~2.1.13]{kristel2020smooth}.
Likewise, a unitary morphism of rigged Hilbert bundles extends uniquely to a unitary Hilbert bundle isomorphism. 

\begin{example}
\label{ex:spinorbundleonloopspace}
Let $M$ be a spin manifold, and let $\Spin(M)$ be its spin structure, a $\Spin(d)$-principal bundle over $M$. 
Let further $\widetilde{L\Spin}(M)$ be a spin structure on $LM$, i.e., a lift of the structure group of $L\Spin(M)$ along the basic central extension of $L\Spin(d)$, see \cref{SectionStringorRep}. 
Then, continuing \cref{ex:riggedspaces}, the associated vector bundle
\begin{equation*}
\mathcal{S}\smooth_{LM}:= \widetilde{L\Spin}(M) \times_{\widetilde{L\Spin}(d)} L^2(A)\smooth
\end{equation*} 
is a rigged Hilbert bundle over $LM$ with typical fibre $L^2(A)\smooth$, the \emph{smooth spinor bundle on loop space}; see \cite[Lemma 2.2.2 \& Def. 4.1.4]{Kristel2019c}. Due to \cref{completion-l2A}, the fibrewise completion of $\mathcal{S}_{LM}\smooth$ becomes the Hilbert bundle \begin{equation*}
\widehat{\mathcal{S}_{LM}\smooth} = \widetilde{L\Spin}(M) \times_{\widetilde{L\Spin}(d)} L^2(A)\text{.}
\end{equation*}
In order to compare this with the continuous spinor bundle $\mathcal{S}_{LM}$  from \cref{DefinitionSpinorBundle} we have to consider the difference between the representations $\Omega$ (used for $\mathcal{S}_{LM}\smooth$) and $\Omega'$ (used for $\mathcal{S}_{LM}$). Let $s: LM \to LM$ be the map induced by complex conjugation ($t \mapsto 2\pi -t$) on $S^1$. We claim that $s$ lifts to an involution $\tilde s$ of $\widetilde{L\Spin}(M)$ in such a way that $\tilde s(\Phi \cdot X) = \tilde s(\Phi) \cdot \tilde\sigma(X)$, where $\tilde\sigma$ is a similar lift of complex conjugation to $\widetilde{L\Spin}(d)$. The lifts $\tilde s$ and $\tilde\sigma$ can be induced from the fusion products $\lambda$ and $\mu$, respectively, see \cref{SubsectionStringStructuresFusiveLoopSpin}. The representation $\Omega$ is compatible with the lift $\tilde\sigma$ in the sense that $\Omega'(X)=J\Omega(X)J = \Omega(\tilde\sigma(X))$; see \cite[Prop. 4.9 \& 4.11]{Kristel2019}. Using this, one can check that $[\Phi,\xi] \mapsto [\tilde s(\Phi),\xi]$ establishes an isomorphism \begin{equation*}
\widehat{\mathcal{S}\smooth_{LM}} = s^{*} \mathcal{S}_{LM}
\end{equation*}
of continuous Hilbert bundles over $LM$.
\end{example}

Similarly, we define rigged \cstar-algebra bundles. If $D$ is a rigged \cstar-algebra, then a \emph{rigged \cstar-algebra bundle} over $M$ with typical fibre $D$ is a Fr\'echet vector bundle $\mc{D}$ where each fibre is equipped with a norm and the structure of a $\ast$-algebra, such that local trivializations can be chosen to be fibrewise isometric $\ast$-isomorphisms. 
 A \emph{morphism of rigged \cstar-algebra bundles} over $M$ is a morphism of Fr\'echet vector bundles that is fibrewise a morphism of $\ast$-algebras and locally bounded with respect to the norms.
The fibrewise norm completion gives a locally trivial continuous bundle of \cstar-algebras with typical fibre $\hat D$, and strongly continuous transition functions \cite[Lem.~2.2.6]{kristel2020smooth}. Likewise, any morphism of rigged \cstar-algebra bundles extends uniquely to a  morphism of continuous bundles of \cstar-algebras.

Let $D$ be a rigged \cstar-algebra and  $E$ be a rigged $D$-module, and let $\mathcal{D}$ be a rigged \cstar-algebra bundle over $M$ with typical fibre $D$. 
A \emph{rigged $\mathcal{D}$-module bundle with typical fibre $E$} is a rigged Hilbert bundle $\mathcal{E}$ with typical fibre $E$, together with, for each $x \in X$, the structure of a rigged $\mathcal{D}_x$-module on $\mathcal{E}_x$, such that around every point in $M$ there exist local trivializations $\phi$ of $\mc{D}$ and $u$ of $\mc{E}$  that fibrewise intertwine the actions, i.e., we have $u_x(a \lact \xi)=\phi_x(a)\lact u_x(\xi)$ for all $x\in M$ over which $\phi$ and $u$ are defined, and all $a \in \mathcal{D}_x$ and $v\in \mathcal{E}_x$.
A pair $(\phi,u)$ of local trivializations with this property is called \emph{local module trivialization}. 
A \emph{unitary intertwiner} between a rigged $\mathcal{D}_1$-module bundle $\mathcal{E}_1$ and a rigged $\mathcal{D}_2$-module bundle $\mathcal{E}_2$ is a pair $(\varphi,U)$ consisting of a morphism $\varphi:\mathcal{D}_1 \to \mathcal{D}_2$ of rigged \cstar-algebra bundles and a unitary morphism $U$ of rigged Hilbert bundles, such that $\varphi_x$ is an intertwiner along $U_x$ for each $x\in M$ \cite[Def. 2.2.6]{Kristel2019c}. 

\begin{definition}\cite[Definition 2.9.9]{Kristel2019c}
Let $(D,E)$ be a rigged von Neumann algebra.
 A \emph{rigged von Neumann algebra bundle} over $M$ with typical fibre $(D,E)$ is a pair $(\mc{D},\mc{E})$, where $\mathcal{D}$ is a rigged \cstar-algebra bundle over $M$ and $\mathcal{E}$ is a rigged $\mathcal{D}$-module bundle with typical fibre $E$.  
\end{definition}

There is also a corresponding notion of morphisms between rigged von Neumann algebra bundles, called \emph{spatial morphisms}, which is just a unitary intertwiner between the rigged module bundles.

\begin{example}
\label{ex:riggedcliffordalgebrabundle}
We consider the smooth action $\omega': L\Spin(d) \times A\smooth \to A\smooth$ recalled in \cref{ex:riggedspaces}. Within the theory of rigged  bundles, we form 
 the associated rigged \cstar-algebra bundle
\begin{equation}
\label{rigged-algebra-bundle-before-pullback}
\mathcal{D}:=L\Spin(M) \times_{L\Spin(d)} A\smooth
\end{equation}
 over $LM$ with typical fibre $A\smooth$ \cite[\S 5.1]{Kristel2019c}.
Next we consider the smooth action $\Omega$ of $\widetilde{L\Spin}(d)$ on $L^2(A)\smooth$ recalled   in \cref{ex:riggedspaces} and form the associated rigged Hilbert bundle
\begin{equation}
\label{rigged-module-bundle-before-pullback}
\mathcal{E} := \widetilde {L\Spin}(M) \times_{\widetilde{L\Spin}(d)} L^2(A)\smooth.
\end{equation}
It is shown in \cite[Prop. 5.1.2]{Kristel2019c} that $\mathcal{E}$ is a rigged module bundle over $\mathcal{D}$, and that the pair $\mathcal{A}\smooth_{LM}:=(\mathcal{D},\mathcal{E})$ is a rigged von Neumann algebra bundle over $LM$. 
\end{example}

Let $(\mc{D},\mc{E})$ be a rigged von Neumann algebra bundle with typical fibre $(D,E)$ over $M$.
In each fibre over $x \in M$, we obtain a rigged von Neumann algebra $(\mc{D}_{x},\mc{E}_{x})$, which can thus be completed to a honest von Neumann algebra 
\begin{equation}
\label{eq:fibrewisefibres}
\mc{D}_{x}'' \subseteq \B(\hat{\mc{E}}_{x})\text{.}
\end{equation}
The collection $\mathcal{D}'':=(\mathcal{D}_x'')_{x\in X}$ of von Neumann algebras can be combined to a continuous bundle of von Neumann algebras, as follows.
Consider an open subset $O \subset M$ supporting compatible local trivializations $\phi$ of $\mc{D}|_{O}$ and $U$ of $\mc{E}|_{O}$ (see \cite[Lemma 2.1.10]{Kristel2019c}).

\begin{lemma}
Suppose $(\phi_1,U_1)$ and $(\phi_2, U_2)$ are compatible local trivializations of $\mathcal{D}$.
Then, the corresponding local trivializations $\phi''_1$ and $\phi''_2$ of $\mathcal{D}^{\prime\prime}$, obtained by extending $\phi_1$ and $\phi_2$ with respect to the ultraweak topology, are compatible in the sense that 
\begin{equation*}
  \phi^{\prime\prime}_1 \circ (\phi^{\prime\prime}_2)^{-1} :  O_1 \cap O_2 \to \Aut(D^{\prime\prime})
\end{equation*}
is continuous with respect to the u-topology, where $O_i$ is the domain of definition of $(\phi_i, U_i)$.
\end{lemma}

\begin{proof}
Denote by $O_1, O_2 \subseteq X$ the open subsets on which the trivializations of $\mathcal{D}$ are defined.
As local trivializations of rigged Hilbert bundles extend uniquely to continuous trivializations of ordinary Hilbert bundles, the map
\begin{equation*}
\hat{U} := \hat{U}_1 \circ (\hat{U_2})^* : O_1 \cap O_2 \to \U(\hat{E})
\end{equation*}
is continuous.
we have for all $x \in O_1 \cap O_2$ that
\begin{equation*}
\hat{U}_x (a \lact \hat{U}_x^* \xi) = (\phi^{\prime\prime}_1 \circ (\phi^{\prime\prime}_2)^{-1})(a) \xi,  \qquad a \in D^{\prime\prime}, ~~\xi \in \hat{E}.
\end{equation*}
Hence we can factorize
\begin{equation*}
  \phi^{\prime\prime}_1 \circ (\phi^{\prime\prime}_2)^{-1} : O_1 \cap O_2 \to \U^\prime(\hat{E}) \stackrel{t}{\to} \Aut(D^{\prime\prime}),
\end{equation*}
where $\U^\prime(\hat{E}) \subset \U(\hat{E})$ is the subgroup of unitaries that preserve $D^{\prime\prime} \subseteq \B(\hat{E})$ upon conjugation and the second map sends a unitary to the automorphism it induces by conjugation.
The first map is continuous as seen above, the second map is continuous by \cite[Remark~B.11]{ludewig2023spinor}.
Hence $\phi^{\prime\prime}_1 \circ (\phi^{\prime\prime}_2)^{-1}$ is continuous.
\end{proof}

\begin{definition}
\label{def:riggedtocontinuous}
Let $(\mc{D},\mc{E})$ be a rigged von Neumann algebra bundle with typical fibre $(D,E)$ over  a Fr\'echet manifold $M$, where $\hat E$ is a standard form of $D''$. 
The \emph{associated continuous von Neumann algebra bundle} is the collection $\mathcal{D}''=(\mathcal{D}_x'')_{x\in X}$ together with the local trivializations $\phi''$ induced from all local module trivializations $(\phi,U)$ of $\mathcal{D}$.
\end{definition}

The following result assures that associated continuous von Neumann algebra bundles are compatible with morphisms of rigged von Neumann algebra bundles.

\begin{proposition}
\label{prop:fromriggedtocontinuous}
\cref{def:riggedtocontinuous} establishes a functor between the category of rigged von Neumann algebra bundles with spatial morphisms  to the category of continuous von Neumann algebra bundles.
\end{proposition}

\begin{proof}
A spatial morphism $(\phi,U): (\mathcal{D}_1, \mathcal{E}_1) \to (\mathcal{D}_2, \mathcal{E}_2)$ between rigged von Neumann algebra bundles extends fibrewise (via conjugation by $U_x$ or, equivalently, ultraweak continuity of $\phi_x$) to normal $*$-homomorphisms $\phi_x: (\mathcal{D}_1'')_x \to (\mathcal{D}_2'')_x$, and these send local trivializations to local trivializations. 
It is clear that all constructions are compatible with pullbacks.  
\end{proof}

\begin{example}
\label{ex:cliffordalgebrabundle}
Applying \cref{def:riggedtocontinuous} to the rigged von Neumann algebra bundle $\mathcal{A}\smooth_{LM}:=(\mathcal{D},\mathcal{E})$ of \cref{ex:riggedcliffordalgebrabundle}, we obtain a continuous von Neumann algebra bundle
$(\mathcal{A}_{LM}\smooth)''$ over $LM$ with typical fibre $A$. In fact, we have
\begin{equation*}
(\mathcal{A}\smooth_{LM})'' = L\Spin(M) \times_{L\Spin(d)} A\text{,}
\end{equation*} 
where the associated bundle is 
formed using the continuous representation $\omega'$ of $L\Spin(d)$ on $A$.
\end{example}

Finally, we have to consider rigged bimodule bundles.
Let  $(\mathcal{D}_{1},\mathcal{E}_1)$ and $(\mathcal{D}_{2},\mathcal{E}_2)$ be rigged von Neumann algebra bundles over a Fr\'echet manifold $M$ with typical fibres $(D_{1},E_1)$ and $(D_{2},E_2)$, respectively, and let $E$ be a rigged $D_1$-$D_2$-bimodule.  
A \emph{rigged $\mathcal{D}_1$-$\mathcal{D}_2$-bimodule bundle $\mathcal{E}$ with typical fibre $E$} is a rigged Hilbert bundle $\mathcal{E}$ over $M$ 
that is both a rigged $\mathcal{D}_{1}$-module bundle and a rigged $\mathcal{D}_{2}^{\opp}$-module bundle such that the actions commute, and such that around every point in $M$ there exist local trivializations  $\phi_1$ of $\mc{D}_{1}$, $U_1$ of $\mathcal{E}_1$,  $\phi_2$ of $\mathcal{D}_{2}$, $U_2$ of $\mathcal{E}_2$, and $V$ of $\mc{E}$ with  $(\phi_1,U_1)$, $(\phi_2,U_2)$,  $(\phi_1,V)$ and $(\phi_2,V)$ all compatible at the same time.    
\cite[Lemma 2.2.14]{Kristel2019c} shows that that each fibre $\mathcal{E}_x$ is a rigged $(\mathcal{D}_1)_x$-$(\mathcal{D}_2)_x$-bimodule, and \cite[Lem. 2.1.16]{Kristel2019c} shows then that the completion $\hat{\mathcal{E}_x}$ is a $(\mathcal{D}_1)_x''$-$(\mathcal{D}_2)_x''$-bimodule.

\begin{lemma}
        \label{prop:riggedtocontinuousbimodulebundle}
        If $(\mc{D}_1, \mathcal{E}_1)$ and $(\mc{D}_2, \mathcal{E}_2)$ are rigged von Neumann algebra bundles, and $\mc{E}$ is a rigged  $\mc{D}_1$-$\mc{D}_2$-bimodule bundle, then 
        the  $(\mathcal{D}_1)_x''$-$(\mathcal{D}_2)_x''$-bimodule structure on the fibres $\hat{\mathcal{E}}_x$ turn the Hilbert bundle $\hat{\mathcal{E}}$ into a $\mathcal{D}_1''$-$\mathcal{D}_2''$-bimodule bundle.
\end{lemma}

\begin{proof}
We consider an open set $O \subset M$ that supports local trivializations 
\begin{align*}
\phi_1:\mc{D}_{1}|_O \to O \times D_{1} && U_1: \mathcal{E}_1|_O \to O \times E_1 
\\
\phi_2: \mathcal{D}_{2}|_O \to O \times D_{2} && U_2:\mathcal{E}_2|_O \to O \times E_{2}
\end{align*}
and
\begin{equation*}
V:\mc{E}|_O \to O \times E
\end{equation*}
The compatibility conditions imply that for each $x\in M$ we have        
\begin{align}\label{eq:IntertwinersForBimodules}
V_x( a_1 \lact \xi \ract a_2) &= (\phi_1)_x(a_1) \lact V_x(\xi) \ract (\phi_2)_x(a_2), \qquad a_i \in \mathcal{D}_i, ~~ \xi \in \mathcal{E}_x.
\end{align}
The further compatibility conditions imply, as discussed above, that $(\phi_i)_x$ extend to local trivializations $(\phi_i'')_x: (\mathcal{D}_i'')_x \to D_i''$ of von Neumann algebra bundles (obtained by conjugation with the completions $(\hat{U}_i)_x: (\hat{\mathcal{E}}_i)_x \to \hat E_i$). 
Thus, \cref{eq:IntertwinersForBimodules} extends by continuity to the completions, and becomes
\begin{align}
\hat{V}_x( a_1 \lact \xi \ract a_2) &= (\phi_1'')_x(a_1) \lact \hat{V}_x(v) \ract (\phi_2'')_x(a_2), \qquad a_i \in (\mathcal{D}_i^{\prime\prime})_x, ~~ \xi \in \hat{\mathcal{E}}_x.
\end{align}
This shows that $\hat{V}_x$ is an intertwiner along $\phi_1''$ and $\phi_2''$. 
\end{proof}


\end{appendix}

\bibliography{bibfile}
\bibliographystyle{kobib}

\end{document}